\documentclass[11pt,letterpaper,twoside]{article}

  \usepackage{eucal}

  \newenvironment{rmenumerate}{\begin{enumerate}}{\end{enumerate}}

  \usepackage{caption}
  \usepackage[margin=3cm, marginparwidth=2.6cm]{geometry}

  \usepackage{amsmath,amssymb,amsthm,amsfonts,bbding,stmaryrd,dsfont,wasysym,pifont,xspace,xcolor}
  \usepackage{parskip,fancyhdr,url}

  \usepackage{stackrel}

  \usepackage{epstopdf,yfonts}
  \usepackage{graphicx,enumerate,epsfig}
  \usepackage{mathrsfs}
  \usepackage{enumitem}
  \usepackage[normalem]{ulem}

  \definecolor{verydarkblue}{rgb}{0,0,0.5}

  \usepackage[
  breaklinks,colorlinks,
  citecolor=verydarkblue,linkcolor=verydarkblue,urlcolor=verydarkblue,
  pdftitle={Coarse and fine geometry of the Thurston metric},
  pdfauthor={David Dumas, Anna Lenzhen, Kasra Rafi, and Jing Tao}
  ]{hyperref}

  \usepackage{tikz}
  \usetikzlibrary{arrows,calc,decorations.pathmorphing,backgrounds,positioning,fit}

  \begingroup
  \makeatletter
  \@for\theoremstyle:=definition,remark,plain\do{%
    \expandafter\g@addto@macro\csname th@\theoremstyle\endcsname{%
      \addtolength\thm@preskip\parskip
    }%
  }
  \endgroup

  \newcommand\address[1]{}
  \newcommand\email[1]{}
  \newcommand\dedicatory[1]{}

  \newtheorem{theorem}{Theorem}[section]
  \newtheorem{proposition}[theorem]{Proposition}
  \newtheorem{corollary}[theorem]{Corollary}
  \newtheorem{lemma}[theorem]{Lemma}

  \theoremstyle{definition}
  
  \newtheorem{remark}[theorem]{Remark}
  
  \newtheorem*{claim*}{Claim}

  \newtheorem*{question*}{Question}
  \newtheorem*{answer*}{Answer}
  \newtheorem*{application*}{Application}

  \newcommand{\secref}[1]{Section~\ref{Sec:#1}}
  \newcommand{\subsecref}[1]{Section~\ref{Subsec:#1}}
  \newcommand{\thmref}[1]{Theorem~\ref{Thm:#1}}
  \newcommand{\corref}[1]{Corollary~\ref{Cor:#1}}
  \newcommand{\lemref}[1]{Lemma~\ref{Lem:#1}}
  \newcommand{\propref}[1]{Proposition~\ref{Prop:#1}}
  
  \newcommand{\remref}[1]{Remark~\ref{Rem:#1}}

  \newcommand{\figref}[1]{Figure~\ref{Fig:#1}}

  \newcommand{\calF}{\mathcal{F}}
  \newcommand{\calG}{\mathcal{G}}

  \newcommand{\calT}{\mathcal{T}}

  \DeclareMathOperator{\csch}{csch}
  \DeclareMathOperator{\SL}{SL}
  \DeclareMathOperator{\PSL}{PSL}
  \DeclareMathOperator{\GL}{GL}
  \DeclareMathOperator{\PGL}{PGL}
  \DeclareMathOperator{\Mod}{Mod}
  \DeclareMathOperator{\Out}{Out}
  \DeclareMathOperator{\Env}{Env}
  \DeclareMathOperator{\In}{In}

  \DeclareMathOperator{\str}{stretch}
  \DeclareMathOperator{\I}{i}
  \DeclareMathOperator{\twist}{twist}
  \DeclareMathOperator{\arcsinh}{arcsinh}
  \DeclareMathOperator{\arccosh}{arccosh}
  \DeclareMathOperator{\Log}{Log}
  \DeclareMathOperator{\EQ}{EQ}
  
  \DeclareMathOperator{\pivot}{Pivot}

  \newcommand{\emul}{\mathrel{\ooalign{\raisebox{1.4\height}{$\mkern0.2mu\scriptstyle\ast$}\cr\hidewidth$\asymp$\hidewidth\cr}}}
  \newcommand{\eadd}{\mathrel{\ooalign{\raisebox{1.6\height}{$\mkern0.2mu\scriptscriptstyle+$}\cr\hidewidth$\asymp$\hidewidth\cr}}}
  \newcommand{\emuladd}{\mathrel{\ooalign{\raisebox{1.4\height}{$\mkern1.2mu\scriptstyle\ast$}\cr\hidewidth$\asymp$\hidewidth\cr\raisebox{-1\height}{$\mkern0.2mu\scriptscriptstyle+$}\cr\hidewidth}}}

  \newcommand{\gmul}{\mathrel{\ooalign{\raisebox{1.4\height}{$\mkern0.2mu\scriptstyle\ast$}\cr\hidewidth$\succ$\hidewidth\cr}}}
  \newcommand{\gadd}{\mathrel{\ooalign{\raisebox{1.6\height}{$\mkern0.2mu\scriptscriptstyle+$}\cr\hidewidth$\succ$\hidewidth\cr}}}

  \newcommand{\lmul}{\mathrel{\ooalign{\raisebox{1.4\height}{$\mkern0.2mu\scriptstyle\ast$}\cr\hidewidth$\prec$\hidewidth\cr}}}
  \newcommand{\ladd}{\mathrel{\ooalign{\raisebox{1.6\height}{$\mkern0.2mu\scriptscriptstyle+$}\cr\hidewidth$\prec$\hidewidth\cr}}}

  \renewcommand{\leq}{\leqslant}
  \renewcommand{\geq}{\geqslant}
  \renewcommand{\le}{\leqslant}
  \renewcommand{\ge}{\geqslant}

  \newcommand{\homeo}{\mathrel{\cong}} 
  
  \newcommand{\isom}{\cong} 

  \newcommand{\ddtzero}[1]{\left . \frac{d \:}{dt} {#1} \right |_{t=0}}

  \newcommand{\HH}{\ensuremath{\mathbf{H}}\xspace}
  \newcommand{\QQ}{\ensuremath{\mathbf{Q}}\xspace}
  \newcommand{\RR}{\ensuremath{\mathbf{R}}\xspace}
  \newcommand{\QP}{\ensuremath{\mathbf{QP}}}
  \newcommand{\RP}{\ensuremath{\mathbf{RP}}}

  \newcommand{\ZZ}{\ensuremath{\mathbf{Z}}\xspace}

  \newcommand{\s}{\ensuremath{S}\xspace}
  \newcommand{\torus}{\ensuremath{S_{1,1}}\xspace}
  \newcommand{\T}{\ensuremath{\calT}\xspace}

  \newcommand{\ML}{\ensuremath{\mathcal{ML}}\xspace}
  \newcommand{\CL}{\ensuremath{\mathcal{CL}}\xspace}
  \newcommand{\RL}{\ensuremath{\mathcal{RL}}\xspace}
  \newcommand{\Lam}{\ensuremath{\mathcal{GL}}\xspace}

  \newcommand{\PML}{\ensuremath{\mathcal{PML}}\xspace}
  \newcommand{\curves}{\ensuremath{\mathcal{S}}\xspace}

  \newcommand{\mg}{\ensuremath{\mathcal{MG}}\xspace}
  \newcommand{\farey}{\ensuremath{\mathcal{F}}\xspace}

  \newcommand{\set}[1]{\ensuremath{\left\{ {#1} \right\}}\xspace}
  \newcommand{\abs}[1]{\ensuremath{\left\vert {#1} \right\vert}\xspace}
  \newcommand{\st}{\ensuremath{\:\colon\:}\xspace}
  \newcommand{\Set}[1]{\ensuremath{\Big\{ {#1} \Big\}}\xspace}
  \newcommand{\tnorm}[1]{\ensuremath{\left\| {#1} \right \|_{\mathrm{Th}}}\xspace}

  \newcommand{\dlog}{\ensuremath{d \kern -0.08em \log}\xspace}
  \newcommand{\dlogC}{\ensuremath{d \kern -0.08em \log \kern -0.05em \mathcal{C}}\xspace}
  \newcommand{\dlogPML}{\ensuremath{d \kern -0.08em \log \kern -0.05em \mathcal{PML}}\xspace}

  \newcommand{\Teich}{{Teichm\"uller }}

  \newcommand{\QF}{\ensuremath{\mathcal{QF}}\xspace}

  \newcommand{\ep}{\ensuremath{\epsilon}\xspace}

  \newcommand{\G}{\ensuremath{\calG}\xspace}
  \newcommand{\dth}{\ensuremath{d_{\text{Th}}}\xspace}
  \newcommand{\dths}{\ensuremath{\overline{d}_{\text{Th}}}\xspace}

  \newcommand{\bE}{{\overline{E}}}

  \newcommand{\tl}{{\widetilde{l}}}
  \newcommand{\talpha}{{\widetilde\alpha}}
  \newcommand{\tbeta}{{\widetilde\beta}}
  \newcommand{\tdelta}{{\widetilde\delta}}

  \newcommand{\tomega}{{\widetilde\omega}}

  \newcommand{\tf}{{\widetilde f}}
  
  \newcommand{\tq}{{\widetilde q}}
  \newcommand{\tw}{{\widetilde w}}
  \newcommand{\tx}{{\widetilde X}}

  \newcommand{\param}{{\mathchoice{\mkern1mu\mbox{\raise2.2pt\hbox{$
            \centerdot$}}
        \mkern1mu}{\mkern1mu\mbox{\raise2.2pt\hbox{$\centerdot$}}\mkern1mu}{
        \mkern1.5mu\centerdot\mkern1.5mu}{\mkern1.5mu\centerdot\mkern1.5mu}}}

\newcommand\blfootnote[1]{%
    \begingroup
    \renewcommand\thefootnote{}\footnote{#1}%
    \addtocounter{footnote}{-1}%
    \endgroup
  }

  \newcommand{\figzero}{Figure~\hyperlink{page.1}{0}\xspace}

  \newcommand{\thin}{\mathrm{thin}}
  \newcommand{\thick}{\mathrm{thick}}

  \newcommand{\from}{\colon\,}

  \fancypagestyle{plain}

\begin{document}

  \title    {Coarse and fine geometry of the Thurston metric}
  \date     {}
  \author   {David Dumas, Anna Lenzhen, Kasra Rafi, and Jing Tao}

  \maketitle
  \blfootnote{\emph{Date:} This version: December 30, 2019.  First
    version: October 24, 2016.}
  \thispagestyle{empty}

  \vspace{-5mm}

  \begin{center}
  \includegraphics[width=0.5\textwidth]{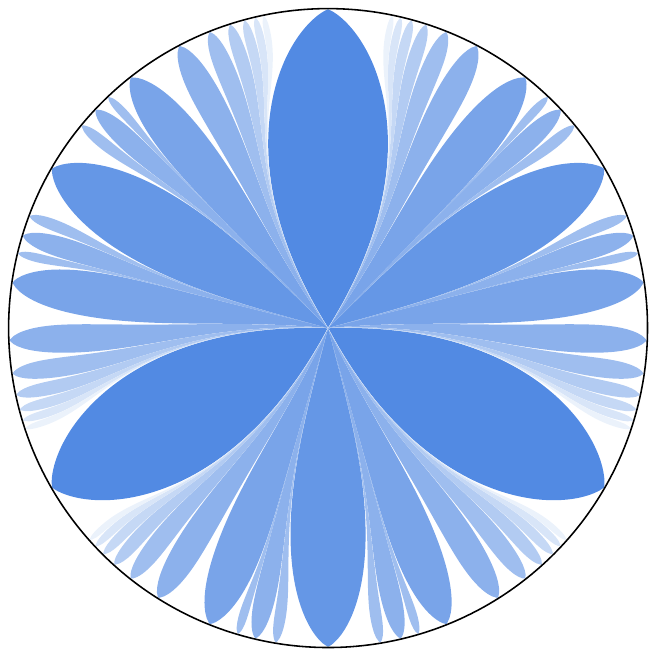}\\
  { \textbf{Figure 0:} \textit{\small In-envelopes in Teichm\"uller
      space; see \remref{vis}.}}
  \end{center}

\section{Introduction}

  In this paper we study the geometry of the Thurston metric on the
  \Teich space $\T(S)$ of hyperbolic structures on a surface $S$.  Some
  of our results on the coarse geometry of this metric apply to
  arbitrary surfaces $S$ of finite type; however, we focus particular
  attention on the case where the surface is a once-punctured torus,
  $S = \torus$.  In that case, our results provide a detailed picture of
  the infinitesimal, local, and global behavior of the geodesics of the
  Thurston metric on $\T(\torus)$, as well as an analogue of Royden's
  theorem (cf. \cite{royden:AT}).

  \subsection*{Thurston's metric}

  Recall that Thurston's metric $\dth : \T(S) \times \T(S) \to \RR$ is
  defined by
  \begin{equation} \label{sup} \dth(X, Y) = \sup_\alpha \, \log \left(
  \frac{\ell_\alpha(Y)}{\ell_\alpha(X)} \right)
  \end{equation}
  where the supremum is over all simple closed curves $\alpha$ in
  $S$ and $\ell_\alpha(X)$ denotes the hyperbolic length of the
  curve $\alpha$ in $X$.  This function defines a forward-complete
  asymmetric Finsler metric, introduced by Thurston in
  \cite{thurston:MSM}.  In the same paper, Thurston introduced two key
  tools for understanding this metric which will be essential in what
  follows: \emph{stretch paths} and \emph{maximally-stretched
    laminations}.

  The maximally stretched lamination $\Lambda(X,Y)$ is a
  chain-recurrent geodesic lamination which is defined for any pair of
  distinct points $X,Y \in \T(S)$.  Typically $\Lambda(X,Y)$ is just a
  simple curve, in which case that curve uniquely realizes the
  supremum defining $\dth$.  In general $\Lambda(X,Y)$ can be a more
  complicated lamination that is constructed from limits of sequences
  of curves that asymptotically realize the supremum.  The precise
  definition is given in \subsecref{thurston} (or
  \cite[Section~8]{thurston:MSM}, where the lamination is denoted
  $\mu(X,Y)$).

  Stretch paths are geodesics constructed from certain
  decompositions of the surface into ideal triangles.  More precisely,
  given a hyperbolic structure $X \in \T(S)$ and a complete geodesic
  lamination $\lambda$ one obtains a parameterized stretch path,
  $\str(X,\lambda,\param) : \RR \to \T(S)$, with
  $\str(X,\lambda,0) = X$ and which satisfies
  \begin{equation}
    \label{eqn:forward-geodesic}
  \dth(\str(X,\lambda,s),\str(X,\lambda,t)) = t-s
  \end{equation}
  for all $s,t \in \RR$ with $s<t$.
  
  Thurston showed that there also exist geodesics in $\T(S)$ that are
  concatenations of segments of stretch paths along different geodesic
  laminations.  The abundance of such ``chains'' of stretch paths is
  sufficient to show that $\dth$ is a geodesic metric space, and also
  that it is not uniquely geodesic---some pairs of points are joined
  by more than one geodesic segment.

  \subsection*{Envelopes}

  The first problem we consider is to quantify the failure of
  uniqueness for geodesic segments with given start and end points.
  For this purpose we consider the set $E(X,Y) \subset \T(S)$ that is
  the union of all geodesics from $X$ to $Y$.  We call this the
  \emph{envelope} (from $X$ to $Y$).

  Based on Thurston's construction of geodesics from chains of
  stretch paths, it is natural to expect that the envelope would admit a
  description in terms of the maximally-stretched lamination
  $\Lambda(X,Y)$ and its completions.  We focus on the punctured torus
  case, because here the set of completions is always finite.

  In fact, a chain-recurrent lamination on $\torus$ (such as
  $\Lambda(X,Y)$, for any $X \neq Y \in \T(\torus)$) is either
  \begin{enumerate}
  \item[(a)] A simple closed curve,
  \item[(b)] The union of a simple closed curve and a spiral geodesic, or
  \item[(c)] A measured lamination with no closed leaves
  \end{enumerate}
  These possibilities are depicted in Figure \ref{fig:torus}. See
  \cite{BZ04} for more details.

  \begin{figure}
    \begin{center}
      \includegraphics{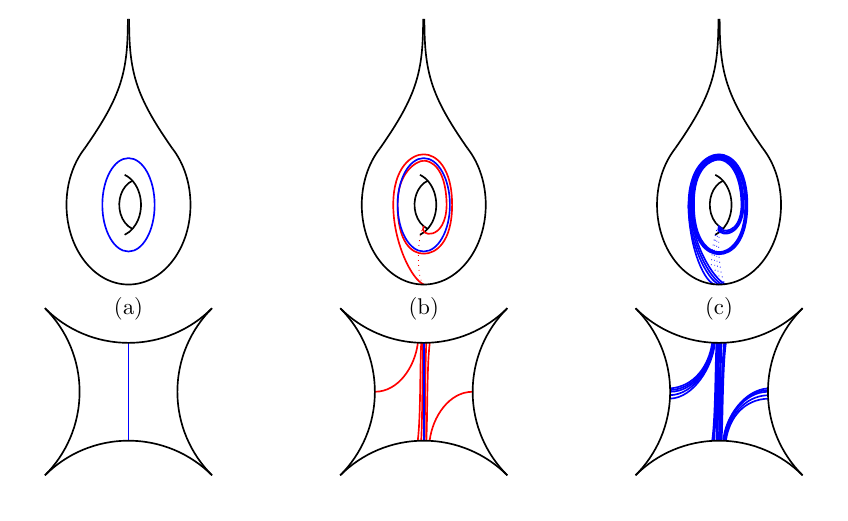}
    \end{center}
    \caption{The three types of chain-recurrent laminations on $\torus$.}
    \label{fig:torus}
    \end{figure}

  We show that the geodesic from $X$ to $Y$ is unique when $\Lambda(X,Y)$
  is of type (b) or (c), and when it has type (a) the envelope has a
  simple, explicit description.  More precisely, we have:

  \begin{theorem}[Structure of envelopes for the punctured torus]
  \label{thm:main-envelopes}
  \mbox{}
  \begin{rmenumerate}
  \item For any $X,Y \in \T(\torus)$, the envelope $E(X,Y)$ is a compact set.
  \item $E(X,Y)$ varies continuously in the Hausdorff topology as a
  function of $X$ and $Y$.
  \item If $\Lambda(X,Y)$ is not a simple closed curve, then $E(X,Y)$
  is a segment on a stretch path (which is then the unique geodesic
      from $X$ to $Y$). 
  \item If $\Lambda(X,Y) = \alpha$ is a simple closed curve, then
  $E(X,Y)$ is a geodesic quadrilateral with $X$ and $Y$ as opposite
  vertices.  Each edge of the quadrilateral is a stretch path along a
  completion of a chain-recurrent geodesic lamination properly containing
  $\alpha$.
  \end{rmenumerate}
  \end{theorem}

  In the course of proving the theorem above, we write explicit
  equations for the edges of the quadrilateral-type envelopes in terms
  of Fenchel-Nielsen coordinates (see
  \eqref{Eq:Twist+}--\eqref{Eq:Twist-}).  Also note that in part (iv)
  of the theorem, a chain-recurrent lamination properly containing
  $\alpha$ has multiple completions, but they all give the same
  stretch path (see \corref{Geodesics}).

  This theorem also highlights a distinction between two cases in
  which the $\dth$-geodesic from $X$ to $Y$ is unique---the cases (b)
  and (c) discussed above.  In case (b) the geodesic to $Y$ is unique
  but some initial segment of it can be chained with another stretch
  path and remain geodesic: The boundary of a quadrilateral-type
  envelope from $X$ with maximally-stretched lamination $\alpha$ furnishes
  an example of this.  In case (c), however, a geodesic that starts
  along the stretch path from $X$ to $Y$ is entirely contained in that
  stretch path (see Proposition \ref{Prop:PropofEnv}).

  \figzero can also be seen as an illustration of this theorem: It shows
  regions in $\T(\torus)$ bounded by pairs of stretch rays from rational points
  on the circle at  infinity to the hexagonal punctured torus. Such
  \emph{in-envelopes} are limiting cases of the envelopes of type (iv) where $X$
  is replaced by a lamination.  These are defined precisely and
  studied in \secref{envelopes}. \figzero is discussed in more detail in
  \remref{vis}.

  \subsection*{Short curves}

  Returning to the case of an arbitrary surface $S$ of finite type, in
  section \ref{sec:twisting} we establish results on the coarse
  geometry of Thurston metric geodesic segments.  This study is
  similar in spirit to the one of \Teich geodesics in \cite{rafi:SC},
  in that we seek to determine whether or not a simple curve $\alpha$
  becomes short along a geodesic from $X$ to $Y$.  As in that case, a
  key quantity to consider is the amount of twisting along $\alpha$
  from $X$ to $Y$, denoted $d_\alpha(X,Y)$ and defined in
  \subsecref{twisting}.

  For curves that \emph{interact} with the maximally-stretched lamination
  $\Lambda(X,Y)$, meaning they belong to the lamination or intersect it
  essentially, we show that becoming short on a geodesic with endpoints
  in the thick part of $\T(S)$ is equivalent to the presence of large twisting:

  \begin{theorem}
  \label{thm:main-length-twisting}
    There exists a constant $\ep_0$ such that the following statement
    holds. Let $X, Y$ lie in the $\ep_0$--thick part of $\T(S)$ and let
    $\alpha$ be a simple curve on $S$ that interacts with $\Lambda(X,Y)$.
    Then the minimum length $\ell_\alpha$ of $\alpha$ along any Thurston
    metric geodesic from $X$ to $Y$ satisfies
    \[
    \frac{1}{\ell_\alpha} \Log \frac{1}{\ell_\alpha} \emuladd
    d_\alpha(X,Y)
    \]
    with implicit constants that depend only $\ep_0$, and where $\Log(x) =
    \max(1, \log(x))$.
  \end{theorem}

  Here $\emuladd$ means equality up to an additive and multiplicative
  constant; see \subsecref{approx-comparisons}.  The theorem
  above and additional results concerning length functions along
  geodesic segments are combined in \thmref{Short}.

  In Section \ref{sec:coarse} we specialize once again to the
  Teichm\"uller space of the punctured torus in order to say more
  about the coarse geometry of Thurston geodesics.  Here every simple
  curve interacts with every lamination, so Theorem
  \ref{thm:main-length-twisting} is a complete characterization of
  short curves in this case.  Furthermore, in this case we can
  determine the order in which the curves become short.

  To state the result, we recall that the pair of points
  $X,Y \in \T(\torus)$ determine a geodesic in the dual tree of the
  Farey tesselation of $\HH^2 \simeq \T(\torus)$.  Furthermore, this
  path distinguishes an ordered sequence of simple
  curves---the \emph{pivots}---and each pivot has an associated
  \emph{coefficient}.  These notions are discussed further in Section
  \ref{sec:coarse}.

  We show that pivots for $X,Y$ and short curves on a $\dth$-geodesic
  from $X$ to $Y$ coarsely coincide in an order-preserving way, once
  again assuming that $X$ and $Y$ are thick:

  \begin{theorem} \label{thm:t11-length-twisting}

    Let $X,Y \in \T(\torus)$ lie in the thick part, and let $\G \from I \to
    \T(\torus)$ be a geodesic of $\dth$ from $X$ to $Y$.  Let $\ell_\alpha$
    denote the minimum of $\ell_\alpha(\G(t))$ for $t \in I$.  We have:

  \begin{rmenumerate}
    \item If $\alpha$ is short somewhere in $\G$, then $\alpha$ is a pivot.
    \item If $\alpha$ is a pivot with large coefficient, then
    $\alpha$ becomes short somewhere in $\G$.
    \item If both $\alpha$ and $\beta$ become short in $\G$, then they do
    so in disjoint intervals whose ordering in $I$ agrees with that of
    $\alpha,\beta$ in $\pivot(X,Y)$.
    \item There is an a priori upper bound on $\ell_\alpha$ for
    $\alpha \in \pivot(X,Y)$. 
  \end{rmenumerate}
  \end{theorem}

  In this statement, various constants have been suppressed (such as
  those required to make \emph{short} and \emph{large} precise). We show
  that all of the constants can be taken to be independent of $X$ and
  $Y$, and the full statement with these constants is given as Theorem
  \ref{thm:pivots-as-short-curves} below. 

  We have already seen that there may be many Thurston geodesics from
  $X$ to $Y$, and due to the asymmetry of the metric, reversing
  parameterization of a geodesic from $X$ to $Y$ does not give a
  geodesic from $Y$ to $X$.  On the other hand, the notion of a pivot
  is symmetric in $X$ and $Y$.  Therefore, by comparing the pivots to
  the short curves of an arbitrary Thurston geodesic, Theorem
  \ref{thm:pivots-as-short-curves} establishes a kind of symmetry and
  uniqueness for the \emph{combinatorics} of Thurston geodesic
  segments, despite the failure of symmetry or uniqueness for the
  geodesics themselves.

  \subsection*{Rigidity}

  A Finsler metric on $\T(S)$ gives each tangent space $T_X \T(S)$ the
  structure of a normed vector space.  Royden showed that for the
  \Teich metric, this normed vector space uniquely determines $X$ up
  to the action of the mapping class group \cite{royden:AT}.  That is,
  the tangent spaces are isometric (by a linear map) if and only if
  the hyperbolic surfaces are isometric.

  We establish the corresponding result for the Thurston's metric on
  $\T(\torus)$ and its corresponding norm \tnorm{\param} (the
  \emph{Thurston norm}) on the tangent bundle.

  \begin{theorem}
  \label{thm:infinitesimal-rigidity}
  Let $X,Y \in \T(\torus)$.  Then there exists an isometry of normed
  vector spaces
  $$(T_X\T(\torus), \tnorm{\param}) \to (T_Y\T(\torus),
  \tnorm{\param})$$
  if and only if $X$ and $Y$ are in the same orbit of the
  extended mapping class group.
  \end{theorem}

  The idea of the proof is to recognize lengths and intersection
  numbers of curves on $X$ from features of the unit sphere in
  $T_{X} \T(S)$.  Analogous estimates for the shape of the cone of
  lengthening deformations of a hyperbolic one-holed torus were
  established in \cite{gueritaud:lengthening}.  In fact, Theorem
  \ref{thm:infinitesimal-rigidity} was known to Gu\'{e}ritaud and can
  be derived from those estimates \cite{gueritaud:email}.  We present
  a self-contained argument that does not use Gu\'{e}ritaud's results
  directly, though \cite[Section~5.1]{gueritaud:lengthening} provided
  inspiration for our approach to the infinitesimal rigidity
  statement.

  A local rigidity theorem can be deduced from the infinitesimal one,
  much as Royden did in \cite{royden:AT}.
  \begin{theorem}
  \label{thm:isometries-local}
  Let $U$ be a connected open set in $\T(\torus)$, considered as a
  metric space with the restriction of $\dth$.  Then any isometric
  embedding $(U,\dth) \to (\T(\torus), \dth)$ is the restriction to $U$
  of an element of the extended mapping class group.
  \end{theorem}
  Intuitively, this says that the quotient of $\T(\torus)$ by the
  mapping class group is ``totally unsymmetric''; each ball fits into
  the space isometrically in only one place.  Of course, applying
  Theorem \ref{thm:isometries-local} to $U = \T(\torus)$ we have the
  immediate corollary
  \begin{corollary}
  \label{cor:isometries-t11}
  Every isometry of $(\T(\torus),\dth)$ is induced by an element of the
  extended mapping class group, hence the isometry group is isomorphic
  to $\PGL(2,\ZZ)$.\qedhere
  \end{corollary}
  Here we have used the usual identification of the mapping class group
  of $\torus$ with $\GL(2,\ZZ)$, whose action on $\T(\torus)$ factors
  through the quotient $\PGL(2,\ZZ)$.

  The analogue of Corollary \ref{cor:isometries-t11} for Thurston's metric
  on higher-dimensional Teichm\"uller spaces was established by Walsh in
  \cite{walsh:HB} using a characterization of the horofunction
  compactification of $\T(S)$.  Walsh's argument does not apply to the
  punctured torus, however, because it relies on Ivanov's
  characterization (in \cite{ivanov:ACC}) of the automorphism group of
  the curve complex (a result which does not hold for the punctured
  torus).

  Passing from the infinitesimal (i.e.~norm) rigidity to local or global
  statements requires some preliminary study of the smoothness of the
  Thurston norm.  In \subsecref{norm} we show that the norm is
  locally Lipschitz continuous on $T \T(S)$ for any finite type
  hyperbolic surface $S$.  By a recent result of Matveev-Troyanov
  \cite{matveev-troyanov}, it follows that any $\dth$-preserving map is
  differentiable with norm-preserving derivative.  This enables the key
  step in the proof of Theorem \ref{thm:isometries-local}, where Theorem
  \ref{thm:infinitesimal-rigidity} is applied to the derivative of the isometry.

  \subsection*{Additional notes and references}
  In addition to Thurston's paper \cite{thurston:MSM}, an exposition of
  Thurston's metric and a survey of its properties can be found in
  \cite{papadopoulos-theret:survey}.  Prior work on the coarse
  geometry of the Thurston metric on Teichm\"uller space and its
  geodesics can be found in \cite{rafi:TL} \cite{LRT1} \cite{LRT2}.
  The notion of the maximally-stretched lamination for a pair of
  hyperbolic surfaces has been generalized to higher-dimensional
  hyperbolic manifolds \cite{kassel} \cite{gueritaud-kassel} and to
  vector fields on $\HH^2$ equivariant for convex cocompact subgroups
  of $\PSL(2,\RR$) \cite{DGK16}.

  \subsection*{Acknowledgments}
  The authors thank the American Institute of Mathematics for hosting
  the workshop ``Lipschitz metric on Teichm\"uller space'' and the
  Mathematical Sciences Research Institute for hosting the semester
  program ``Dynamics on Moduli Spaces of Geometric Structures'' where
  some of the work reported here was completed.  The authors
  gratefully acknowledge grant support from NSF DMS 0952869 and DMS
  1709877 (DD), NSF DMS 1611758 (JT), NSERC RGPIN 435885 (KR), and
  from NSF DMS 1107452, 1107263, 1107367 ``RNMS: GEometric structures
  And Representation varieties'' (the GEAR Network).  The authors also
  thank Fran\c{c}ois Gu\'eritaud for helpful conversations related to
  this work, and specifically for suggesting the statement of
  \thmref{Length}.  Finally, the authors thank the anonymous referees
  for their careful reading of the paper and for helpful comments and
  corrections.

\section{Background}

  \subsection{Approximate comparisons}
  \label{Subsec:approx-comparisons}

  We use the notation $a \emul b$ to mean that quantities $a$ and $b$
  are equal up to a uniform multiplicative error, i.e.~that there
  exists a positive constant $K$ such that $K^{-1} a \leq b \leq Ka$.
  Thus for example $a \emul 1$ means that $a$ is bounded above and
  below by positive constants.  Similarly, the notation $a \lmul b$
  means that $a \leq K b$ for some $K$.

  The analogous relations up to additive error are $a \eadd b$,
  meaning that there exists $C$ such that $a-C \leq b \leq a+C$, and
  $a \ladd b$ which means $a \leq b + C$ for some $C$.  Hence
  $a \eadd 0$ means that $a$ is bounded above and below by constants.

  For equality up to both multiplicative and additive error, we write
  $a \emuladd b$.  That is, $a \emuladd b$ means that there exist
  constants $K,C$ such that $K^{-1} a - C \leq b \leq Ka + C$.

  Unless otherwise specified, the implicit constants depend only on the
  topological type of the surface $S$. When the constants depend on the
  Riemann surface $X$, we use the notation $\emul_X$ and $\eadd_X$
  instead. 

  For functions $f,g$ of a real variable $x$ we write $f \sim
  g$ to mean that $\lim_{x \to \infty} \frac{f(x)}{g(x)} = 1$.

  \subsection{Surfaces, curves, and laminations}
  \label{Subsec:curve-lam}
  
  Throughout this paper \s denotes an oriented surface of finite type,
  i.e.~the complement of a finite subset $P$ of the interior of
  $\bar{S}$, a compact oriented surface with boundary.  Elements of
  $P$ are the \emph{punctures}.

  A \emph{multicurve} is a closed 1-manifold on \s defined up to
  homotopy such that no connected component is homotopic to a point, a
  puncture, or boundary of \s. A connected multicurve will just be
  called a \emph{curve}. Note that with our definition, there are no
  curves on the two- or three-punctured sphere, so we will ignore
  those cases henceforth.  The \emph{geometric intersection number}
  $\I(\alpha,\beta)$ between two curves is the minimal number of
  intersections between representatives of $\alpha$ and $\beta$. If we
  fix a hyperbolic metric on \s, then every (multi)curve has a unique
  geodesic representative, and $\I(\alpha,\beta)$ is just the number
  of intersections between the geodesic representative of $\alpha$ and
  the geodesic representative of $\beta$. For any curve $\alpha$ on
  \s, we denote by $D_\alpha$ the \emph{left} Dehn twist about
  $\alpha$.

  Fix a complete hyperbolic metric of finite area on \s, so that the
  boundary components (if any) are geodesic.  A \emph{geodesic
    lamination} $\lambda$ on \s is a closed subset which is a disjoint
  union of simple complete geodesics.  These geodesics are called the
  \emph{leaves} of $\lambda$.  Two different hyperbolic metrics on \s
  determine canonically isomorphic spaces of geodesic laminations, so
  the space of geodesic laminations $\Lam(\s)$ depends only on the
  topology of \s. This is a compact metric space equipped with the
  metric of Hausdorff distance on closed sets. The closure of the set
  of multicurves in $\Lam(\s)$ is the set of \emph{chain-recurrent}
  laminations.
  
  We will call a geodesic lamination \emph{maximal chain-recurrent} if
  it is chain-recurrent and not properly contained in another
  chain-recurrent lamination. A geodesic lamination is \emph{complete}
  if its complementary regions in \s are ideal triangles.  Note that
  all chain-recurrent laminations are necessarily compactly supported.
  Thus, when \s has punctures, a chain-recurrent lamination can never
  be complete.  For a given geodesic lamination $\lambda$, we refer to
  any complete lamination containing $\lambda$ as a \emph{completion}
  (of $\lambda)$.  

  In the case of the punctured torus \torus, the maximal
  chain-recurrent laminations are types (b) and (c) in Figure
  \ref{fig:torus}.  Case (b), i.e.~a curve and a spiraling geodesic,
  will be especially important in the sequel, and so we introduce the
  following notation for these laminations: Given a curve $\alpha$,
  let $\alpha_0^+ = \alpha \cup \delta$ where the geodesic $\delta$
  spirals toward $\alpha$ in each direction, turning to the
  \emph{left} as it does so.  Similarly we define $\alpha^-_0$ to be
  the union of $\alpha$ and a spiraling leaf that turns right.
  (Adding a leaf that turns opposite ways on its two ends yields a
  non-chain-recurrent lamination.)

  The motivation for this sign convention for $\alpha_0^\pm$ is that
  it is compatible with a common way to describe simple curves on
  $\torus$ in terms of \emph{slope} while regarding $\alpha$ as
  vertical.  More precisely, consider an oriented curve $\vec{\eta}$
  with $\I(\eta,\alpha)=1$, and let $\vec{\alpha}$ denote the
  orientation of $\alpha$ so that the homology classes
  $[\vec{\eta}],[\vec{\alpha}]$ give a positive ordered basis of
  $H_1(\torus)$ with respect to the orientation of $\torus$.  If a simple
  curve $\gamma \neq \alpha$ has homology class
  $q [\vec{\eta}] + p [\vec{\alpha}]$ for some orientation, then
  $p/q \in \QQ$ is the slope of $\gamma$ (relative to that basis).  We
  consider $\alpha$ itself to have slope $1/0 = \infty \in \QP^1$ and
  this exhibits a bijection between $\QP^1$ and the set of simple
  curves on $\torus$.

  Now, a sequence of simple curves distinct from $\alpha$ whose slopes
  go to $+\infty$ have Hausdorff limit $\alpha_0^+$, while a sequence
  with slopes going to $-\infty$ has Hausdorff limit $\alpha_0^-$.
  Thus $\alpha_0^+$ (resp.~$\alpha_0^-$) is approximated by curves of
  large positive (resp.~negative) slope.

  All of the maximal chain-recurrent laminations on $\torus$ have a
  single complementary region, which is a punctured bigon.  Such a
  lamination therefore has exactly three completions, corresponding to
  the three ways to add leaves that cut the bigon into ideal triangles
  shown in Figure \ref{fig:completions}.  (For more detail on
  classifying laminations on the punctured torus, we refer the reader
  to \cite{BZ04}.)
  
  A convenient way to distinguish among the completions of a maximal
  chain-recurrent lamination $\lambda$ on the punctured torus is to
  use the hyperelliptic involution.  This is an involutive
  orientation-preserving isometry $\iota$ that preserves every simple
  closed geodesic, and thus every chain-recurrent lamination.  The
  action of $\iota$ on the complementary bigon of a maximal
  chain-recurrent lamination exchanges the two spikes, and therefore
  the only completion which is $\iota$-invariant is the one with
  leaves going to both spikes, i.e.~type (i) in Figure
  \ref{fig:completions}.  We call this the \emph{canonical completion}
  of $\lambda$.

  We denote the canonical completion of $\alpha_0^+$ by $\alpha^+$,
  and that of $\alpha_0^-$ by $\alpha^-$.  Thus
  $\alpha^\pm = \alpha_0^\pm \cup w \cup w'$ where $w$ and $w'$ are
  leaves emanating from the puncture and spiraling into $\alpha$.  For
  example, $\alpha^+$ is shown in Figure~\ref{fig:alphaplus}.

  \begin{figure}
  \begin{center}
  \includegraphics{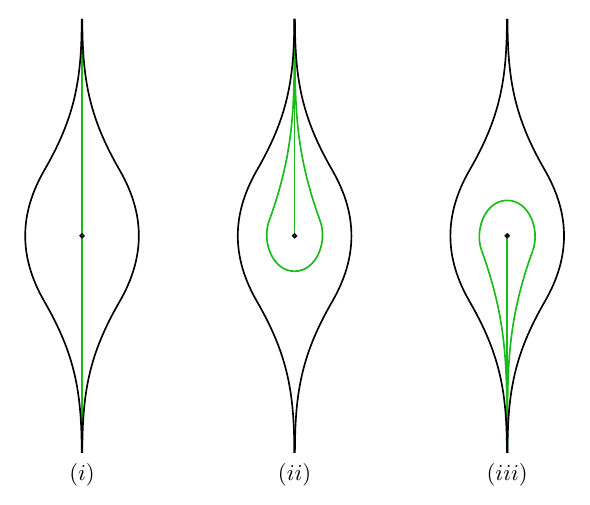}
  \end{center}
  \caption{The three ways to complete a maximal chain-recurrent
    lamination on \torus by adding two leaves in its complementary
    bigon.}
  \label{fig:completions}
  \end{figure}

  The \emph{stump} of a geodesic lamination (in the terminology of
  \cite{theret:negative-convergence}) is its maximal compactly-supported
  sublamination that admits a transverse measure of full support.

  \subsection{\Teich space}
  \label{Subsec:teich}

  Let $\T(\s)$ be the \Teich space of complete finite-area hyperbolic
  structures on \s.  We will only consider $\T(\s)$ in cases where \s
  has no boundary.  The space $\T(\s)$ is homeomorphic to
  $\RR^{6g-6+2n}$, if \s has genus $g$ and $n$ punctures.  Given
  $X \in \T(\s)$ and a curve $\alpha$ on \s, we denote by
  $\ell_\alpha(X)$ the length of the geodesic representative of
  $\alpha$ on $X$.  For brevity we refer to $\ell_\alpha(X)$ as the
  \emph{length of $\alpha$ on $X$}.

  For any $\ep > 0$, we will denote by $\T_\ep(\s)$ the set of points
  in $\T(\s)$ on which every curve has length at least $\ep$; this is
  the $\ep$-\emph{thick part} of \Teich space.

  A positive real number $\ep$ is called a (two-dimensional)
  \emph{Margulis number} if two distinct curves on a hyperbolic
  surface of length less than $\ep$ are necessarily disjoint.  Fix a
  Margulis number $\ep_M < 1$ such that that for any curve $\alpha$ of
  length less than $\ep_M$, the shortest curve $\beta$ that intersects
  $\alpha$ has $\I(\alpha,\beta) \leq 2$.  It follows from the collar
  lemma that any sufficiently small $\ep_M$ has this property.
  
  \subsection{Shearing of ideal triangles}
  \label{Subsec:shearing-tri}

  Let \HH denote the upper half plane model of the hyperbolic plane,
  with ideal boundary $\partial \HH = \RR \cup \set{\infty}$. In this
  section, we will define the shearing of two ideal triangles in \HH
  which share an ideal vertex. This is a specific case of the more
  general shearing defined in \cite[Section~2]{bonahon:shearing}.
  
  Two distinct points $x, y \in \partial \HH$ determine a geodesic
  $[x,y]$ and three distinct points $x, y, z \in \partial \HH$
  determine an ideal triangle $\Delta(x,y,z)$. Recall that an ideal
  triangle in \HH has a unique inscribed circle which is tangent to
  all three sides of the triangle. Each tangency point is called the
  \emph{midpoint} of the side.

  Let $\gamma=[\gamma^+,\gamma^-]$ be a geodesic in $\HH$. Suppose two
  ideal triangles $\Delta$ and $\Delta'$ lie on different sides of
  $\gamma$. We allow the possibility that $\gamma$ is an edge of $\Delta$
  or $\Delta'$ (or both). Suppose $\Delta$ is asymptotic to $\gamma^+$ and
  the $\Delta'$ is asymptotic to $\gamma^-$. Let $m$ be the midpoint along
  the side of $\Delta$ closest to $\gamma$. The pair $\gamma^+$ and $m$
  determine a horocycle that intersects $\gamma$ at a point $p$. Let $m'$
  and $p'$ be defined similarly using $\Delta'$ and $\gamma^-$. We say $p'$
  is to the left of $p$ (relative to $\Delta$ and $\Delta'$) if the path
  along the horocycle from $m$ to $p$ and along $\gamma$ from $p$ and $p'$
  turns left; $p'$ is to the right of $p$ otherwise. Note that $p'$ is to
  the left of $p$ if and only if $p$ is to the left of $p'$. The shearing
  $s_\gamma(\Delta,\Delta')$ along $\gamma$ relative to the two triangles
  is the signed distance between $p$ and $p'$, where the sign is positive
  if $p'$ is to the left of $p$ and negative otherwise. Note that this sign
  convention gives $s_\gamma(\Delta, \Delta') = s_\gamma(\Delta',\Delta)$.

  \subsection{Shearing coordinates in \Teich space}
  \label{Subsec:shearing-t11}
  
  Given any complete geodesic lamination $\lambda$, there is an
  embedding $s_\lambda \from \T(\s) \to \RR^N$ by the \emph{shearing
    coordinates} relative to $\lambda$, where $N = \dim \T(\s)$.  The
  image of this embedding is an open convex cone.  Details of the
  construction of this embedding can be found in
  \cite{bonahon:shearing} and \cite[Section~9]{thurston:MSM}.

  Using the shearing of ideal triangles discussed above, we will
  define the shearing coordinates in the case where $\lambda$ is the
  canonical completion of a maximal chain-recurrent lamination on
  \torus with finitely many leaves.  That is, we consider
  $\lambda = \alpha^+$ or $\lambda = \alpha^-$ for a simple curve
  $\alpha$, and describe the map
  $s_\lambda \from \T(\torus) \to \RR^2$.
  
  \begin{figure}
  \begin{center}
  \includegraphics{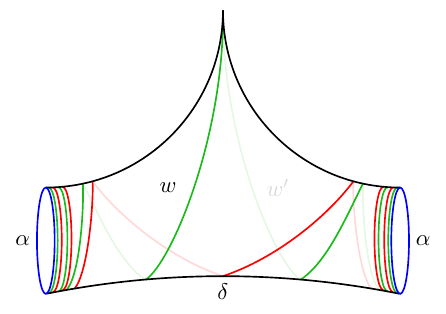}
  \end{center}
  \caption{Leaves of $\alpha^+$ (the canonical completion of
    $\alpha_0^+$) shown in the torus cut open along $\alpha$.}
  \label{fig:alphaplus}
  \end{figure}

  We begin with an auxiliary map $s_\lambda^0 : \T(\torus) \to \RR^4$
  which records a shearing parameter for each leaf of $\lambda$, and
  then we identify the $2$-dimensional subspace of $\RR^4$ that
  contains the image in this specific situation.

  Let $l$ be a leaf of $\lambda$ and fix a lift $\tl$ of $l$ to
  $\tx = \HH$.  If $l$ is a non-compact leaf, then $l$ bounds two
  ideal triangles in $X$, which admit lifts $\Delta$ and $\Delta'$
  with common side $\tl$. If $l = \alpha$ is the compact leaf, then we
  choose $\Delta$ and $\Delta'$ to be lifts of the two ideal triangles
  complementary to $\lambda$ that lie on different sides of $\tl$
  and which are each asymptotic to one of the ideal points of
  $\tl$.  Now define $s_l(X) = s_{\tl}(\Delta,\Delta')$, and let the
  $s_\lambda^0 : \T(\torus) \to \RR^4$ be the map defined by
  \[ s_\lambda^0(X) = (s_\delta(X), s_\alpha(X), s_{w}(X),
  s_{w'}(X)).\]
  We claim that in fact, $s_w(X) = s_{w'}(X) = 0$ and that
  $s_\delta(X) = \mp \ell_\alpha(X)$ for $\lambda = \alpha^\pm$.
  It will then follow that $s_\lambda^0$ takes values in a
  $2$-dimensional linear subspace of $\RR^4$, allowing us to
  equivalently consider the embedding
  $s_\lambda : \T(\torus) \to \RR^2$ defined by
  \[ s_\lambda(X) = (\ell_\alpha(X), s_\alpha(X)). \]

  To establish the claim, cut the surface $X$ open along $\alpha$ to
  obtain a pair of pants which is further decomposed by $w,w',\delta$
  into a pair of ideal triangles.  The boundary lengths of this
  hyperbolic pair of pants are $\ell_\alpha$, $\ell_\alpha$, and $0$.
  Gluing a pair of ideal triangles along their edges but with their
  edge midpoints shifted by signed distances $a,b,c$ gives a pair of
  pants with boundary lengths $|a+b|, |b+c|, |a+c|$, and with the
  signs of $a+b, b+c, a+c$ determining the direction in which the
  seams spiral toward those boundary components (this is discussed in
  more detail in \cite[Section~3.9]{thurston:GT}).  Specifically, a
  positive sum corresponds to the seam turning to the right while
  approaching the corresponding boundary geodesic, and a negative sum
  corresponds to the seam turning to the right.  Applying this to our
  situation, and recalling that for $\lambda = \alpha^+$ all spiraling
  leaves turn left when approaching the boundary of the pair of pants,
  and we obtain
  \[ s_w(X) + s_\delta(X) = s_{w'}(X)+s_\delta(X) = -\ell_\alpha \]
  and
  \[ s_w(X) + s_{w'}(X) = 0. \] This gives $s_w(X) = s_{w'}(X) = 0$
  and $s_\delta(X) = -\ell_\alpha(X)$.  For the case
  $\lambda = \alpha^-$ the equations are the same except that
  $-\ell_\alpha$ is replaced by $\ell_\alpha$, and the solution
  becomes $s_w(X) = s_{w'}(X) = 0$ and $s_\delta(X) = \ell_\alpha$.

  Finally, we consider the effect of the various choices made in the
  construction of $s_\lambda(X)$. The coordinate $\ell_\alpha(X)$ is of
  course canonically associated to $X$, and independent of any choices.
  For $s_\alpha(X)$, however, we had to choose a pair of triangles
  $\Delta,\Delta'$ on either side of the lift $\Tilde{\alpha}$.  In this
  case, different choices differ by finitely many moves in which one of
  the triangles is replaced by a neighbor on the other side of a lift of
  $w$, $w'$, or $\delta$.  Each such move changes the value of
  $s_\alpha(X)$ by adding or subtracting one of the values $s_{w}(X)$,
  $s_{w'}(X)$, or $s_\delta(X)$; this is the additivity of the shearing
  cocycle established in \cite[Section~2]{bonahon:shearing}.  By the
  computation above each of these moves actually adds $0$ or $\pm
  \ell_\alpha(X)$.  Hence $s_\alpha(X)$ is uniquely determined
  up to addition of an integer multiple of $\ell_\alpha(X)$.
  
  \subsection{The Thurston metric}
  \label{Subsec:thurston}

  For a pair of points $X, Y \in \T(\s)$, in the introduction we defined
  the quantity
  $$\dth(X,Y) = \sup_{\alpha} \log \frac{\ell_\alpha(Y)}{\ell_\alpha(X)}$$
  where the supremum is taken over all simple curves.  Another measure
  of the difference of hyperbolic structures, in some ways dual to this
  length ratio, is
  $$ L(X,Y) = \inf_f \log L_f$$
  where $L_f$ is the Lipschitz constant, and where the infimum is
  taken over Lipschitz maps $f : X \to Y$ in the preferred homotopy
  class.  Thurston showed:

  \begin{theorem} \label{Thm:Thurston}
  For all $X,Y \in \T(\s))$ we have $\dth(X,Y) = L(X,Y)$, and this
  function is an asymmetric metric, i.e.~it is positive unless $X=Y$ and
  it obeys the triangle inequality.
  \end{theorem}

  Denote by $\dths(X,Y) = \max \set{ \dth(X,Y), \dth(Y,X)}$. The
  topology of $\T(\s)$ is compatible with $\dths$, so by $X_i \to X$
  we will mean $\dths(X_i,X) \to 0$. By the Hausdorff distance on
  closed sets in $\T(\s)$ we will mean with respect to the metric
  $\dths$.

  Thurston showed that the infimum Lipschitz constant is realized by a
  homeomorphism from $X$ to $Y$.  Any map which realizes the infimum is
  called \emph{optimal}.

  Further, Thurston constructs a chain-recurrent lamination
  $\Lambda(X,Y)$ such that there exists a $e^{\dth(X,Y)}$-Lipschitz
  map in the preferred homotopy class from a neighborhood of
  $\Lambda(X,Y)$ in $X$ to an neighborhood of the same lamination in
  $Y$, multiplying arc length along $\Lambda(X,Y)$ by a factor of
  $e^{\dth(X,Y)}$, and so that $\Lambda(X,Y)$ is the largest
  chain-recurrent lamination with this property.  We call
  $\Lambda(X,Y)$ the \emph{maximally-stretched lamination} (from $X$
  to $Y$).  The same lamination is also characterized in terms of
  optimal maps: $\Lambda(X,Y)$ is the largest chain-recurrent
  lamination such that every optimal map from $X$ to $Y$ multiplies arc
  length on $\Lambda(X,Y)$ by a factor of $e^{\dth(X,Y)}$.
  
  The length ratio for simple curves extends continuously to $\PML(\s)$,
  which is compact.  Therefore, the length-ratio supremum is always
  realized by some measured lamination.  Any measured lamination that
  realizes the supremum has support contained in the stump of
  $\Lambda(X,Y)$.

  Suppose a parameterized path $\G \from [0,d] \to \T(\s)$ is a geodesic
  from $X$ to $Y$ (parameterized by unit speed). Then the following holds:
  for any $s,t \in [0,d]$ with $s < t$ and for any arc $\omega$ contained
  in the geometric realization of $\Lambda(X,Y)$ on $X$, the arc length of
  $\omega$ is stretched by a factor of $e^{t-s}$ under an optimal map from
  $\G(s)$ to $\G(t)$. We will sometimes denote $\Lambda(X,Y)$ by
  $\lambda_\G$.

  \subsection{Stretch paths}
  \label{Subsec:stretch}

  Certain geodesics of Thurston's metric can be described using
  shearing coordinates.  Let $\lambda$ be a complete geodesic
  lamination and $X \in \T(\s)$.  For any $t \in \RR$ let
  $\str(X, \lambda, t)$ be the unique point in $\T(\s)$ such that
  \[s_\lambda(\str(X, \lambda, t)) = e^t s_\lambda(X).\] Letting $t$
  vary, we have that $\str(X,\lambda,t)$ is a parameterized path in
  $\T(\s)$ that maps to an open ray from the origin in $\RR^N$ under
  the shearing coordinates.  This is the \emph{stretch path along
    $\lambda$ from $X$}.

  Thurston showed that the path $t \mapsto \str(X,\lambda,t)$ is a
  geodesic in $\T(\s)$ in the sense of \eqref{eqn:forward-geodesic}.
  Note that we always consider the stretch path to be oriented in the
  direction of increasing $t$, which is natural since the asymmetry of
  the metric implies that the same path parameterized in the opposite
  direction may not be geodesic.

  Also, if $\lambda_0 \subset \lambda$ is the largest chain-recurrent
  sublamination, then $\lambda_0$ is the maximally-stretched
  lamination for any pair of points $\str(X,\lambda,s)$ and
  $\str(X,\lambda,t)$ with $s<t$.

  Removing the point $X$ from a stretch path from $X$ leaves two
  (open) \emph{stretch rays}; of these, the one corresponding to
  $t > 0$ is a stretch ray \emph{starting at $X$} and that with $t<0$
  is the one \emph{ending at $X$}.
    
  Thurston used stretch paths to show that $\T(\s)$ equipped with the
  Thurston metric is a geodesic metric space.  We summarize his
  results below. See the statement and proof of \cite[Theorem
  8.5]{thurston:MSM}) for more details.

  \begin{theorem}[\cite{thurston:MSM}]\label{Thm:Geodesics}

    For any $X, Y \in \T(\s)$, let $\Lambda(X,Y)$ be the
    maximally-stretched lamination from $X$ to $Y$. Let $\lambda$ be any
    completion of $\Lambda(X,Y)$. Then there exists a geodesic $\calG$
    from $X$ to $Y$ consisting of a finite concatenation of stretch path
    segments \[ \calG = \calG_1 \cdots \calG_n,\] where $\calG_1$ is a
    segment of $\str(X,\lambda,t)$, and all other $\calG_i$'s stretch
    along some complete lamination containing $\Lambda(X,Y)$.
    Furthermore, such a geodesic can be chosen so that if $X_i$ is the
    initial point of $\calG_i$, then for all $i>1$ we have $\Lambda(X_i,Y)
    \supsetneq \Lambda(X_{i-1},Y)$.  In particular, we can always take 
    $n \le 2 |\chi(\s)|$.

  \end{theorem}

  In general, geodesics of the Thurston metric from $X$ to $Y$ are not
  unique. But when $\Lambda(X,Y)$ is maximal chain-recurrent, then
  there is a unique geodesic. This statement follows from
  \thmref{Geodesics} but it is not explicitly stated in
  \cite{thurston:MSM}. For completeness, we provide a proof:

  \begin{corollary} \label{Cor:Geodesics}

    Given $X, Y \in \T(\s)$, suppose $\Lambda(X,Y)$ is maximal
    chain-recurrent. Let $\lambda$ be a completion of
    $\Lambda(X,Y)$. Then $\str(X,\lambda,t)$ is the unique geodesic
    from $X$ to $Y$. In particular, for the punctured torus $\torus$,
    the three completions of $\Lambda(X,Y)$ give rise to the same
    stretch path in $\T(\torus)$.

  \end{corollary}

  \begin{proof}

    We first show that the stretch path for $\lambda$ connects $X$ to
    $Y$, i.e.~$\str(X,\lambda,t)=Y$ for some $t$. By
    \thmref{Geodesics}, there is a geodesic path $\calG$ from $X$ to
    $Y$ consisting of a concatenation of segments along stretch paths
    $\calG_1,\ldots,\calG_n$, where $\calG_1$ is a segment of
    $\{ \str(X,\lambda,t) \st t \geq 0 \}$. Let $X_i$ be the
    initial point of $\calG_i$. If $n \ge 2$, then
    $\Lambda(X,Y) = \Lambda(X_1,Y) \subsetneq \Lambda(X_2,Y)$ by
    \thmref{Geodesics}. But this is impossible since $\Lambda(X,Y)$ is
    maximal chain-recurrent, so $n=1$ and $Y$ lies on $\cal G_1$.

    Now suppose $\cal G$ is any geodesic from $X$ to $Y$. Let $Z$ be a
    point on $\cal G$. We have $\Lambda(X,Y) \subset
    \Lambda(X,Z)$. Since $\Lambda(X,Y)$ is maximal chain-recurrent,
    $\Lambda(X,Z) = \Lambda(X,Y)$. By the previous discussion, we can
    connect $X$ to $Z$ by a segment of $\str(X,\lambda,t)$. Since this
    true for all $Z$ in $\calG$, the geodesic $\calG$ must be a
    segment of $\str(X,\lambda,t)$. \qedhere
    
  \end{proof}

  \subsection{Twisting}
  \label{Subsec:twisting}

  There are several notions of twisting which we will define below.
  While these notions are defined for different classes of objects, in
  cases where several of the definitions apply, they are equal up to an
  additive constant.

  Let $A$ be an annulus. Fix an orientation of the core curve $\alpha$ of
  $A$. For any simple arc $\omega$ in $A$ with endpoints on different
  components of $\partial A$, we orient $\omega$ so that the algebraic
  intersection number $\omega \cdot \alpha$ is equal to one. Given an
  ordered pair of simple arcs $\omega$ and $\omega'$, the choice of the
  orientation above allows us to assign a sign to each intersection point
  in the interior of $A$ between $\omega$ and $\omega'$. The sum $\omega
  \cdot \omega'$ of these signed intersections is called the
  \emph{algebraic intersection number} between $\omega$ and $\omega'$. Note
  that $\omega \cdot \omega'$ is independent of the choice of the
  orientation of $\alpha$. Also note that we do not consider intersections
  between $\omega$ and $\omega'$ in the boundary of $A$. With our choice,
  we always have $\omega \cdot D_\alpha(\omega) = 1$, where as above
  $D_\alpha$ denotes the left Dehn twist about $\alpha$.

  Now let \s be a surface and $\alpha$ is a simple closed curve on
  $S$. Let $\widehat{S} \to \s$ be the covering space associated to
  $\pi_1(\alpha) < \pi_1(S)$.  Then $\widehat{S}$ has a natural Gromov
  compactification that is homeomorphic to a closed annulus.  By
  construction, the core curve $\widehat{\alpha}$ of this annulus maps
  homeomorphically to $\alpha$ under this covering map.

  Let $\lambda$ and $\lambda'$ be two geodesic laminations (possibly
  curves) on \s, both intersecting $\alpha$ transversely. We define their
  (signed) twisting relative to $\alpha$ as
  $\twist_\alpha(\lambda,\lambda') = \min {\widehat{\omega} \cdot
  \widehat{\omega}'}$, where $\widehat{\omega}$ is a lift of a leaf of
  $\lambda$ and $\widehat{\omega'}$ is a lift of a leaf of $\lambda'$, with
  both lifts intersecting $\widehat{\alpha}$, and the minimum is taken over
  all such leaves and their lifts. Note that for any two such lifts
  $\omega$ and $\omega'$ (still intersecting $\widehat{\alpha}$)  the
  quantity $\widehat{\omega} \cdot \widehat{\omega}'$ exceeds
  $\twist_\alpha(\lambda,\lambda')$ by at most $2$.

  Next we define the twisting of two hyperbolic metrics $X$ and $Y$ on
  \s relative to $\alpha$.  Let $\widehat{X}, \widehat{Y}$ denote the
  lifts of these hyperbolic structures to $\widehat{S}$.  Using the
  hyperbolic structure $\widehat{X}$, choose a geodesic
  $\widehat{\omega}$ that is orthogonal to the geodesic in the homotopy
  class of $\widehat{\alpha}$.  Let $\widehat{\omega}'$ be a geodesic
  constructed similarly from $\widehat{Y}$.  We set
  $\twist_\alpha(X,Y) = \min \widehat{\omega} \cdot \widehat{\omega}'$, where
  the minimum is taken over all possible choices for $\widehat{\omega}$
  and $\widehat{\omega}'$.  Similar to the previous case, this minimum
  differs from the intersection number $\widehat{\omega} \cdot
  \widehat{\omega}'$ for a particular pair of choices by at most $2$.

  Finally, we define $\twist_\alpha(X,\lambda)$, the twisting of a
  lamination $\lambda$ about a curve $\alpha$ on $X$. This is defined if
  $\lambda$ contains a leaf that intersects $\alpha$ transversely. Let
  $\widehat{\omega}$ be a geodesic of $\widehat{X}$ orthogonal to the
  geodesic homotopic to $\widehat{\alpha}$.  Let $\omega'$ be any leaf of
  $\lambda$ intersecting $\alpha$, and let $\widehat{\omega}'$ be a lift of
  this leaf to $\widehat{X}$ which intersects $\widehat{\alpha}$. Then
  $\twist_\alpha(X,\lambda) = \min{\widehat{\omega} \cdot
  \widehat{\omega}'}$, with the minimum taken over all choices of
  $\omega'$, $\widehat{\omega}'$, and $\widehat{\omega}$.  

  Each type of twisting defined above is \emph{signed}.  In some cases
  the absolute value of the twisting is the relevant quantity; we use
  the notation $d_\alpha(\param,\param) =
  \abs{\twist_\alpha(\param,\param)}$ for the corresponding
  \emph{unsigned twisting} in each case.

  The following way to compute the unsigned twisting $d_\alpha(X,\lambda) =
  \abs{\twist_\alpha(X,\lambda)}$ will be useful in the sequel.  Consider
  the universal cover $\tx \homeo \HH$. Let $\talpha$ be a lift of $\alpha$
  and let $\tomega'$ be a lift of a leaf of $\lambda$ intersecting
  $\talpha$. Let $L$ be the length of the orthogonal projection of
  $\tomega'$ to $\talpha$ and let $\ell$ be the length of the geodesic
  representative of $\alpha$ on $X$.  Let $\tomega$ be an orthogonal
  geodesic of $\talpha$.  There is a loxodromic isometry $T$ of $\HH$
  associated to $\alpha$ that preserves $\talpha$, and applying powers of
  this isometry to $\tomega$ gives a family of orthogonal geodesics to
  $\talpha$ which meet it at points spaced by distance $\ell$.  Then
  $d_\alpha(X,\lambda)$ is the number of these translates that intersect
  $\tomega'$, as each such translate gives one intersection in the quotient
  $\widehat{X} = \HH/\langle T \rangle$ considered above. Therefore, this
  number is between $(\lfloor L/\ell\rfloor-1)$ and $\lfloor L/\ell
  \rfloor$, and $d_\alpha(X,\lambda) \eadd L/\ell$ with additive error at
  most $2$ (see also \cite[Section 3]{minsky:PR} for more details).

\section{Twisting parameter along a Thurston geodesic}

  \label{sec:twisting}

  In this section, \s is any oriented surface of finite type and $\T(S)$ is
  the associated \Teich space.

  Recall that $\T_{\ep}(S)$ denotes the $\ep$-thick part of $\T(S)$.
  Consider two points $X,Y \in \T_{\ep}(S)$.  Recall that we say a
  curve $\alpha$ \emph{interacts} with a geodesic lamination $\lambda$
  if $\alpha$ is a leaf of $\lambda$ or if $\alpha$ intersects
  $\lambda$ essentially.  Suppose $\alpha$ is a curve that interacts
  with $\Lambda(X,Y)$.  Let $\G \from[0,T] \to \T$ be any geodesic from 
  $X$ to $Y$, and let $\ell_\alpha = \min_t \ell_\alpha(t)$.  We are
  interested in curves which become short somewhere along $\G$.  We call an interval 
  of time $[a,b] \subset [0,T]$ the \emph{active interval} for $\alpha$ along $\G$ if $[a,b]$ 
  is the maximal such interval with $\ell_\alpha(a) = \ell_\alpha(b) = \ep$.  Note that any 
  curve which is sufficiently short somewhere on $\G$ has a nontrivial active
  interval.

  The main goal of this section is to prove the following theorem,
  which in particular establishes Theorem \ref{thm:main-length-twisting}.
  As in the introduction we use the notation $\Log(x) = \min(1, \log(x))$.
  Denote $X_t=\G(t)$.
  \begin{theorem}\label{Thm:Short}
    There exists a constant $\ep_0$ such that the following statement
    holds. Let $X,Y \in \T_{\ep_0}(S)$ and $\alpha$ be a curve that
    interacts with $\Lambda(X,Y)$. Let $\G$ be any geodesic from $X$ to $Y$
    and $\ell_\alpha = \min_t \ell_\alpha(t)$. Then
    \[d_\alpha(X,Y) \emuladd \frac 1{\ell_\alpha} \Log \frac
    1{\ell_\alpha}.\]
    If $\ell_\alpha < \ep_0$, then
    $d_\alpha(X,Y) \eadd d_\alpha(X_a,X_b)$, where $[a,b]$ is the active
    interval for $\alpha$. Further, for all sufficiently small
    $\ell_\alpha$, the twisting $d_\alpha(X_t,\lambda)$ is uniformly
    bounded for all $t \leq a$ and
    $\ell_\alpha(t) \emul e^{t-b} \ell_\alpha(b)$ for all $t \geq b$. All
    errors in this statement depend only on $\ep_0$.
  \end{theorem}

  Note that if $\alpha$ is a leaf of $\Lambda(X,Y)$, then it does not
  have an active interval because its length grows exponentially along
  $\G$, and the theorem above says that in this $d_\alpha(X,Y)$ is
  uniformly bounded. If $\alpha$
  crosses a leaf of $\Lambda(X,Y)$, then $d_\alpha(X,Y)$ is large if
  and only if $\alpha$ gets short along any geodesic from $X$ to
  $Y$. Moreover, the minimum length of $\alpha$ is the same for any
  geodesic from $X$ to $Y$, up to a multiplicative constant.  Further,
  the theorem says that, essentially, all of the twisting about $\alpha$ occurs in
  the active interval $[a,b]$ of $\alpha$.

  Before proceeding to the proof of the theorem, we need to introduce
  a notion of horizontal and vertical components for a curve that
  crosses a leaf of $\Lambda(X,Y)$ and analyze how their lengths
  change in the active interval.  This analysis will require some
  lemmas from hyperbolic geometry.

  \begin{lemma} \label{Lem:Quad}

    Let $\omega$ and $\omega'$ be two disjoint geodesics in \HH with no
    endpoint in common. Let $p \in \omega$ and $p' \in \omega'$ be the
    endpoints of the common perpendicular between $\omega$ and $\omega'$.
    Let $x \in \omega$ be arbitrary and let $x'\in \omega'$ be the point
    on the same side of $[p, p']$ as $x$ such that 
    $d_\HH(x,p) = d_\HH(x',p')$. Then
    \begin{align} \label{Quad1}
      \sinh \frac{d_\HH(p,p')}{2} \cosh d_\HH (x,p) = \sinh
      \frac{d_\HH(x,x')}{2}.
    \end{align}
    For any $y \in \omega'$, we have
    \begin{align} \label{Quad2}
      \sinh d_\HH(p,p') \cosh d_\HH (x,p) \leq \sinh d_\HH(x,y)
    \end{align}
    and
    \begin{align} \label{Quad3}
      d_\HH(x',y) \leq d_\HH(x,y).
    \end{align}

  \end{lemma}

  \begin{figure}
  \begin{center}
  \includegraphics{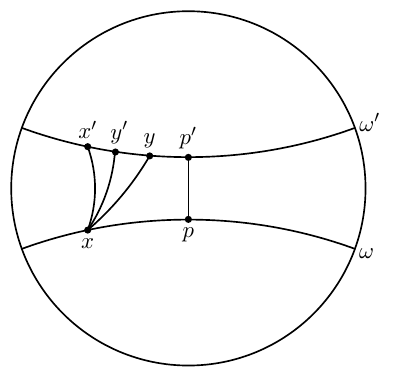}
  \end{center}
  \caption{Saccheri and Lambert quadrilaterals.}
  \label{Fig:Quad}
  \end{figure}

  \begin{proof}

    We refer to \figref{Quad} for the proof. Equation \eqref{Quad1} is well
    known, as the four points $x, x', p', p$ form a Saccheri quadrilateral.
    The point $y' \in \omega'$ closest to $x$ has $\angle x y' p' = \pi/2$,
    so $x, y', p',p$ forms a Lambert quadrilateral and the following
    identity holds
    \[
    \sinh d_\HH(p,p') \cosh d_\HH (x,p) = \sinh d_\HH(x,y') \\
    \]
    Equation \eqref{Quad2} follows since $d_\HH(x,y') \leq d_\HH(x,y)$. For
    \eqref{Quad3}, set $A = \angle x x' y$ and $B = \angle x' x y$ and
    consider the triangle $\triangle xx'y$. Depending on which side of $x'$
    the point $y$ is, $A$ is obtuse or acute. In any case, $A\geq B$. It is
    a standard fact that the side opposite the bigger angle in a triangle
    is longer. Hence $d_\HH(x,y) \geq d_\HH(x',y)$. \qedhere

  \end{proof}

  In this section we will often use the following elementary estimates for
  hyperbolic trigonometric functions. The proofs are omitted.
  \begin{lemma} \label{Lem:Estimates}
    \mbox{}
    \begin{rmenumerate}
      \item \label{Est1} If $0 \leq x \leq 1$ or $0 \leq \sinh(x) \leq 1$, we have $\sinh(x) \leq 2x$.

      \item \label{Est2} For all  $x \geq 0$,   $\frac 1 2 e^x\leq\cosh(x)\leq e^x$ and $x\leq \sinh(x)\leq \frac 1 2 e^x$.

      \item \label{Est3} For all $x \geq 1$, we have $\sinh(x) \geq \frac 1 4 e^x$.

      \item \label{Est4} For all $x \geq 1$, we have \[ \log (2)\leq
      \arcsinh(x)-\log (x) \leq \log (3)\] and \[ 0\leq \arccosh(x)
        -\log(x)\leq \log(2). \]
    \end{rmenumerate}
  \end{lemma}
  
  Now consider $X \in \T(\s)$ and a geodesic lamination $\lambda$ on
  $X$.  If $\alpha$ crosses a leaf $\omega$ of $\lambda$, define
  $V_X(\omega, \alpha)$ to be a shortest arc with endpoints on
  $\omega$ that, together with an arc $H_X(\omega,\alpha)$ of
  $\omega$, form a curve homotopic to $\alpha$.  Thus
  $V_X(\omega,\alpha)$ and $H_X(\omega,\alpha)$ meet orthogonally and
  $\alpha$ passes through the midpoints of both of these arcs (see \figref{HV}).  If
  $\alpha$ is a leaf of $\lambda$, then we set
  $H_X(\omega,\alpha) = \alpha$ and let $V_X(\omega,\alpha)$ be the
  empty set.

  Define $h_X$ and $v_X$ to be the lengths of $H_X(\omega,\alpha)$ and
  $V_X(\omega,\alpha)$ respectively.  By considering the right
  triangles formed by these curves and $\alpha$ (which have hypotenuse
  along $\alpha$), it is immediate that
  \begin{equation}
    \label{Eqn:basic-h-v}
    \max(h_X,v_X) \leq \ell_\alpha(X) \leq h_X + v_X.
  \end{equation}

  The quantities $h_X$ and $v_X$ can be computed in the universal
  cover $\tx \isom \HH$ as follows. Let $\tomega$ and $\talpha$ be
  intersecting lifts of $\omega$ and $\alpha$ to \HH. Let $\phi$ be
  the hyperbolic isometry with axis $\talpha$ and translation length
  $\ell_\alpha(X)$. Set $\tomega'=\phi(\tomega)$ and let $\psi$ be the
  hyperbolic isometry taking $\tomega$ to $\tomega'$ with axis
  perpendicular to the two geodesics. Since $\phi$ and $\psi$ both
  take $\omega$ to $\omega'$, their composition $\psi^{-1}\phi$ is a
  hyperbolic isometry with axis $\tomega$. The quantity $v_X$ is the
  translation length of $\psi$ and $h_X$ is the translation length of
  $\psi^{-1}\phi$.  For the latter, this means that
  $h_X = d_\HH(\psi(q),\phi(q))$ for any $q \in \tomega$.

  In the following, let $X_t = \G(t)$ be a geodesic segment and let
  $\lambda = \lambda_\G$. Let $\alpha$ be a curve that interacts with
  $\lambda$. We will refer to $V_{X_t}(\omega, \alpha)$ and
  $H_{X_t}(\omega, \alpha)$ as the \emph{vertical} and
  \emph{horizontal} components of $\alpha$ at $X_t$. We are interested
  in the lengths $h_t=h_{X_t}$ and $v_t=v_{X_t}$ of the horizontal and
  vertical components of $\alpha$ as functions of $t$. We will show
  that $v_t$ decreases super-exponentially, while $h_t$ grows
  exponentially. These statements are trivial if $\alpha$ is a leaf of
  $\lambda$, so we will always assume that $\alpha$ crosses a leaf
  $\omega$ of $\lambda$.

  \begin{lemma} \label{Lem:h-growth}

    Suppose $\alpha$ crosses a leaf $\omega$ of $\lambda$. For any $t \ge
    s$, \[ h_t \geq e^{t-s} \left( h_s - v_s\right).\]

  \end{lemma}

  \begin{proof}

    In \HH, choose a lift $\talpha$ of the geodesic representative of
    $\alpha$ on $X_s$ and a lift $\tomega$ of $\omega$ that crosses
    $\talpha$. Let $\tomega'=\phi_s(\tomega)$ where $\phi_s$ is the
    hyperbolic isometry with axis $\talpha$ and translation length
    $\ell_s(\alpha)$. Let $\psi_s$ be the hyperbolic isometry taking
    $\tomega$ to $\tomega'$ with axis perpendicular to the two
    geodesics. Let $p\in\tomega$ be the point lying on the axis of
    $\psi_s$. By definition,
    \[ v_s=d_\HH(p,\psi_s(p)) \quad \text{and} \quad
    h_s=d_\HH(\psi_s(p),\phi_s(p)).\]

    The configuration of points and geodesics in $\HH$ constructed above
    is depicted in \figref{HV}; it may be helpful to refer to this
    figure in the calculations that follow.  Note that for brevity the
    subscript $s$ is omitted from the labels involving $\psi, \phi$ in
    the figure.

    Let $f:X_s \to X_t$ be an optimal map and let $\tf: \HH \to \HH$ be
    a lift of $f$. Since $f$ is an $e^{t-s}$--Lipschitz map such that
    distances along leaves of $\lambda$ are stretched by a factor of
    exactly $e^{t-s}$, the images $\tf(\tomega)$ and $\tf(\tomega')$ are
    geodesics and
    \[
    d_\HH \left( \tf\psi_s(p),\tf\phi_s(p) \right)=e^{t-s}h_s \quad
    \text{and} \quad d_\HH \left( \tf(p),\tf\psi_s(p) \right) \le
    e^{t-s}v_s.
    \]
    Let $\psi_t$ be the hyperbolic isometry taking $\tf(\tomega)$ to
    $\tf(\tomega')$ with axis their common perpendicular. Let $\phi_t$ be
    the hyperbolic isometry corresponding to $f \alpha$ taking
    $\tf(\tomega)$ to $\tf(\tomega')$. Note that
    $\phi_t \tf = \tf \phi_s$, since $\tf$ is a lift of $f$. But $\psi_s$
    and $\psi_t$ do not necessarily correspond to a conjugacy class of
    $\pi_1(S)$, so $\tf$ need not conjugate $\psi_s$ to $\psi_t$.

    By definition,
    \[
    h_t=d_\HH \left( \psi_t\tf(p),\phi_t\tf(p) \right)
    = d_\HH \left( \psi_t\tf(p),\tf\phi_s(p) \right).
    \]
    By \lemref{Quad}\eqref{Quad3},
    \[
    d_\HH \left( \tf\psi_s (p), \psi_t\tf(p) \right) \leq d_\HH \left(
    \tf \psi_s (p), \tf(p) \right).
    \]
    Using the triangle inequality and the above equations, we obtain the
    conclusion.
    \begin{align*}
      h_t
      & \geq d_\HH \left( \phi_t\tf (p), \tf\psi_s (p) \right) -
        d_\HH \left( \tf\psi_s(p), \psi_t\tf(p) \right) \\
      & \geq d_\HH \left( \phi_t\tf (p), \tf\psi_s (p) \right) -
        d_\HH \left( \tf\psi_s(p), \tf(p)\right) \\
      & = d_\HH \left( \tf \phi_s(p), \tf\psi_s (p) \right) -
        d_\HH \left( \tf\psi_s(p), \tf(p)\right) \\
      & \geq e^{t-s}h_s - e^{t-s}v_s \qedhere
    \end{align*}

  \end{proof}

  \begin{figure}
  \begin{center}
  \includegraphics{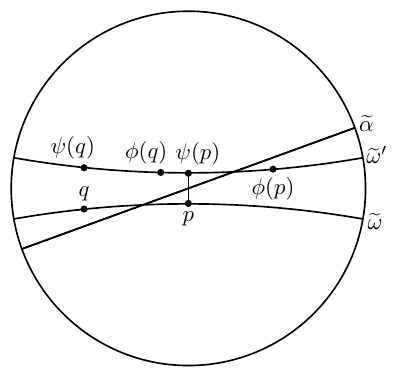}
  \end{center}
  \caption{Estimating $h_t$.}
  \label{Fig:HV}
  \end{figure}

  \begin{lemma} \label{Lem:v-decay}

    Suppose $\alpha$ crosses a leaf $\omega$ of $\lambda$. There exists
    $\ep_v > 0 $ such that if $v_a \leq \ep_v$, then for all $t \geq a$,
    we have:
    \[
    v_t \leq e^{-A e^{t-a}}, \quad \text{where} \quad A >0 \text { and }
    A\eadd \log\frac{1}{v_a}
    \]
    and where  the additive error is at most $\log 4+1$.
  \end{lemma}

  \begin{proof}

    We refer to \figref{HV2}. As before, choose a lift $\talpha$ to \HH of
    the geodesic representative of $\alpha$ on $X_a$ and a lift $\tomega$
    of $\omega$ that crosses $\talpha$. Let $\tomega'=\phi(\tomega)$ where
    $\phi$ is the hyperbolic isometry with axis $\talpha$ and translation
    length $\ell(\alpha)$. Let $p\in\tomega$ and $p'\in\tomega'$ be the
    endpoints of the common perpendicular between $\tomega$ and $\tomega'$,
    so $v_a=d_\HH(p,p')$.

    We assume $v_a < \frac 1 2$. Let $[x,y]\subset \tomega$ and
    $[x',y']\subset \tomega'$ be segments of the same length with midpoints
    $p$ and $p'$, such that $[x,x']$ and
    $[y,y']$ have length $1$ and are disjoint from $[p,p']$. By
    \eqref{Quad1} from \lemref{Quad},
    \[
    d_\HH(x,p)= \arccosh \frac{\sinh{1/2}}{\sinh{v_a/2}}
    \]
    We can apply \lemref{Estimates}\ref{Est1} and \ref{Est4}, which give
    \begin{align}\label{col} 
      \left| d_\HH(x,y)-2\log\frac {1}{v_a} \right|\leq 2\log 4.
    \end{align}
    In particular, $v_a$ is small if and only if $d_\HH(x,y)$ is large. Let
     $\ep_v$ be small enough so that $d_\HH(x,y) \geq 4$.

    Let $f:X_a\to X_t$ be an optimal map and $\tf:\HH\to\HH$ a lift of
    $f$.  Let $r\in\tf(\tomega)$ and $r'\in\tf(\tomega')$ be the endpoints
    of the common perpendicular between $\tf(\tomega)$ and
    $\tf(\tomega')$, so $v_t=d_\HH(r,r')$. Without a loss of generality,
    assume that $r$ is farther away from $\tf(x)$ than $\tf(y)$. This
    means
    \begin{align} \label{half}
      d_\HH \left( \tf(x),r \right) \geq \frac{1}{2}d_\HH \left(
      \tf(x),\tf(y) \right).
    \end{align}
    We also have
    \begin{align} \label{Bounds}
      d_\HH \left( \tf (x), \tf(y) \right) = e^{t-a}d_\HH(x,y)
      \quad \text{and} \quad
      d_\HH \left(\tf (x), \tf(x')\right) \leq e^{t-a}.
    \end{align}
    By \eqref{Quad2} from \lemref{Quad},
    \[
    \sinh {d_\HH(r,r')} \cosh d_\HH \left( \tf(x), r \right)
    \leq \sinh d_\HH \left( \tf(x),\tf(x') \right).
    \]

    Incorporating  \eqref{half} and \eqref{Bounds} to the above
    inequality yields
    \[
    \sinh {d_\HH(r,r')}
    \leq \frac{\sinh e^{t-a}}{\cosh \left( \frac{1}{2}e^{t-a} d_\HH(x,y)
    \right)}
    \]
    Now use \lemref{Estimates}\ref{Est2}  to obtain
    \[
    d_\HH(r,r')
    \leq e^{-e^{t-a} \left(\frac{d_\HH(x,y)}{2}-1 \right)}.
    \]
    Setting
    $A = \frac{d_\HH(x,y)}{2}-1$ and applying \eqref{col} 
    we have that $A>0$ and $|A-\log\frac 1 {v_a}|\leq  \log 4+1$.  
    This finishes the proof. \qedhere
  \end{proof}

  \begin{figure}
  \begin{center}
  \includegraphics{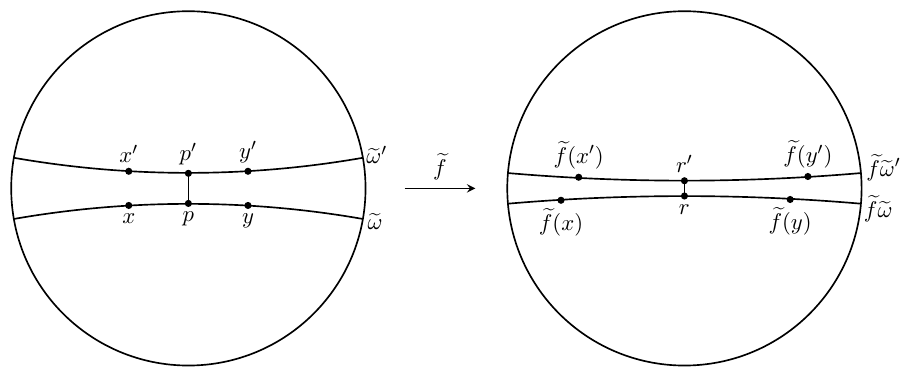}
  \end{center}
  \caption{Bounding $v_t$ from above.}
  \label{Fig:HV2}
  \end{figure}

  \begin{lemma} \label{Lem:StayShort}

    Suppose $\alpha$ crosses a leaf $\omega$ of $\lambda$.  Let
    $\ep_v$ be the constant from \lemref{v-decay}.  If $[a,b]$ is an
    interval of times with
    $\ell_\alpha(a) = \ell_\alpha(b) = \ep < \ep_v$, then
    $\ell_\alpha(t) \lmul \ep$ for all $t \in [a,b]$ with the
    multiplicative error at most $12e$.

  \end{lemma}

  \begin{proof}
    
    Let $t\in[a,b]$. 
   
    Suppose first that $v_t>\frac 1 2 h_t$. Here one can replace $\frac 1 2$
      by any other number in $(0,1)$. Then $\ell_\alpha(t)\leq 3 v_t$. 
    By \lemref{v-decay}, 
    \[
    v_t\leq e^{-A}\emul  v_\alpha
    \]
    where the multiplicative error is at most $4e$, and since $v_a$ is
    bounded above by $\ep$, we have $\ell_\alpha(t)\lmul \ep$, with error
    at most $12e$.
    
    Now suppose $v_t\leq \frac 1 2 h_t$. Then by \lemref{h-growth} 
    \[h_b\geq e^{b-t}(h_t-v_t)\geq \frac 1 2 e^{b-t} h_t\geq \frac 1 2  h_t
    \]
    Hence
     \[ \ell_\alpha(t)\leq h_t+v_t\leq \frac{3}{2}h_t\leq 3 h_b\leq 3\ep.
     \]
    This finishes the proof. \qedhere

  \end{proof}

  For our purposes, an important consequence of \lemref{StayShort} is
  that if the curve is short enough at the endpoints of an interval,
  then its length will be below $\ep_M$ throughout that interval.
  Specifically, fix $\ep_0>0$ so that
  \[
    \ep_0 < \min \left ( \frac{\ep_M}{12e}, \ep_v \right ),
    \label{Eqn:ep0def}
  \]
  where $\ep_M$ is the Margulis number chosen in \subsecref{teich} and
  $\ep_v$ is the constant from \lemref{v-decay}.  Then as an immediate
  corollary of \lemref{StayShort} we have:

  \begin{corollary}
    \label{Cor:always-short}
    If $[a,b]$ is an interval such that $\ell_\alpha(a) = \ell_\alpha(b) = \ep_0$, then
    $\ell_\alpha(t) < \ep_M$ for all $t \in [a,b]$. \hfill\qedsymbol
  \end{corollary}

  Next we study the relationship between the relative twisting
  $d_\alpha(X,\lambda)$ and the length of $V_X(\omega,\alpha)$ 
  and $\ell_\alpha(X)$.

  \begin{lemma} \label{Lem:l-p}

    Suppose $\alpha$ crosses a leaf $\omega$ of $\lambda$. Fix $X=X_t$ and
    let $\ell = \ell_\alpha(X)$ and $v$ be the length of
    $V_X(\omega,\alpha)$. Then the following statements hold.
    \begin{rmenumerate}
      \item If $\ell \leq\ep_M$, then
      \[
        \ell \, d_\alpha(X, \lambda) \eadd 2 \log \frac {\ell}{v}.
      \]
      \item If $\displaystyle d_\alpha(X,\lambda) \geq \frac{4\ep_M}{\ell}+2$, then
      \[ v \leq 2e^{-\frac{\ell}{4} d_\alpha(X,\lambda)}.\]
    \end{rmenumerate}
  
  \end{lemma}

  \begin{proof}

    The reader may find it helpful to look at \figref{HV} for this proof.

    Let $B$ be the angle between $\talpha$ and $\tomega$. Let $L$ be
    the length of the projection of $\tomega$ to $\talpha$. Recall
    that $d_\alpha(X,\lambda) \eadd \frac{L}{\ell}$ with additive
    error at most $2$.  Since $\ell \leq \ep_M$, this implies
    \begin{equation}\label{dalpha_Ll}
      \ell \, d_\alpha(X,\lambda) \eadd L
    \end{equation}
    with additive error at
    most $2 \ep_m$.  By hyperbolic geometry, $L$ satisfies
    \begin{align} \label{cosh-proj}
      1=\cosh\frac{L}{2}\cdot \sin{B}.
    \end{align}
    To find $\sin B$, denote by $\phi$ the hyperbolic isometry with axis
    $\talpha$ and with translation length $\ell$. Let
    $\tomega' = \phi(\tomega)$. Denote by $x$ the intersection of
    $\talpha$ and $\tomega$ and set $x'=\phi(x)$. Let $p \in \tomega$
    and $p' \in \tomega'$ be the points on the common perpendicular
    between $\tomega$ and $\tomega'$.  That is, $p' = \psi(p)$ where
    $\psi$ is the translation along an axis perpendicular to $\tomega$
    such that $\psi(\tomega) = \tomega'$.  By construction,
    $d_\HH(x,x') = \ell$ and $d_\HH(p,p') = v$.  Then the
    intersection point of $[p,p']$ and $[x,x']$ is the midpoint of both.
    Thus $\sin B$ can be found from
    \begin{equation} \label{angle}
    \sin B \sinh \frac \ell 2=\sinh \frac v 2.
    \end{equation}
    Combining \eqref{cosh-proj} and \eqref{angle} we obtain
    \begin{align} \label{proj}
      L=2\arccosh \frac{\sinh \ell/2 }{\sinh v/2}.
    \end{align}

    When $\ell \leq \ep_M < 1$, we can apply \lemref{Estimates}\ref{Est1}
    and \ref{Est4} to simplify \eqref{proj}, obtaining
    \[
      L  \eadd 2 \log \frac{\ell}{v}
    \]
    which in combination with \eqref{dalpha_Ll} gives
    \[
    \ell \, d_\alpha(X,\lambda) \eadd 2 \log \frac{\ell}{v},
    \]
    with the additive error in the latter estimate at most $2\log 4+2\ep_M$.
    
    Now we consider the upper bound on $v$ under the assumption
    $d_\alpha(X,\lambda) \geq \frac{4\ep_M}{\ell}+2$.  By
    \eqref{dalpha_Ll} we have $L\geq\ell\,d_\alpha(X,\lambda) -2\ep_M$
    and incorporating this with \eqref{proj} gives
    \[
      \frac{\sinh{\ell/2}}{\sinh{v/2}}\geq\cosh{(\frac{\ell}{2} d_\alpha(X,\lambda)-\ep_M)}.
    \]
    Therefore,
    \[
     \frac{v}{2}\leq \sinh \frac{v}{2} \leq \frac{\sinh \ell/2}{\cosh{(\frac{\ell}{2}d_\alpha(X,\lambda) - \ep_M)} }
     \leq \frac{e^\frac{\ell}{2}}{ e^{\frac{\ell}{2}d_\alpha(X,\lambda)-\ep_M}} = e^{\frac{\ell}{2} + \ep_M - \frac{\ell}{2} d_\alpha(X,\lambda)}
   \]
   where the third inequality above uses the fact that
   $\frac{\ell}{2}d_\alpha(X,\lambda)-\ep_M > 0$ to apply
   \lemref{Estimates}\ref{Est2}. Furthermore, our assumed lower bound on
   $d_\alpha(X,\lambda)$ gives
   \[ \frac{\ell}{2} + \ep_M - \frac{\ell}{2} d_\alpha(X,\lambda) \leq
     -\frac{\ell}{4} d_\alpha(X,\lambda) \] and substituting this into the
     previous bound on $\frac{v}{2}$ we find
   \[v \leq 2e^{-\frac{\ell}{4} d_\alpha(X,\lambda)} \] which completes the
   proof. \qedhere

  \end{proof}

  The following lemma implies that the length of the vertical
  component does not decrease too quickly along a geodesic ray if the
  curve starts out being approximately vertical and remains short
  throughout the ray.

  \begin{lemma}\label{Lem:v_lower}

    Suppose $\alpha$ crosses a leaf $\omega$ of $\lambda$. There
    exists $A > 0$ with $A \eadd \log\frac{1}{\ep_0}$ such that the
    following holds.  If $\ell_\alpha(a) = \ep_0$ and
    $v_a\geq \frac{\ep_0}{4}$, and if $\ell_\alpha(t) < \ep_M$ for all
    $t \geq a$, then we have
    \[  v_t \gmul e^{-Ae^{t-a}}.\]
    
  \end{lemma}
  
  \begin{proof}

    Let $\beta$ be a shortest curve at time $a$ that intersects $\alpha$.
    Recall that $\ep_M$ was chosen so that $\ell_\alpha(a) < \ep_M$ implies
    $\I(\alpha,\beta) \in \{1,2\}$. We will give the proof in the case
    $\I(\alpha,\beta) = 1$, with the other case being essentially the same.
    Since $\alpha$ is short for all $t > a$, the part of $\beta$ in a
    collar neighborhood of $\alpha$ has length that can be estimated in
    terms of the length of $\alpha$ and the relative twisting of $X_t$ and
    $\beta$ (see \cite[Lemma~7.3]{choi-rafi-series:minima-and-teich}),
    giving a lower bound for the length of $\beta$ itself:
    \[
    \ell_\beta(t) \gadd d_\alpha(X_t, \beta)
    \ell_\alpha(t)+2\log\frac{1}{\ell_\alpha(t)}.
    \]
    On the other hand, since $\ell_\alpha(a) = \ep_0$ and $v_a\geq \frac{\ep_0}{4}$,
     applying
    \lemref{l-p} to $X_a$ tells us that $d_\alpha(X_a,\lambda)$ is
    bounded. Hence $|d_\alpha(X_t,\lambda)-d_\alpha(X_t,\beta)|\ladd 1$
    which means that we can write
    \[
    \ell_\beta(t) \gadd d_\alpha(X_t, \lambda) \ell_\alpha(t)+2\log\frac{1}{\ell_\alpha(t)}.
    \]
    The length of $\beta$ cannot grow faster than the length of $\lambda$,
    therefore
    \begin{align*}
      d_\alpha(X_t,
      \lambda)\ell_\alpha(t)+2\log\frac{1}{\ell_\alpha(t)}\ladd
      e^{t-a}\ell_\beta(a).
    \end{align*}
    Applying \lemref{l-p} again, now to $X_t$, we have
    \[
      2\log \frac {\ell_\alpha(t)}{ v_t}+2\log\frac{1}{\ell_\alpha(t)}\eadd
      d_\alpha(X_t,
      \lambda)\ell_\alpha(t)+2\log\frac{1}{\ell_\alpha(t)}\ladd
      e^{t-a}\ell_\beta(a)
    \]
    which implies 
    \[ v_t\gmul e^{-\frac{1}{2}e^{t-a}\ell_\beta(a)}.\] 
    The claim now follows from the fact that $\ell_\beta(a) \eadd
    2\log\frac{1}{\ep_0}$. \qedhere

  \end{proof}

  \begin{theorem} \label{Thm:Twist}

  Suppose $\alpha$ crosses a leaf $\omega$ of $\lambda$. Let $[a,b]$ be
  an interval such that $\ell_\alpha(a) = \ell_\alpha(b) =
  \ep_0$. Then
  \[ d_\alpha(X_a, X_b) \, \emuladd \, e^{b-a}.\]
  The length of $\alpha$ is minimum in the interval $[a,b]$ at a time
    $t_\alpha \in [a,b]$ satisfying 
    \begin{equation}
      \label{main-talpha-estimate}
      t_\alpha-a \, \eadd \, \Log (b-t_\alpha),
    \end{equation}
    and the minimum length is
    $\ell_{\alpha}(t_\alpha) \emul e^{-(b-t_\alpha)}$.

    Furthermore, if $(b-a)$ is sufficiently large, then
    $\log(b-t_\alpha) > 1$ and so \eqref{main-talpha-estimate} also
    holds with $\Log$ replaced by $\log$.
  \end{theorem}

  In some of the preceding lemmas we indicated the dependence of
  multiplicative and additive errors on $\ep_0$.  However, since
  $\ep_0$ is a fixed constant, we will ignore such dependence in most
  cases from now on.
  
  \begin{proof}
    We split the proof into two cases, depending on whether the
    interval $[a,b]$ is ``short'' or ``long''.  More precisely we
    consider the cases $(b-a) \leq Q$ and $(b-a) > Q$ for some
    positive real $Q$, the \emph{threshold}.  The implicit constants
    in the approximate comparisons we derive in each case will depend
    on $Q$, and at various points in the long-interval case it will be
    necessary to assume $Q$ is sufficiently large (i.e.~greater than
    some universal constant).  At the end we can fix any $Q$ large
    enough to satisfy all of those assumptions.

    First we consider the short-interval case, $(b-a) \leq Q$.  Here,
    all of the claims of approximate equality in the theorem will hold
    because all of the quantities in question are bounded.  Since
    $t_\alpha \in [a,b]$, both $(t_\alpha-a)$ and $\Log(b-t_\alpha)$
    are nonnegative and bounded above,
    i.e.~$t_\alpha-a \eadd \Log (b-t_\alpha) \eadd 0$.

    The surface $X_{t_\alpha}$ admits maps from $X_a$ and to $X_b$
    with bounded Lipschitz constant (at most $e^Q$).  Since $\alpha$
    has length $\ep_0$ on both $X_a$ and $X_b$, this shows that
    $\ell_\alpha(t)$ is bounded above and below by positive
    constants depending on $Q$ for all $t \in [a,b]$, i.e.~that
    $\ell_\alpha(t_\alpha) \emul 1$.  Since
    $1 \geq e^{-(b-t_\alpha)} \geq e^{-(b-a)} \geq e^{-Q}$, we also
    have $e^{-(b-t_\alpha)} \emul 1$, and thus
    $\ell_{\alpha}(t_\alpha) \emul e^{-(b-t_\alpha)}$.

    To obtain the bound on $d_\alpha(X_a, X_b)$ in the short-interval
    case, we recall from \cite{minsky:PR} that the rate at which
    $d_\alpha(X_a, \param)$ can change is bounded with the bound
    depending on the length of $\alpha$. As noted above we have upper
    and lower bounds for the length of $\alpha$ along the geodesic
    between $X_a$ to $X_b$, hence $d_\alpha(X_a, X_b) \eadd 0$.  We
    are assuming an upper bound on $(b-a)$, so this implies
    $d_\alpha(X_a, X_b) \eadd e^{b-a}$.

    Now we turn to the long-interval case, $(b-a)>Q$.  First we
    require $Q > \log(2)$, so that $e^{b-a} > 2$.  It follows that
    $h_a - v_a \leq \ep_0/2$; to see this, assume for contradiction
    that $h_a - v_a > \ep_0/2$.  Then \lemref{h-growth} gives
    \[ h_b \geq \frac{\ep_0}{2} e^{b-a} > \ep_0, \] while
    \eqref{Eqn:basic-h-v} gives
    \[ h_b \leq \ell_\alpha(b) = \ep_0, \]
    a contradiction.

    Now, since $h_a - v_a \leq \ep_0/2$ and
    $h_a + v_a \geq \ell_a(\alpha) = \ep_0$, we find
    $\frac{\ep_0}{4} \leq v_a \leq \ep_0$, i.e.~at time $t=a$ the
    curve is nearly perpendicular to $\lambda$, and
    $\ell_\alpha(a)/v_a \emul 1$.  Applying \lemref{l-p} we obtain
    \[ \ep_0\, d_\alpha(X_a,\lambda) = \ell_\alpha(a)
      d_\alpha(X_a,\lambda) \eadd 0. \] Dividing by
    $\ep_0$ we obtain $d_\alpha(X_a,\lambda) \eadd 0$.
    
    By \corref{always-short} we have $\ell_\alpha(t) < \ep_M$ for all
    $t \in [a,b]$.  Using this, the bounds of \lemref{v-decay} and
    \lemref{v_lower} show that there are $A,B > 0$ such that
    \begin{align}
      \label{vert0}
      e^{-Be^{t-a}} \lmul v_t\le e^{-Ae^{t-a}}, \text{ for all } t
      \in [a,b].
    \end{align}
    (And in fact those lemmas show $A,B \eadd\log\frac{1}{\ep_0}.$)
    Taking the logarithm of \eqref{vert0} gives
    \begin{align} \label{vert}
      \log\frac{1}{v_t} \emuladd e^{t-a}, \text{ for all } t
      \in [a,b],
    \end{align}
    where the additive error comes from the multiplicative error in
    \eqref{vert0}, and the multiplicative error from the constants
    $A,B$.

    We claim that for $Q$ sufficiently large there exists
    $s \in [a,b]$ such that $h_s=2v_s$.  Indeed, if $h_s<2v_s$ for all
    $s\in [a,b]$, then, since $h_b+v_b\geq \ep_0$, we have
    $\frac 1 3 \ep_0 \le v_b \le \ep_0$.  Using \eqref{vert} with
    $t=b$ this gives an upper bound on $e^{b-a}$, which is a
    contradiction if $Q$ is large enough.  On the other hand, if
    $h_s>2v_s$ for all $s\in [a,b]$, then \lemref{h-growth} implies
    that $h_b$ is large if $(b-a)$ is sufficiently large.
    Specifically, by taking $Q$ larger than a universal constant we
    would have $h_b > \ep_0$, contradicting that
    $\ep_0 = \ell_\alpha(b) \geq h_b$.  Thus by requiring $Q$ to be
    large enough so that both of these arguments apply, we have
    $h_s = 2 v_s$ for some $s \in [a,b]$.  For the rest of the proof,
    let $s$ denote any such point in the interval.

    Since $v_s$ and $h_s$ are comparable, it follows from
    \eqref{Eqn:basic-h-v} that $v_s \emul \ell_\alpha(s)$. Since
    $e^{b-s}$ is the Lipschitz constant from $X_s$ to $X_b$, we have
    $\ell_\alpha(s) e^{b-s} \ge \ell_\alpha(b) = \ep_0$.  In
    particular $\ep_0 \lmul v_s e^{b-s}$. On the other hand,
    \lemref{h-growth} gives
    \[ h_b \geq \frac 1 2 h_s e^{b-s}=v_s e^{b-s}.\] Thus
    $v_s e^{b-s} \le h_b \le \ell_\alpha(b) = \ep_0$. All together, we
    obtain
    \begin{align} \label{hor}
      v_s \emul e^{-(b-s)}. 
    \end{align}
    Now using \eqref{vert} with $t=s$ and \eqref{hor} together we find
    \begin{equation}
      \label{vert-and-horiz-conseq}
      e^{s-a} \emuladd (b-s).
    \end{equation}
    From this, it follows that
    \begin{equation}
      \label{Log-eadd}
      \Log(b-s) \eadd (s-a).
    \end{equation}
    Indeed, if $\log(b-s) \geq 1$ then $\log(b-s)=\Log(b-s)$ and
    \eqref{Log-eadd} is the result of taking the logarithm of
    \eqref{vert-and-horiz-conseq}.  Otherwise $\log(b-s) < 1$, in
    which case $\Log(b-s) = 1$ and \eqref{vert-and-horiz-conseq} gives
    a uniform upper bound on $(s-a)$, so \eqref{Log-eadd} holds simply
    because both sides are nonnegative and bounded.

    Finally, since $v_t$ is essentially decreasing
    double-exponentially, $h_t$ is increasing exponentially and
    $\ell_\alpha(t)\geq \max \{v_t, h_t\}$, it follows that
    $t_\alpha \eadd s$.  This gives us the order of the minimal length
    of $\alpha$, which is approximated by $\ell_\alpha(s) \emul v_s$.
    Also, using $t_\alpha \eadd s$, we find that equation
    \eqref{Log-eadd} also holds if we replace $s$ by $t_\alpha$, which
    gives \eqref{main-talpha-estimate}.

    To complete the long-interval case we estimate
    $d_\alpha(X_a,X_b)$. By \lemref{l-p} and \eqref{vert},
    \begin{align*}
    \ep_0\,d_\alpha(X_b,\lambda)
    & \eadd 2 \log\frac{\ep_0}{v_b} \emul e^{b-a},
    \end{align*}
    and we can absorb the additive error in the multiplicative error since
    the expression on the right is bounded away from 0. Since $v_t$
    decreases double-exponentially, $v_b$ is very small compared to $\ep_0$
    for $(b-a)$ large, so $\ep_0\,d_\alpha(X_b,\lambda)$ is bounded away
    from $0$. Dividing by $\ep_0$ (and absorbing this into the
    multiplicative error as well) we find $d_\alpha(X_b,\lambda) \emul
    e^{b-a}$.  Since $d_\alpha(X_a,\lambda)\eadd 0$, this is the
    desired estimate.

    Fixing a value for the threshold $Q$ large enough to satisfy all
    of the conditions derived in the long-interval analysis above, the
    estimates in both parts of the proof become uniform (i.e.~no
    longer depend on an additional parameter).

    It only remains to prove the final claim from the statement of the
    theorem.  For this, we show $(b-t_\alpha)$ can be made larger than
    a given constant just by assuming that $(b-a)$ is sufficiently
    large.  Suppose for contradiction that $(b-t_\alpha)$, and hence
    also $(b-s)$, can be bounded with $(b-a)$ arbitrarily large.  Then
    $e^{s-a} \emul e^{b-a}$ is large while $(b-s)$ is bounded,
    contradicting \eqref{vert-and-horiz-conseq}.\qedhere
  \end{proof}
    
  Note that Theorem \ref{Thm:Twist} highlights an interesting contrast
  between the behavior of Thurston metric geodesics and that of
  Teichm\"uller geodesics: Along a Teichm\"uller geodesic, a curve
  $\alpha$ achieves its minimum length near the midpoint of the
  interval in which $\alpha$ is short (see \cite[Section~3]{rafi:HT}),
  and this minimum is on the order of $d_\alpha(X,Y)^{-1}$.  However,
  for a Thurston metric geodesic, the minimum length occurs much
  closer to the start of the interval (asssuming the interval is
  sufficiently long) since $(t_\alpha - a)$ is only on the order of
  $\log(b - t_\alpha)$.  In addition, the minimum length on the
  Thurston geodesic is larger than in the Teichm\"uller case, though
  only by a logarithmic factor.

  To exhibit this difference, Figure \ref{fig:comparison} shows a
  Teichm\"uller geodesic segment and a stretch path segment (for
  lamination $\beta^+$) joining the same pair of points in the upper
  half plane model of $\T(\torus)$.  Here $\beta$ is a simple closed
  curve.  In this model, the imaginary part of a point $z \in \HH$ is
  approximately $\pi/\ell_\alpha(z)$, where $\alpha$ is a curve which
  has approximately the same length at both endpoints but which
  becomes short somewhere along each path.  Thus the expected (and
  observed) behavior of the Thurston geodesic is that its maximum
  height is lower than that of the Teichm\"uller geodesic, but that
  this maximum height occurs closer to the starting point for the
  Thurston geodesic. Further properties of Thurston geodesics in the
  punctured torus case are explored in the next section.

  \begin{figure}
    \begin{center}
      \includegraphics[width=0.75\textwidth]{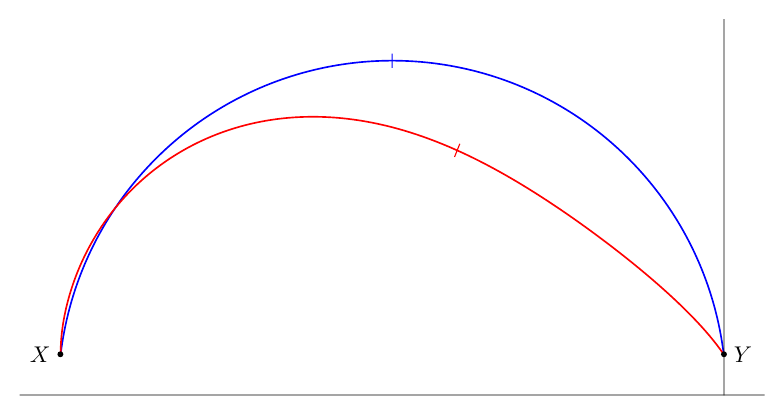}
    \end{center}
    \caption{A Teichm\"uller geodesic (blue) and a stretch path (red) in
      the Teichm\"uller space $\T(\torus) \simeq \HH$ of the punctured
      torus.  Both geodesic segments start at $X = -16.302 + i$ and
      end at $Y = i$, and each has its midpoint marked.}
    \label{fig:comparison}
    \end{figure}

  Continuing toward the proof of \thmref{Short}, we show:

  \begin{lemma} \label{Lem:l-growth}

    Suppose $\alpha$ crosses a leaf $\omega$ of $\lambda$. There
    exists a constant $C>0$ such that if $\ell_\alpha(s) \geq \ep_0$
    and $d_\alpha(X_s, \lambda) \geq C$, then
    $\ell_\alpha(t) \emul e^{t-s} \ell_\alpha(s)$ for all $t \geq s$.

  \end{lemma}

  \begin{proof}

    By \lemref{l-p}, if $d_\alpha(X_s,\lambda) > \frac{4 \ep_M}{\ell_\alpha(s)} + 2$, then
    \[ v_s \leq 2 e^{-\frac{\ell_\alpha(s)}{4} d_\alpha(X_s,\lambda)}. \] Since
    $\ell_\alpha(s) \geq \ep_0$ in this case, there is a universal constant $C$ so
    that this estimate applies when $d_\alpha(X_s,\lambda) > C$.  Furthermore, we can choose $C$ so that the inequality above gives
    \[
      v_s \leq \frac 1 3 \ell_\alpha(s)
    \]
    and so $h_s\geq \frac 2 3 \ell_\alpha(s) $ and $2v_s\leq h_s$.
    Incorporating \lemref{h-growth}, we have that for all $t > s$,
    \[ \ell_\alpha(t) \geq h_t \geq \frac 1 2 e^{t-s} h_s \geq \frac 1 3
      e^{t-s} \ell_\alpha(s).\] On the other hand,
    $\ell_\alpha(t) \leq e^{t-s} \ell_\alpha(s)$. This finishes the
    proof.  \qedhere

  \end{proof}

  \begin{lemma} \label{Lem:no-twist}

  Suppose $\alpha$ interacts with $\lambda$. If  $\ell_\alpha(t) \ge
  \ep_0$ for all $t \in [a,b]$, then $d_\alpha(X_a,X_b) \eadd 0$.

  \end{lemma}

  \begin{proof}

    We first show that for any $s \leq t$, if $\ell_\alpha(t) \emul
    e^{t-s}\ell_\alpha(s)$, then $d_\alpha(X_s,X_t) \eadd 0$. Let $\beta$
    be the shortest curve at $X_s$ that intersects $\alpha$. At time $t$,
    the length of $\beta$ satisfies
    \[\ell_\beta(t)\geq \ell_\alpha(t)d_\alpha(\beta,X_t)-D\ell_\alpha(t),\]
    where $D\geq 0$ is universal. Also, $d_\alpha(\beta,X_s)$
    is bounded by the choice of $\beta$.
    Hence we can write
    \[ \frac{\ell_\beta(t)}{\ell_\alpha(t)}\gadd d_\alpha(X_s,X_t).\]
    Therefore, since $\ell_\beta(t)\leq e^{t-s}\ell_\beta(s)$ and
    $\ell_\alpha(t)\emul e^{t-s}\ell_\alpha(s)$, we have
    \[  d_\alpha(X_s,X_t) \ladd \frac{\ell_\beta(s)}{\ell_\alpha(s)}.\]
  
    Let $\ep_B$ be the Bers constant. If $\ell_\alpha(s)\geq \ep_B$,  then
    $\ell_\beta(t)\leq \ep_B$ and so
    $\frac{\ell_\beta(s)}{\ell_\alpha(s)}<1$. If $\ell_\alpha(s)\leq
    \ep_B$, then $\ell_\beta(s)$ is up to a bounded multiplicative error
    the width of the collar about $\alpha$, So in this case, since
    $\ell_\alpha(s)\geq \ep_0$, we have
    \[
    \frac{\ell_\beta(s)}{\ell_\alpha(s)}\lmul \frac{1}{\ep_0}\log\frac {1}{\ep_0} \emul 1.
    \]
    If $\alpha$ is a leaf of $\lambda$, then $\ell_\alpha(b) =
    e^{b-a}\ell_\alpha(a)$, so the conclusion follows from the paragraph
    above. Now suppose that $\alpha$ crosses a leaf of $\lambda$. Let $C$
    be the constant of \lemref{l-growth}. If $d_\alpha(X_t,\lambda) < C$
    for all $t \in [a,b]$, then we are done. Otherwise, there is an
    earliest time $t \in [a,b]$ such that $d_\alpha(X_t,\lambda) \geq C$. It
    is immediate that $d_\alpha(X,X_t) \eadd 0$. By \lemref{l-growth},
    $\ell_\alpha(b) \emul e^{b-t} \ell_\alpha(t)$, so $d_\alpha(X_t,X_b)
    \eadd 0$ by the above paragraph. The result follows. \qedhere

  \end{proof}

  We will now prove the theorem stated at the beginning of this section.

  \begin{proof}[Proof of \thmref{Short}]

  If $\ell_\alpha \geq \ep_0$, then by
  \lemref{no-twist},
  \[ d_\alpha(X,Y) \eadd 0 \eadd \frac{1}{\ell_\alpha} \Log
  \frac{1}{\ell_\alpha}.\]

  Now suppose $\ell_\alpha < \ep_0$ and let $[a,b]$ be the active interval
  for $\alpha$. From \thmref{Twist}, the minimal length $\ell_\alpha$
  occurs at $t_\alpha \in [a,b]$ satisfying $t_\alpha - a \eadd
  \Log(b-t_\alpha)$, and $\ell_\alpha\emul e^{-(b-t_\alpha)}$. We then have
  \begin{equation*}
    \begin{split}
      d_\alpha(X_a,X_b) \emul  e^{b-a}  &= e^{b-t_\alpha} e^{t_\alpha-a}\\
      &\emul e^{b-t_\alpha}e^{\Log(b-t_\alpha)}
    \end{split}
  \end{equation*}

  If $(b-a)$ is large enough so that \thmref{Twist} gives
  $\Log(b-t_\alpha) = \log(b-t_\alpha)$, then this shows
  $d_\alpha(X_a,X_b) \emul e^{b-t_\alpha} (b-t_\alpha) \emul
  \frac{1}{\ell_\alpha} \log(\frac{1}{\ell_\alpha})$, and since
  $\ell_\alpha \le \ep_0$, we have
  $\frac{1}{\ell_\alpha} \log \frac{1}{\ell_\alpha} \eadd
  \frac{1}{\ell_\alpha} \Log \frac{1}{\ell_\alpha}$ with equality for
  $\ell_\alpha$ small enough.  By \lemref{no-twist}, $d_\alpha(X,X_a)$
  and $d_\alpha(X_b,Y)$ are both uniformly bounded. Thus
  $d_\alpha(X,Y) \eadd d_\alpha(X_a,X_b)$ and the estimate on
  $d_\alpha(X,Y)$ from \thmref{Twist} follows in this case.

  Otherwise, $(b-a)$ is bounded above by a universal constant, in
  which case we will show
  $d_\alpha(X,Y) \eadd \frac{1}{\ell_\alpha} \Log
  \frac{1}{\ell_\alpha}$ by showing that both sides are uniformly
  bounded.  First, the upper bound on $(b-a)$ gives a positive lower
  bound on $\ell_\alpha$ (which is already bounded above by $\ep_0$)
  and so $\frac{1}{\ell_\alpha} \Log \frac{1}{\ell_\alpha} \eadd 0$.
  On the other hand, using the bound on $(b-a)$, \thmref{Twist} gives
  $d_\alpha(X_a,X_b) \eadd 0$, and as before
  $d_\alpha(X,Y) \eadd d_\alpha(X_a,X_b)$.  We conclude
  $d_\alpha(X,Y) \eadd 0$, as required.

  For the last statement of \thmref{Short}, let $C$ be the constant of
  \lemref{l-growth}.  By assumption $\ell_\alpha(t) > \ep_0$ for all
  $t \leq a$. If there exists $t \leq a$ such that
  $d_\alpha(X_t,\lambda) \geq C$, then
  $\ell_\alpha(t_\alpha) \emul e^{t_\alpha-t} \ell_\alpha(t)$, where
  $t_\alpha$ is the time of the minimal length of $\ell_\alpha$. This
  is impossible for all sufficiently small $\ell_\alpha$. Finally,
  since
  $d_\alpha(X_a,X_b) \emul \frac{1}{\ell_\alpha} \Log
  \frac{1}{\ell_\alpha}$, for all sufficiently small $\ell_\alpha$, we
  can guarantee that $d_\alpha(X_b,\lambda) \geq C$. The final
  conclusion follows by \lemref{l-growth}.  \qedhere

  \end{proof}

  Recall that two curves that intersect cannot both have lengths less than
  $\ep_M$ at the same time. Therefore, if $\alpha$ and $\beta$ intersect
  and $\ell_\alpha < \ep_0$ and $\ell_\beta < \ep_0$, then their active
  intervals must be disjoint. This defines an ordering of $\alpha$ and
  $\beta$ along $\G$. In the next section, we will focus on the torus
  \torus and show that the order of $\alpha$ and $\beta$ along $\G$ will
  always agree with their order in the projection of $\G(t)$ to the Farey
  graph.

\section{Coarse description of geodesics in $\T(\torus)$}

  \label{sec:coarse}

  \subsection{Farey graph}

  See \cite{minsky:PT} for background on the Farey graph.

  Let \torus be the once-punctured torus and represent its universal
  cover by the hyperbolic plane \HH.  Identify the ideal boundary
  $\partial \HH$ with $\RR \cup \set{\infty}$. The point $\infty$ is
  considered an extended rational number with reduced form $1/0$.  As
  in \subsecref{curve-lam}, fix a positive ordered basis for
  $H_1(\torus)$ and use this to associate a slope
  $p/q \in \QP^1 = \QQ \cup \set{\infty}$ to every simple curve.  In
  this section we pass freely between a rational number and the
  associated simple curve.
  
  Given two curves $\alpha = p/q$ and $\beta=r/s$ in reduced fractions,
  their geometric intersection number is $|ps-rq|$.  Form a graph with
  vertex set $ \QP^1$ as follows: Connect $p/q$ and $r/s$ by an edge if
  $|ps-rq| = 1$. The resulting graph \farey is called the \emph{Farey
  graph}, which is also the curve graph of \torus.  This graph embeds
  naturally in $\HH \cup \partial \HH$, with its edges realized as
  hyperbolic geodesics (see \figref{FareyGraph}).  These geodesics cut
  $\HH$ into ideal triangles; this is the \emph{Farey tesselation}. In this
  tesselation, each edge bounds exactly two ideal triangles with zero
  relative shearing. Thus each edge of $\calF$ is equipped with a
  well-defined midpoint.

  \begin{figure}
  \begin{center}
  \includegraphics{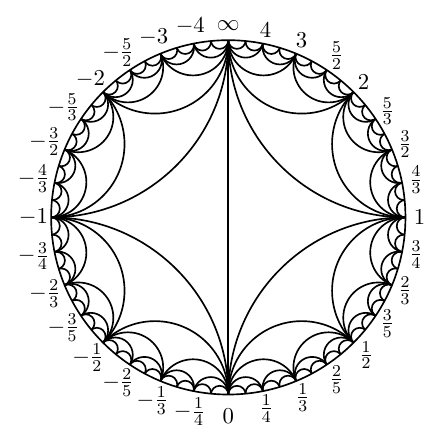}
  \end{center}
  \caption{The Farey graph.}
  \label{Fig:FareyGraph}
  \end{figure}

  Let $\alpha$ denote the curve with slope $1/0$.  The action of $D_\alpha$
  on curves distinct from $\alpha$ corresponds to the mapping of slopes $m
  \mapsto m+1$.  Let $\beta_0 \in \farey$ be any curve with $\I(\alpha,
  \beta_0) = 1$. The associated Dehn twist family about $\alpha$ is
  $\beta_n = D_\alpha^n (\beta_0)$. Then $\{ \beta_n \}_{n \in \ZZ}$ is
  exactly the set of vertices of \farey that are connected to $\alpha$ by
  an edge, or equivalently, the set of curves with slope in $\ZZ$.
  
  \subsection{Markings and pivots}

  A \emph{marking} on \torus is an unordered pair of curves
  $\set{\alpha,\beta}$ such that $\I(\alpha,\beta) = 1$.  Given a
  marking $\set{\alpha,\beta}$, there are four markings that are
  obtained from $\mu$ by an \emph{elementary move}, namely:
  \[
  \set{\alpha,D_\alpha(\beta)},
  \quad
  \set{\alpha,D_\alpha^{-1}(\beta)},
  \quad
  \set{\beta,D_\beta(\alpha)},
  \quad
  \set{\beta,D_\beta^{-1}(\alpha)}.
  \]
  Note that the set of markings on \torus can be identified with the set of
  edges of \farey, and two edges differ by an elementary move if and only
  if they bound a common triangle in the Farey tesselation of \HH. Denote
  by \mg the graph with markings as vertices and an edge connecting two
  markings that differ by an elementary move. Then \mg has the following
  property.

  \begin{lemma} \label{Lem:Marking}

  For any $\mu, \mu' \in \mg$, there exists a unique geodesic connecting
  them.

  \end{lemma}

  \begin{proof}

    Each edge of \farey separates \HH into two disjoint half-spaces. Let
    $E(\mu,\mu')$ be the set of edges in \farey that separate the interior
    of $\mu$ from the interior of $\mu'$. Set $\bE(\mu,\mu') = E(\mu,\mu')
    \cup \set{\mu,\mu'}$. Every $\nu \in E(\mu,\mu')$ disconnects $\mu$
    from $\mu'$, and thus must appear in every geodesic from $\mu$ and
    $\mu'$. Conversely, any $\nu \in \mg$ lying on a geodesic from $\mu$ to
    $\mu'$ must lie in $\bE(\mu,\mu')$. For each $\nu \in \bE(\mu,\mu')$,
    let $H_\nu$ be the half-space in \HH containing the interior of $\mu'$.
    There is a linear order on $\bE(\mu,\mu') = \set{\mu_1 < \mu_2 < \cdots
    < \mu_n}$ induced by the relation $\mu_i < \mu_{i+1}$ if and only if
    $H_{\mu_i} \supset H_{\mu_{i+1}}$. The sequence
    $\mu=\mu_1,\mu_2,\ldots,\mu_n=\mu'$ is the unique geodesic path in \mg
    from $\mu$ to $\mu'$. \qedhere

  \end{proof}
  
  Given two markings $\mu$ and $\mu'$ and a curve $\alpha$, let $n_\alpha$
  be the number of edges in $\bE(\mu,\mu')$ containing $\alpha$. We say
  $\alpha$ is a \emph{pivot} for $\mu$ and $\mu'$ if $n_\alpha \geq 2$, and
  $n_\alpha$ is the \emph{coefficient} of the pivot.  Let
  $\pivot(\mu,\mu')$ be the set of pivots for $\mu$ and $\mu'$. This set is
  naturally linearly ordered as follows. Given $\alpha \in
  \pivot(\mu,\mu')$, let $e_\alpha$ be the last edge in $\bE(\mu,\mu')$
  containing $\alpha$. Then for $\alpha, \beta \in \pivot(\mu,\mu')$, we
  set $\alpha < \beta$ if $e_\alpha$ appears before $e_\beta$ in $g$.

  Recall that in \subsecref{twisting} we defined the unsigned twisting
  (along $\alpha$) for a pair of curves $\beta,\beta'$; this is
  denoted $d_\alpha(\beta,\beta')$.  Generalizing this, we define
  unsigned twisting for the pair of markings $\mu,\mu'$ by
  \[
  d_\alpha(\mu,\mu') = \min_{\beta \subset \mu, \: \beta' \subset \mu'}
  d_\alpha(\beta,\beta'),
  \]
  where $\beta$ is a curve in $\mu$ and $\beta'$ is a curve in
  $\mu'$.  Similarly we define $d_\alpha(\beta,\mu') = \min_{\beta'
    \subset \mu'} d_\alpha(\beta,\beta')$.  In terms of these
  definitions, we have:

  \begin{lemma}[\cite{minsky:PT}] \label{Lem:Pivot}

    For any $\mu, \mu' \in \mg$ and curve $\alpha$, we have $n_\alpha \eadd
    d_\alpha(\mu,\mu')$. For $\alpha, \beta \in \pivot(\mu,\mu')$, if
    $\alpha < \beta$, then $d_\alpha(\beta,\mu') \eadd 1$ and
    $d_\beta(\mu,\alpha) \eadd 1$. Conversely, if $n_\alpha$ is
    sufficiently large and $d_\alpha(\beta,\mu') \eadd 1$, then $\alpha <
    \beta$.

  \end{lemma}
  
  Identify $\T(\torus)$ with \HH in the usual way. Under this
  identification, if $e$ is an edge of $\calF$ with endpoints $\alpha$ and
  $\beta$, then the set points along $e$ correspond to the set of surfaces
  on which $\alpha$ and $\beta$ are the shortest curves and they intersect
  perpendicularly. The midpoint of $e$ correspond to when the two curves
  have the same length. This length is a uniform constant independent of
  the edge $e$.   

  For any $X \in \T(\torus)$, there exists an ideal triangle $\triangle$ in
  the Farey tessellation of \HH containing $X$. The three vertices of
  $\triangle$ correspond to the three shortest curves on $X$. We will
  define a \emph{short marking} on $X$ as follows. If $X$ has at least $2$
  systoles, then let $A$ be the set of systoles on $X$. If $X$ has a unique
  systole, then let $A$ be set consisting of the systole plus the second
  shortest curves on $X$. In either case, $A$ is a subset of the vertices
  of $\triangle$, so $A$ has cardinality at most $3$ and every pair of
  curves in $A$ correspond to an edge in $\triangle$. A short marking on
  $X$ is any pair of curves in $A$. Note that in our definition, there is
  either a unique marking or three short markings on $X$. This implies
  that, given $X, Y \in \T(\torus)$, there are well-defined short markings
  $\mu_X$ and $\mu_Y$ on $X$ and $Y$ such that $d_\mg(\mu_X,\mu_Y)$ is
  minimal among all short markings on $X$ and $Y$. By \lemref{Marking}, the
  geodesic from $\mu_X$ to $\mu_Y$ is unique. Note that any edge of
  $E(\mu_X,\mu_Y)$ separates $\mu_X$ from $\mu_Y$, and hence it separates
  $X$ from $Y$. We will denote by $\pivot(X,Y) = \pivot(\mu_X,\mu_Y)$ and
  refer to $\pivot(X,Y)$ as the set of pivots for $X$ and $Y$.

  Given $X, Y \in \T(\torus)$, we have that $d_\alpha(X,Y) \eadd
  d_\alpha(\mu_X,\mu_Y)$.

  Let $\ep_0$ be the constant of the previous section. The following
  statements establish Theorem \ref{thm:t11-length-twisting} of the
  introduction.

  \begin{theorem}
  \label{thm:pivots-as-short-curves}

  Suppose $X, Y \in \T_{\ep_0}(\torus)$ and let $\G(t)$ be any
  geodesic from $X$ to $Y$, parameterized by an interval
  $I \subset \RR$. Let $\ell_\alpha = \inf_t \ell_\alpha(t)$. There are
  positive constants $\ep_1$, $C_1$, and $C_2$ such that
  \begin{rmenumerate}
  \item If $\ell_\alpha \leq \ep_1$, then $\alpha \in \pivot(X,Y)$ and
  $d_\alpha(X,Y) \geq C_1$. \label{Pivot1}
  \item If $d_\alpha(X,Y) \geq C_2$, then $\ell_\alpha \leq \ep_1$ and
  $\alpha \in \pivot(X,Y)$. \label{Pivot2}
  \item Suppose $\alpha$ and $\beta$ are distinct curves such that there
    exist $s, t \in I$ with $\ell_\alpha(s) \leq \ep_1$ and $\ell_\beta(t)
      \le \ep_1$. Then $\alpha < \beta$ in $\pivot(X,Y)$ if and only if $s
      < t$. \label{Pivot3}
  \item For any $\alpha \in \pivot(X,Y)$, $\ell_\alpha \lmul 1$.
  \label{Pivot4}

  \end{rmenumerate}

  \end{theorem}

  \begin{proof}

    The proof will show that any sufficiently small $\ep_1$ works.  We
    first require $\ep_1 < \ep_0$ where $\ep_0$ is the constant selected in
    the previous section.
    
    Let $\lambda = \Lambda(X,Y)$. On the torus, every curve $\alpha$
    interacts with $\lambda$. If $\lambda$ contains $\alpha$, then
    $\ell_\alpha(t) = e^t \ell_\alpha(X)$. But this implies $\ell_\alpha =
    \ell_X(\alpha) \geq \ep_0$ and $d_\alpha(X,Y) \eadd 0$ by
    \lemref{no-twist}. Thus, we may assume that $\alpha$ crosses a leaf of
    $\lambda$. By \thmref{Short}, $d_\alpha(X,Y) \emul
    \frac{1}{\ell_\alpha} \log \frac{1}{\ell_\alpha}$.  Since
    $d_\alpha(X,Y) \eadd d_\alpha(\mu_X,\mu_Y) \eadd n_\alpha$ (the latter
    by \lemref{Pivot}), we can select $\ep_1$ small enough and $C_1>0$ so
    that $\ell_\alpha \leq \ep_1$ implies that $d_\alpha(X,Y) \geq C_1$ and
    that $n_\alpha \geq 2$, i.e.~$\alpha$ is a pivot.  This gives
    \ref{Pivot1}.  Using the same approximate equalities, if
    $d_\alpha(X,Y)$ is large we find that $\ell_\alpha$ is small, and we
    can select $C_2$ satisfying \ref{Pivot2}. 
    
    We now fix our constants $\ep_1$, $C_1$ and $C_2$ so \ref{Pivot1} and
    \ref{Pivot2} are satisfied. By fixing these constants, we can now ignore
    the dependence of any additive or multiplicative errors on them.
    
    For \ref{Pivot3}, suppose $\ell_\alpha \le \ep_1$ and $\ell_\beta \le
    \ep_1$. By \ref{Pivot1}, they are both pivots. Let $[a,b]$ be the
    active interval for $\alpha$. Recall that that this is the longest
    interval such that $\ell_\alpha(a) = \ell_\beta(b) = \ep_0$.  Recall
    that by \corref{always-short} we have $\ell_\alpha(t) < \ep_M$ for all
    $t \in [a,b]$. Similarly, let $[c,d]$ be the active interval for
    $\beta$. On the torus, two curves always intersect, so $\alpha$ and
    $\beta$ cannot be simultaneously shorter than $\ep_M$, so $[a,b]$ and
    $[c,d]$ must be disjoint. By \lemref{Pivot}, $\alpha < \beta$ if and
    only if $d_\alpha(\beta,\mu_Y) \eadd 1$.  By \thmref{Short} and
    \lemref{no-twist}, $b < c$ if and only if $d_\alpha(X_c,Y) \eadd 1$.
    Since $\beta$ is $\ep_0$--short on $X_c$, we have
    $d_\alpha(\beta,\mu_Y) \eadd d_\alpha(X_c,Y)$. This finishes
    \ref{Pivot3}.
    
    Before we prove \ref{Pivot4}, we introduce some notation. For each
    curve $\alpha$, let $H_\alpha \subset \T(\torus)$ be the set of
    hyperbolic structures where $\ell_\alpha(X) \le \ep_1$. Since
    $\ep_1 < \ep_M$, the sets $H_\alpha$ and $H_\beta$ are disjoint if
    $\alpha \ne \beta$. Let $e$ be an edge of $\calF$ and denote its
    endpoints by $\alpha$ and $\beta$. The segment of $e$ outside of
    $H_\alpha$ and $H_\beta$ is a closed interval containing the
    midpoint of $e$. Along this interval, the length of $\alpha$ and
    $\beta$ is uniformly bounded (by a constant that depends only on
    $\ep_1$).

    To prove \ref{Pivot4}, let $\alpha \in \pivot(X,Y)$ and assume
    $\ell_\alpha > \ep_1$. Let $e \in E(\mu_X,\mu_Y)$ be an edge containing
    $\alpha$. Let $\beta$ be the other curve of $e$. The edge $e$ separates
    $X$ and $Y$, so any geodesic $\G(t)$ from $X$ to $Y$ must cross $e$ at
    some point $X_t$. If $\ell_\beta(t) > \ep_1$, then neither $\alpha$ or
    $\beta$ is $\ep_1$--short on $X_t$, so $X_t$ lies in the segment of $e$
    outside of $H_\alpha$ and $H_\beta$. Hence $\ell_\alpha(t) \lmul 1$ by
    the discussion in the previous paragraph. On the other hand, if
    $\ell_\beta(t) \le \ep_1$, then $\beta$ is a pivot by \ref{Pivot1}.
    Either $\alpha < \beta$ or $\beta < \alpha$ in $\pivot(X,Y)$. If
    $\alpha < \beta$, then $d_\beta(X,\alpha) \eadd 1$ by \lemref{Pivot}.
    Let $[a,b]$ be the active interval for $\beta$. By \thmref{Short} we
    have $d_\beta(X,X_a) \eadd 1$, and $d_\beta$ satisfies the triangle
    inequality up to additive error (by \cite[Equation~2.5]{minsky:CCII}),
    so we conclude $d_\beta(\alpha,X_a) \eadd 1$.  This, together with
    $\ell_\beta(a) = \ep_0$, yields $\ell_\alpha(a) \emul 1$. If $\beta <
    \alpha$, then the same argument using $X_b$ and $Y$ in place of $X_a$
    and $X$ also yields $\ell_\alpha(b) \emul 1$. This concludes the proof.
    \qedhere

  \end{proof}

\section{Envelopes in $\T(\torus)$}
  
  \label{Sec:envelopes}

  \subsection{Fenchel-Nielsen coordinates along stretch paths in
  $\T(\torus)$}

  \label{Subsec:in-and-out}

  We now focus on the once-punctured torus \torus, and on the completions
  $\alpha^\pm$ of the maximal chain-recurrent laminations containing a
  simple closed curve $\alpha$ discussed in \subsecref{curve-lam}.

  Consider the curve $\alpha$ as a pants decomposition of $\torus$ and
  define $\tau_\alpha(X)$ to be the Fenchel-Nielsen twist coordinate
  of $X$ relative to $\alpha$. Note that $\tau_\alpha(X)$ is well defined
  up to a multiple of $\ell_\alpha(X)$, and after making a choice at
  some point, $\tau_\alpha(X)$ is well defined. The Fenchel-Nielsen
  theorem states that the pair of functions
  $\big( \log \ell_\alpha(\param), \tau_\alpha(\param) \big)$ define a
  diffeomorphism of $\T(\torus) \to \RR^2$.

  Each $\alpha^\pm$ defines a foliation $\calF_\alpha^\pm$ on $\T(\torus)$
  whose leaves are the $\alpha^{\pm}$-stretch paths. In the $\alpha^\pm$
  shearing coordinate system, the image of $\T(\torus)$ in $\RR^2$ is a
  convex cone, and the foliation $\calF_\alpha^\pm$ are by open rays from
  the origin. 

  In this section we denote a point on the $\alpha^\pm$ stretch path
  through $X$ by $X_t^\pm = \str(X,\alpha^\pm,t)$.  The
  function $\log \ell_\alpha(X_t^\pm) = \log \ell_\alpha(X) + t$ is
  smooth in $t$. Our first goal is to establish the following theorem.
  
  \begin{theorem} \label{Thm:Smoothness}
    
    For any simple closed curve $\alpha$ on \torus and any point $X=X_0
    \in \T(\torus)$, the functions $\tau_\alpha(X_t^\pm)$ are smooth in
    $t$. Further, 
    \[ 
      \tau_\alpha(X_t^+) > \tau_\alpha(X_t^-) 
      \qquad \text{and} \qquad 
      \frac{d}{dt} \tau_\alpha(X_t^+) \big\vert_{t=0} > \frac{d}{dt}
      \tau_\alpha(X_t^-) \big\vert_{t=0}. 
    \] 
    That is, the pair of foliations $\calF_\alpha^+$ and $\calF_\alpha^-$
    are smooth and transverse.
  \end{theorem}
   
  We proceed to prove smoothness of $\tau_\alpha(X_t^+)$. Recall that
  the $\alpha^+$ shearing embedding is
  $s_{\alpha^+}(X) = (\ell_\alpha(X), s_\alpha(X))$ where
  $s_\alpha(X)$ was defined in \subsecref{shearing-t11}, and that like
  $\tau_\alpha$, the function $s_\alpha$ is defined only up to a
  adding an integer multiple of $\ell_\alpha(X)$.  To further lighten
  our notation, we will often write $\ell_\alpha(t)$ instead of
  $\ell_\alpha(X^+_t)$, and $s_\alpha(t)$ for $s_\alpha(X^+_t)$.

  We also denote $\tau_\alpha(0)$ and $\ell_\alpha(0)$ by $\tau_0$ and
  $\ell_0$ respectively. Note that the values of $\tau_\alpha$ and
  $\ell_\alpha$ do not depend on the choice of $\alpha^+$ or $\alpha^-$ but
  the values of the shearing coordinates do.

  We know, from the description of stretch paths in
  \subsecref{stretch}, that
  \[
    s_\alpha(t) = s_\alpha(0) e^t
    \quad\text{and}\quad
    \ell_\alpha(t) = \ell_0 \, e^t.
    \]

  \begin{figure}
   \begin{center}
  \includegraphics{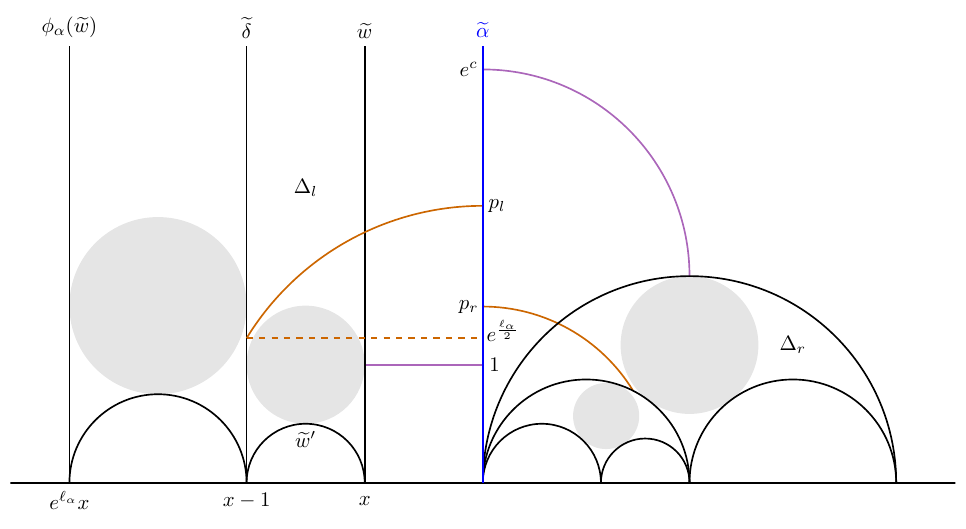}
  \end{center}
  \caption{Computing the Fenchel-Nielsen twist along the $\alpha^+$
  stretch path.}\label{Fig:Twist}
  \end{figure}

  We can now compute $\tau_\alpha(X_t^+)$ as follows, referring to
  \figref{Twist}. Fix a lift $\talpha$ of $\alpha$ to be the imaginary line
  (shown in blue in \figref{Twist}) in the upper half-plane $\HH$. Now
  develop the picture on both sides of $\talpha$.  Since we are considering
  $\alpha^+$, all the triangles on the left of $\talpha$ are asymptotic to
  $\infty$ and all the triangles on the right of $\talpha$ are asymptotic
  to $0$. Below we will choose some normalization, but first note that the
  hyperelliptic involution exchanges the two complementary triangles $T$
  and $T'$ of $\alpha^+$ while preserving $\alpha$ as a set. Let $\iota
  \from \HH \to \HH$ be a lift of this involution chosen to preserve
  $\talpha$, which therefore has the form $\iota (z) = -e^c/z$ for some $c
  = c_t \in \RR$. Notice that $\iota$ exchanges the two sides of $\talpha$
  and that it fixes a unique point $ie^{c/2}$ in $\HH$.

  To fix the shearing coordinate $s_\alpha(t)$, we make the choice of
  triangles in $\HH$ required by the construction of
  \subsecref{shearing-tri}. Choose two triangles $\Delta_l$ and
  $\Delta_r$ in $\HH$ separated by $\talpha$ so that one is a lift of
  $T$ and the other is a lift of $T'$ and $\iota(\Delta_r) = \Delta_l$.
  Let $\tw$ be the edge of $\Delta_l$ that is a lift of $w$, namely,
  $\tw=[x,\infty]$ for some $x < 0$. Let
  $\phi_\alpha(z) = e^{\ell_\alpha} z$ be the isometry associated to
  $\talpha$ oriented toward $\infty$. The image
  $\phi_\alpha(\tw)=[e^{\ell_\alpha} x, \infty]$ is another lift of
  $w$. Let $\tdelta$ be the lift of $\delta$ that is asymptotic to
  $\tw$ and $\phi_\alpha(\tw)$.  By applying a further dilation to the
  picture if necessary, we can assume that $\tdelta = [x-1, \infty]$.
  Now, the geodesic $\tw'=[x,x-1]$ is a lift of $w'$.

  With our normalization, the midpoint of $[x,\infty]$ associated to
  $\Delta_l$ is the point $(x,1)$. Recall that $s_\delta=-\ell_\alpha$
  in this case, which is the shearing between triangles
  $\Delta_l=[x, x-1, \infty]$ and $[x-1, e^{\ell_\alpha}x,
  \infty]$.  This means that their midpoints on $\Tilde{\delta}$ have
  $y$-coordinates with ratio $e^{\ell_\alpha}$, i.e.
  \[
  \frac{ |e^{\ell_\alpha} x - (x-1)| }{|(x-1) - x|} = e^{\ell_\alpha}
  \]
  from which it follows that
  \[
  x=- \coth(\ell_{\alpha}/2).
  \]
  Let $h_l$ be the endpoint on $\talpha$ of the horocycle based at infinity
  containing the midpoint of $\widetilde{w}$ considered as an edge of
  $\Delta_l$.  Let $h_r = \iota(h_l)$, By construction, $h_l = 1$ and $h_r =
  e^c$. We can normalize so that $s_\alpha = c$.

  To visualize the Fenchel-Nielsen twist parameter $\tau_+(t)$ at
  $X_t$ about $\alpha$, consider the shortest geodesic arc $\beta$
  with both endpoints on $\alpha$ intersecting perpendicularly (so
  $\beta$ only intersects $\alpha$ twice).  By symmetry, this arc
  intersects $\delta$ at a point $q$ that is equidistant to the
  midpoints of $\delta$ associated to $T$ and $T'$. We choose a lift
  $\tbeta$ that passes through $\tq = (x-1, e^{\ell_\alpha/2})$. Let
  $p_l$ be the endpoint on $\talpha$ of the lift of $\beta$ that
  passes through $\tdelta$. Since $\tbeta$ is perpendicular to
  $\talpha$, we have $q$ and $p_l$ lie on a Euclidean circle centered
  at the origin. Using the Pythagorean theorem, we obtain:
  \[
    p_l = \sqrt{(x-1)^2 + \left(e^{\ell_\alpha/2}\right)^2} =
    e^{\ell_\alpha/2} \coth \frac{\ell_\alpha}{2}.
  \]
  Let $ip_r = \iota(ip_l) = -e^c/ip_l = i e^c/p_l$. Up to an integral
  multiple of $\ell_\alpha$, the twisting $\tau_\alpha(t)$ is the signed
  distance between $ip_l$ and $ip_r$; that is
  \begin{align*}
    \tau_\alpha 
    & = \log \frac{p_r}{p_l} \mod \ell_\alpha \\ 
    & = \log \frac{e^c}{e^{\ell_\alpha}
    \coth^2(\frac{\ell_\alpha}2)} \mod \ell_\alpha 
    \\
    & = c - 2\log \coth \frac{\ell_\alpha}{2} \mod \ell_\alpha\ \\
    & = s_\alpha - 2\log \coth \frac{\ell_\alpha}{2} \mod \ell_\alpha 
  \end{align*}
  In particular, at $t=0$, we obtain 
  \begin{align} \label{Twist}
    \tau_0 = s_\alpha(0) - 2 \log \coth \frac{\ell_0}{2} \mod  \ell_0.
  \end{align}
  As we mentioned previously, $s_\alpha(t)=s_\alpha(0) e^t$ and
  $\ell_\alpha(t) = \ell_0 e^t$. Hence,
  \begin{align*}
    \tau_\alpha(t)
    = e^t s_\alpha(0) - 2\log \coth \frac{e^t \ell_0}{2} \mod \ell_\alpha. 
  \end{align*}
  Solving for $\tau_0$ using \eqref{Twist}, we obtain
  \begin{equation}
  \label{Eq:Twist+}
  \tau_\alpha \big( X_t^+ \big) = e^t \tau_0 + 2 e^t \log \coth
    \frac{\ell_0}{2} - 2\log \coth \frac{e^t \, \ell_0}{2} \mod
    \ell_\alpha.
  \end{equation}

  Now let $X_t^-$ be the stretch path starting from $X$ associated to
  $\alpha^-$. The computation in this case is similar; in fact
  $2 \tau = s_\alpha^+ + s_\alpha^- \mod \ell_\alpha$. Thus  
  \begin{equation}
  \label{Eq:Twist-}
    \tau_\alpha\big(X_t^- \big) =
    e^t \tau_0 - 2e^t \log \coth \frac{\ell_0}{2} + 2\log \coth
    \frac{e^t \ell_0}{2} \mod \ell_\alpha.
  \end{equation}
  
  This shows that $\tau(X_t^+)$ and $\tau(X_t^-)$ are both smooth
  functions of $t$. Note that $\tau(X_t^+) - \tau(X_t^-)$ is
  well-defined.  By a simple computation, we see that
  $\tau(X_t^+) - \tau(X_t^-) > 0$, and
  \begin{align*}
    \frac{d}{dt} \Big(\tau_\alpha(X_t^+) - \tau_\alpha(X_t^-) \Big)
    \Big|_{t=0} = 4 \log \coth \frac{\ell_0}{2} + 2 \ell_0 \tanh
    \frac{\ell_0}{2} \csch^2 \frac{\ell_0}{2} > 0. 
  \end{align*}
  This finishes the proof of \thmref{Smoothness}.

  \subsection{Structure of envelopes in general} \label{Subsec:envelopes}

  For any surface \s of finite type and a chain-recurrent lamination $\lambda$ on
  \s and $X \in \T(\s)$, define 
  \[ 
    \Out(X,\lambda) =
    \set{Z \in \T(\s) \st \lambda = \Lambda(X,Z)} 
    \quad\text{and}\quad
    \In(X,\lambda) = \set{Z \in \T(\s) \st \lambda = \Lambda(Z,X)}.
  \]
  We call these the \emph{out-envelope} and \emph{in-envelope} of $X$
  (respectively) in the direction $\lambda$.
  \begin{proposition} \label{Prop:PropofEnv}
    The out-envelopes and in-envelopes have the following properties. 
    \begin{rmenumerate}
    \item \label{Env1} If $\lambda$ is maximal chain-recurrent, then
      for any completion $\widehat{\lambda}$ of $\lambda$, the set
      $\Out(X,\lambda)$ is the stretch ray starting at $X$ associated
      to $\widehat{\lambda}$, and the set $\In(X,\lambda)$ is the
      stretch ray associated to $\widehat{\lambda}$ ending at
      $X$.
    \item \label{Env2} The closure of $\Out(X,\lambda)$ consists of
      points $Y$ with $\lambda \subset \Lambda(X,Y)$. Similarly, the
      closure of $\In(X,\lambda)$ is the set of points $Y$ with
      $\lambda \subset \Lambda(X,Y)$.
    \item \label{Env3} If $\lambda$ is a simple closed curve, then
      $\Out(X,\lambda)$ and $\In(X,\lambda)$ are open sets.
    \end{rmenumerate}
  \end{proposition}

  \begin{proof}  
    First assume $\lambda$ is maximal chain-recurrent and let
    $\widehat{\lambda}$ be a completion of it. By \corref{Geodesics},
    if $\Lambda(X,Y) = \lambda$, then there exists $t > 0$ such that
    $Y = \str(X,\widehat{\lambda},t)$ and this is the only geodesic
    from $X$ to $Y$.  That is, any point in $\Out(X,\lambda)$ can be
    reached from $X$ by following the stretch ray along
    $\widehat{\lambda}$ starting at $X$.  Similarly if
    $Y \in \In(X,\lambda)$, then the stretch ray along
    $\widehat{\lambda}$ starting at $Y$ contains $X$, or equivalently,
    the stretch ray along $\widehat{\lambda}$ ending at $X$ contains
    $Y$.  This is \ref{Env1}.
 
    For the other statements, we use \cite[Theorem~8.4]{thurston:MSM},
    which shows that if $Y_i$ converges $Y$, then any limit point of
    $\Lambda(X,Y_i)$ in the Hausdorff topology is contained in
    $\Lambda(X,Y)$. Applying this to a point $Y$ in the closure of
    $\Out(X,\lambda)$, and a sequence $Y_i \in \Out(X,\lambda)$
    converging to $Y$, we obtain $\lambda \subset \Lambda(X,Y)$. For
    the other direction of \ref{Env2}, let $Y$ be any point such that
    $\lambda \subset \Lambda(X,Y)$. To show $Y$ is in the closure of
    $\Out(X,\lambda)$, we find a point $Z \in \Out(X,\lambda)$ such
    that $\dth(Y,Z) = \ep$, for any \ep. Let $\lambda'$ be any maximal
    chain-recurrent lamination such that
    $\lambda = \lambda' \cap \Lambda(X,Y)$, and let
    $Z=\str(Y,\lambda',\ep)$. We have $\dth(Y,Z) = \ep$. Since
    $\lambda = \lambda' \cap \Lambda(X,Y)$, we must have
    $\Lambda(X,Z) = \lambda$. This shows \ref{Env2} for
    $\Out(X,\lambda)$. The analogous statement for $\In(X,\lambda)$ is
    proven similarly.
    
    To obtain \ref{Env3}, let $\lambda$ be a simple closed curve, $Y \in
    \Out(X,\lambda)$, and $Y_i$ is any sequence converging to $Y$, then any
    limit point of $\Lambda(X,Y_i)$ is contained in $\lambda$. Since
    $\lambda$ is a simple closed curve, $\Lambda(X,Y_i) = \lambda$ for all
    sufficiently large $i$. This shows $\Out(X,\lambda)$ is open. The same
    proof also applies to $\In(X,\lambda)$. \qedhere
     
  \end{proof}
  
  Let $X,Y\in \T(\s)$, and denote $\lambda=\Lambda(X,Y)$. We define the
  envelope of geodesics from $X$ to $Y$ to be the set
  \[
    \Env(X,Y) = \Set{Z \st Z \in [X,Y] \text{ for some geodesic } [X,Y] }.
  \]
  
  \begin{proposition}
    For any $X, Y \in \T(\s)$, $\Env(X,Y) = \overline{\Out\big( X,\lambda
    \big)} \cap \overline{\In\big( Y, \lambda \big)}$. 
  \end{proposition}
  
  \begin{proof}
    For any $Z \in \Env(X,Y)$, since $Z$ lies on a geodesic from $X$ to
    $Y$, $\lambda$ must be contained in $\Lambda(X,Z)$ and in
    $\Lambda(Z,Y)$. This shows $\Env(X,Y) \subset
    \overline{\Out(X,\lambda)} \cap \overline{\In(Y,\lambda)}$. On the
    other hand, if $Z \in \overline{\Out(X,\lambda)} \cap
    \overline{\In(Y,\lambda)}$, then $\lambda \subset \Lambda(X,Z)$ and
    $\lambda \subset \Lambda(Z,Y)$. That is, if $\mu$ is the stump of
    $\lambda$, then $\dth(X,Z)=\log\frac{\ell_\mu(Z)}{\ell_\mu(X)}$ and
    $\dth(Z,Y)=\log\frac{\ell_\mu(Y)}{\ell_\mu(Z)}$, so
    $\dth(X,Y)=\dth(X,Z)+\dth(Z,Y)$. Thus, the concatenation of any
    geodesic from $X$ to $Z$ and from $Z$ to $Y$ is a geodesic from $X$ to
    $Y$. \qedhere
  \end{proof}

  \subsection{Structure of envelopes in $\T(\torus)$}
  \label{Subsec:t11-envelopes}
  
  In this section, we specialize our study of envelopes to the case of
  $\s=\torus$, and prove Theorem \ref{thm:main-envelopes} of the
  introduction. The proof is divided into several propositions.

  \begin{proposition} \label{Prop:Sector}
    
    Let $\alpha$ be a simple closed curve on \torus. For any
    $X \in \T(\torus)$, the set $\Out(X,\alpha)$ is an open region
    bounded by the stretch rays along $\alpha^\pm$ starting at $X$.
    Similarly, $\In(X,\alpha)$ is an open region bounded by the
    stretch rays along $\alpha^\pm$ ending at $X$.

  \end{proposition}

  \begin{proof}

    Set $X_t^\pm = \str(X,\alpha^\pm,t)$. By \thmref{Geodesics}, for
    any surface \s and any two points $X, Y \in \T(\s)$, Thurston
    constructed a geodesic from $X$ to $Y$ that is a concatenation of
    stretch paths, where the number of stretch paths needed in the
    concatenation is bounded by $2|\chi(S)|$, i.e.\ the number of
    triangles in an ideal triangulation of \s. In our setting where
    $\s = \torus$, for $Y \in \Out(X, \alpha)$, this would be either a
    single stretch path or a union of two stretch paths $[X,Z]$ and
    $[Z,Y]$ where both $\Lambda(X,Z)$ and $\Lambda(Z,Y)$ contain
    $\alpha$.  By \corref{Geodesics}, each one of these is a stretch
    path along either $\alpha^+$ or $\alpha^-$. The initial path can
    be chosen to stretch along $\alpha^+$ or $\alpha^-$
    arbitrarily. Assuming we first stretch along $\alpha^-$, then
    there are $t_1$ and $t_2$ such that $Z= \str(X,\alpha^-,t_1)$,
    $Y= \str(Z,\alpha^+,t_2)$, and $\dth(X,Y) = t_1+t_2$. Set
    $Z_t^- = \str(Z,\alpha^-,t)$ and $Z_t^+ = \str(Z,\alpha^+,t)$. By
    the calculations of the previous section,
    $\tau_\alpha\left(Z^-_t\right) < \tau_\alpha(Z^+_t)$. Since
    $Z^-_t = X^-_{t+t_1}$ and $Z^+_{t_2} = Y$, we have
    \[ 
      \ell_\alpha\left(X^-_{t_1+t_2}\right) = \ell_\alpha(Y) 
      \qquad\text{and}\qquad
      \tau_\alpha \left(X^-_{t_1+t_2} \right) < \tau_\alpha(Y)
    \]
    Similarly, if we stretch along $\alpha^+$ first, then there are $s_1$
    and $s_2$ such that  $W= \str(X,\alpha^+,s_1)$, $Y =
    \str(W,\alpha^-,s_2)$, and $s_1+s_2 = t_1 + t_2$. Then $X^+_{t_1+t_2} =
    X^+_{s_1+s_2}$ and by the same argument as above
    \[ 
      \ell_\alpha\left(X^+_{t_1+t_2}\right) = \ell_\alpha(Y) 
      \qquad\text{and}\qquad
      \tau_\alpha(Y) < \tau_\alpha \left(X^+_{t_1+t_2} \right)
    \]
    That is, $Y$ is inside of the sector bounded by the stretch rays
    $X_t^+$ and $X_t^-$ for $t>0$. By replacing geodesics from $X$ to
    $Y$ by by geodesics from $Y$ to $X$, we obtain the statement for
    $\In(X,\alpha)$. \qedhere

  \end{proof}

  \begin{remark}[Visualization of envelopes] \label{Rem:vis} \figzero (on the title page) illustrates Proposition
\ref{Prop:Sector} by showing the sets $\In(X,\alpha)$ in the Poincar\'e disk
model of $\T(\torus)$ for $X$ the hexagonal punctured torus and for several
simple curves $\alpha$, including the three systoles.  In the figure, the disk
model is normalized so that the origin corresponds to the hexagonal punctured
torus.  This figure was produced as follows:   The Fenchel-Nielsen coordinate
computations of \eqref{Eq:Twist+}--\eqref{Eq:Twist-} make it straightforward
to compute stretch paths passing through a given point in the relative
$\SL(2,\RR)$ character variety of $\pi_1(\torus)$.  The software package
\texttt{CP1} \cite{cp1} allows the computation of the uniformization map from
the disk to the relative character variety; this map was numerically inverted
using Newton's method to transport the computed stretch paths to the disk.

  By the results of \cite{theret:negative-convergence}, the stretch lines appearing as boundaries of in-envelopes for $\T(\torus)$ are exactly those which limit on rational points on the circle at infinity as $t \to -\infty$.  Thus \figzero can be alternatively described as showing regions bounded by the pairs of stretch rays joining several rational points at infinity to the hexagonal punctured torus.
  \end{remark}

  \begin{corollary} \label{Cor:Sector}
    
    Given $X, Y \in \T(\torus)$, if $\Lambda(X,Y)$ is a simple closed
    curve, then $\Env(X,Y)$ is a compact quadrilateral. 

  \end{corollary}
  
  \begin{proof}

    The statement follows from \propref{Sector} and the fact that
    $\Env(X,Y) = \overline{\Out(X,\alpha)} \cap \overline{\In(Y,\alpha)}$.
    \qedhere 

  \end{proof}

  \begin{proposition} \label{Prop:Continuity}
   
    In $\T(\torus)$, the set $\, \Env(X,Y)$ varies continuously as a
    function of $X$ and $Y$ with respect to the topology induced by the
    Hausdorff distance on closed sets. 
    
  \end{proposition}

  \begin{proof}
    
    First suppose $\Lambda(X,Y)$ is a simple closed curve $\alpha$. By
    \cite[Theorem~8.4]{thurston:MSM}, if $X_i \to X$ and $Y_i \to Y$, then
    $\Lambda(X,Y)$ contains any limit point of $\Lambda(X_i,Y_i)$; thus
    $\Lambda(X_i,Y_i) = \alpha$ for all sufficiently large $i$. That is,
    for sufficiently large $i$, $\Env(X_i,Y_i)$ is a compact quadrilateral
    bounded by segments in the foliations $\calF_\alpha^\pm$. Let $Z$ be
    the \emph{left corner} of $\Env(X,Y)$, i.e.~the intersection point of the
    leaf of $\calF^+_\alpha$ through $X$ and the leaf of $\calF^-_\alpha$
    through $Y$. For any neighborhood $U$ of $Z$, by smoothness and
    transversality of $\calF_\alpha^\pm$, there is a neighborhood $U_X$ of
    $X$ and a neighborhood $U_Y$ of $Y$, such that for all sufficiently
    large $i$, $X_i \in U_X $, $Y_i \in U_Y$, and the leaf of
    $\calF^+_\alpha$ through $X_i$ and the leaf of $\calF^-_\alpha$ through
    $Y_i$ will intersect in $U$. That is, for all sufficiently large $i$,
    the left corner of $\Env(X_i,Y_i)$ lies close to the left corner of
    $\Env(X,Y)$. A similar argument holds for the right corners. This shows
    $\Env(X_i,Y_i)$ converges to $\Env(X,Y)$.

    Now suppose $\Lambda(X,Y) = \lambda$ is a maximal chain-recurrent
    lamination and $X_i \to X$ and $Y_i \to Y$. Let $\widehat\lambda$
    be the canonical completion of $\lambda$, and let $\calG$ be the
    stretch path along $\widehat{\lambda}$ passing through $X$ and
    $Y$. Also let $\calG_i$ and $\calG_i'$ be the stretch paths along
    $\widehat{\lambda}$ through $X_i$ and $Y_i$ respectively.  Since
    stretch paths along $\widehat{\lambda}$ foliate $\T(\torus)$,
    $\calG_i$ and $\calG_i'$ either coincide or are disjoint. In the
    backward direction, all stretch paths along $\widehat{\lambda}$
    converge to $\lambda$ (the stump of $\widehat{\lambda}$) in $\PML$
    \cite{Pap}. If they coincide, then $\Lambda(X_i,Y_i) = \lambda$
    and $\Env(X_i,Y_i)$ is a segment of $\calG_i$. If they are
    disjoint, then they divide $\T(\torus)$ into three disjoint
    regions. Let $M_i$ be the closure of the region bounded by
    $\calG_i \cup \calG_i'$; see \figref{Sandwich}. In the case that
    $\calG_i = \calG_i'$, set $M_i = \calG_i$.
    For any geodesic $L$ from $X_i$ to $Y_i$, since $X_i, Y_i \in M_i$, if
    $L$ leaves $M_i$, then it must cross either $\calG_i$ or $\calG_i'$ at
    least twice. But two points on a stretch path cannot be connected by
    any other geodesic in the same direction, so $L$ must be contained
    entirely in $M_i$. Therefore, $\Env(X_i,Y_i) \subset M_i$ (see
    \figref{Sandwich}). Since $\calG_i$ and $\calG_i'$ converge to $\calG$,
    $M_i$ also converges to $\calG$. Therefore $\Env(X_i,Y_i)$ converges to
    a subset of $\calG$. For any $Z_i \in \Env(X_i,Y_i)$, $\dth(X_i,Z_i) +
    \dth(Z_i,Y_i) = \dth(X_i,Y_i)$, so by continuity of $\dth$, $Z_i$ must
    converge to a point $Z \in \calG$ with $\dth(X,Z) + \dth(Z,Y) =
    \dth(X,Y)$. In other words, $Z$ lies on the geodesic from $X$ to $Y$.
    This shows $\Env(X_i,Y_i)$ converges to $\Env(X,Y)$. \qedhere
  
    \end{proof}
    
    \begin{figure}
    \begin{center}
      \includegraphics{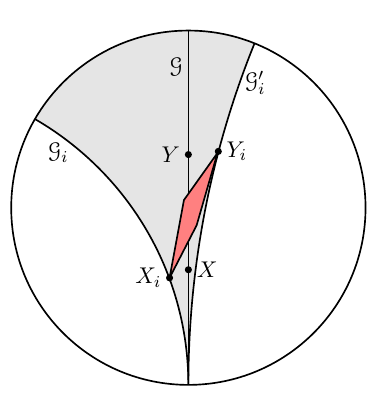}
    \end{center}
      \caption{$\Env(X_i,Y_i)$ is sandwiched between $\calG_i$ and
      $\calG_i'$.}
    \label{Fig:Sandwich}
    \end{figure}

    We can now assemble the proof of Theorem \ref{thm:main-envelopes}: Part
    (ii) is Proposition \ref{Prop:Continuity}, part (iii) is Proposition
    \ref{Prop:PropofEnv}(i), and part (iv) is Corollary \ref{Cor:Sector}.
    Part (i) is immediate by Corollary \ref{Cor:Sector} for simple closed
    curves and by part (iii) for the remaining case.

\section{Thurston norm and rigidity}

  In this section we introduce and study Thurston's norm, which is the
  infinitesimal version of the metric $\dth$, and prove Theorems
  \ref{thm:infinitesimal-rigidity} and \ref{thm:isometries-local}.

  \subsection{The norm}
  \label{Subsec:norm}

  Thurston showed in \cite{thurston:MSM} that the metric $\dth$ is
  Finsler, i.e.~it is induced by a norm \tnorm{\param} on the tangent
  bundle.  This norm is naturally expressed as the infinitesimal
  analogue of the length ratio defining $\dth$:
  \begin{equation}
  \label{eqn:thurston-norm}
  \tnorm{v} = \sup_{\alpha} \frac{d_X
    \ell_\alpha(v)}{\ell_\alpha(X)} = \sup_\alpha d_X (\log
  \ell_{\alpha})(v), \;\; v \in T_X\T(S)
  \end{equation}
  
  The following regularity of the norm will be needed in our study of
  isometries of Thurston's metric:

  \begin{theorem}
  \label{thm:thurston-norm-lipschitz}
  Let $S$ be a surface of finite hyperbolic type.  Then the Thurston
  norm function $T\T(S) \to \RR$ is locally Lipschitz.
  \end{theorem}

  The Thurston norm is defined as a supremum of a collection of
  $1$-forms; we will deduce its regularity from that of the forms.  In
  preparation for stating a result to that effect, we must introduce
  some terminology.

  Let $M$ be a smooth manifold, let $\pi: V \to M$ be a vector bundle
  over $M$, and let $\mathcal{E}$ be a collection of sections of $V$.
  We say that $\mathcal{E}$ is \emph{locally uniformly bounded} if for
  each $x \in M$ there exists a neighborhood $U$ of $x$ and a compact
  set $K \subset V$ such that for each $y \in U$ and
  $e \in \mathcal{E}$ we have $e(y) \in K$.  We say that $\mathcal{E}$
  is \emph{locally uniformly Lipschitz} if for each $x \in M$ there
  exists a neighborhood $U$ of $x$, a local trivialization
  $\varphi : \pi^{-1}(U) \xrightarrow{\sim} U \times \RR^n$, and a
  constant $M$ so that for each $e \in \mathcal{E}$, if we use the
  local trivialization $\varphi$ to regard the section $e$ as a map
  $U_i \to \RR^n$, then this function is $M$-Lipschitz.  Here we fix
  any background norm on $\RR^n$ in order to define Lipschitz functions
  to that space; because all such norms are bi-Lipschitz equivalent,
  the definition of locally uniformly Lipschitz does not depend on
  that choice.

  \begin{lemma}
  \label{lem:lipschitz-sections}
  Let $M$ be a smooth manifold and $\mathcal{E}$ a collection of
  $1$-forms on $M$.  Suppose that $\mathcal{E}$, considered as a
  collection of sections of $T^*M$, is locally uniformly bounded and
  locally uniformly Lipschitz.  Then the function $E : TM \to \RR$
  defined by
  \[
  E(v) := \sup_{e \in \mathcal{E}} e(v)
  \]
  is locally Lipschitz (assuming it is finite at one point).
  \end{lemma}

  Note that ``locally Lipschitz'' is a well-defined property of a
  function on a smooth manifold or a section of a vector bundle; it is
  equivalent to saying that the collection consisting of only that
  section (or function) is locally uniformly Lipschitz.

  \begin{proof}
  Any linear function $\RR^n \to \RR$ is Lipschitz, however the
  Lipschitz constant is proportional to its norm as an element of
  $(\RR^n)^*$.  Thus, for example, a family of linear functions is
  uniformly Lipschitz only when the corresponding subset of $(\RR^n)^*$
  is bounded.

  For the same reason, if we take a family of $1$-forms on $M$
  (sections of $T^*M$) and consider them as fiberwise-linear functions
  $TM \to \RR$, then in order for these \emph{functions} on $TM$ to be
  locally uniformly Lipschitz, we must require the \emph{sections} of
  $T^*M$ to be both locally uniformly Lipschitz and locally uniformly
  bounded.  Here, the compact set $K$ in the definition of locally
  uniformly bounded ensures that the pointwise norms of the sections
  in $T^*M$ are bounded in a neighborhood of any point.

  Thus the hypotheses on $\mathcal{E}$ are arranged exactly so that the
  family of functions $TM \to \RR$ of which $E$ is the supremum is
  locally uniformly Lipschitz.

  The supremum of a family of locally uniformly Lipschitz functions is
  locally Lipschitz or identically infinity.  Since the function $E$ is
  such a supremum, we find that it is locally Lipschitz once it is
  finite at one point.
  \end{proof}

  \begin{proof}[Proof of Theorem \ref{thm:thurston-norm-lipschitz}.]
  By \eqref{eqn:thurston-norm}, the Thurston norm is a supremum of the
  type considered in Lemma \ref{lem:lipschitz-sections}.  Therefore,
  it suffices to show that the set
  \[ \dlogC := \set{ \dlog \ell_\alpha \st \alpha \text{ a simple curve } }
  \]
  of $1$-forms on $\T(S)$ is locally uniformly bounded and locally
  uniformly Lipschitz.  

  To see this, first recall that length functions extend continuously
  from curves to the space $\ML(S)$ of measured laminations (see
  e.g.~\cite{thurston:hyp2}, \cite[Prop.~4.6]{Bon86}), and also that
  they extend from real-valued functions on Teichm\"uller space to
  holomorphic functions on the complex manifold $\QF(S)$ of
  quasi-Fuchsian representations (see \cite[p.~292]{bonahon:shearing})
  in which $\T(S)$ is a totally real submanifold.  The resulting
  length function $\ell_\lambda : \QF(S) \to \mathbf{C}$ depends
  continuously on $\lambda$ in the locally uniform topology of
  functions on $\QF(S)$ \cite[pp.~20--21]{bonahon:variations}.

  For holomorphic functions, locally uniform convergence implies
  locally uniform convergence of derivatives of any fixed order, so we
  find that the derivatives of $\ell_\lambda$ also depend continuously
  on $\lambda$.

  Restricting to $\T(S) \subset \QF(S)$, and noting that the length of a
  nonzero measured lamination does not vanish on $\T(S)$, we see that
  the $1$-form $\dlog(\ell_\lambda)$ on $\T(S)$ is real-analytic, and
  that the map $\lambda \mapsto \dlog(\ell_\lambda)$ is continuous from
  $\ML(S) \setminus \{0\}$ to the $C^1$ topology of $1$-forms on any
  compact subset of $\T(S)$.

  Since the $1$-form $\dlog \ell_\lambda$ is invariant under scaling
  $\lambda$, it is naturally a function (still $C^1$ continuous) of
  $[\lambda] \in \PML(S) = (\ML(S) \setminus \set{0})/\RR^+$.  Because
  $\PML(S)$ is compact, this implies that the collection of $1$-forms
  \[
  \dlogPML := \set{\dlog \ell_\lambda \st [\lambda] \in \PML(S) }
  \] is locally
  uniformly bounded in $C^1$.  In particular it is locally uniformly
  Lipschitz, and since this collection contains $\dlogC$, we are
  done.
  \end{proof}

  \subsection{Shape of the unit sphere}
  
  Fix $X \in \T(S)$ for the rest of this section.  Let
  $T^1_X\T(S)$ denote the unit sphere of Thurston's
  norm, i.e.~
  \[
  T^1_X\T(S) = \{ v \in T_X\T(S) \st \tnorm{v} = 1 \}.
  \]
  Similarly, let $T_X^{\leq 1} \T(S)$ denote the unit ball of
  Thurston's norm.

  The dual of the convex set $T_X^{\leq 1} \T(S)$ has a convenient description in terms of measured laminations:
  \begin{theorem}[{Thurston \cite{thurston:MSM}}]
    \label{thm:pml-embedding}
    The map $\PML(S) \to T_X^*\T(S)$ given by
    $\mu \mapsto d_X \log \ell_\mu$ embeds $\PML(S)$ as the boundary
    of a convex neighborhood of the origin.  This convex neighborhood
    is the dual convex set of $T_X^{\leq 1} \T(S)$.
  \end{theorem}

  Unlike this dual set, a typical point in the boundary of
  $T_X^{\leq 1}\T(S)$ does not have a canonical description in terms
  of a lamination on $S$.  However, certain points in the sphere arise
  from stretch paths.  Specifically, let $\CL$ denote the set of all
  complete geodesic laminations on $S$.  We have a map
  \[
    v_X \from \CL \to T^1_X\T(S)
  \]
  where $v_X(\lambda)$ is the tangent vector at $t=0$ to the stretch
  path $t \mapsto \str(X,\lambda,t)$.  This map is ``dual'' to the map
  $d_X \log \ell_{\param}$ in the weak sense that
  $d_X \log \ell_\mu(v_X(\lambda)) = 1$ if $\mu$ is a measured
  lamination whose support is contained in $\lambda$.

  For later use, it will be important to note the continuity of the
  map $v_X$, which follows easily from the results of
  \cite{bonahon:variations}:

  \begin{lemma}
    \label{lem:vx-continuity}
    The map $v_X$ is continuous with respect to the Hausdorff topology on $\CL$.
  \end{lemma}

  \begin{proof}
    Let $\lambda_n \in \CL$ be a sequence that converges in the
    Hausdorff topology.  In \cite[pp.20--21]{bonahon:variations},
    Bonahon shows that the associated shearing embeddings
    $s_{\lambda_n} : \T(S) \to \RR^N$ converge in the $C^k$
    topology\footnote{More precisely, Bonahon shows locally uniform
      convergence of a sequence of holomorphic embeddings that
      complexify the shearing coordinates.  Locally uniform
      convergence of holomorphic maps implies local $C^k$
      convergence.}  to $s_\lambda$ on any compact subset of $\T(S)$.
    Since stretch paths are rays from the origin in the shearing
    coordinates, this shows that the tangent vectors $v_X(\lambda_n)$
    to such stretch paths converge to $v_X(\lambda)$.
  \end{proof}  

  Now we specialize to the punctured torus case.  That is, for the
  rest of this section we assume $S = \torus$.  An example of the
  Thurston unit sphere (circle) and its dual are shown in Figure
  \ref{fig:normballs}.  We will show that in this case, the shape of the
  unit sphere determines the hyperbolic structure $X$ up to the action
  of the mapping class group.

  \begin{figure}
  \begin{center}
  \includegraphics[width=0.8\textwidth]{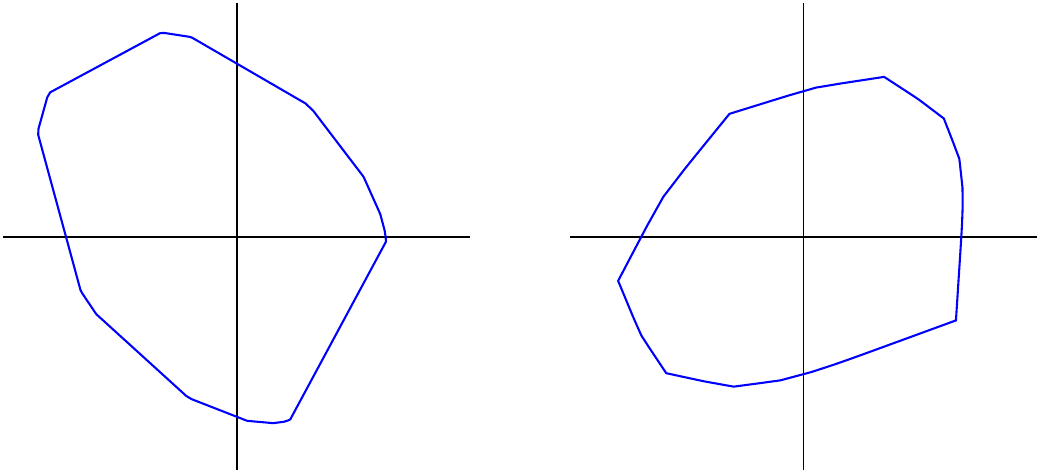}
  \end{center}
  \caption{At left, the unit sphere $T^1_X\T(\torus)$ of the Thurston
    norm on the tangent space at the point $X = 0.35 + 1.8i$ in the
    upper half-plane model of $\T(\torus)$.  At right, the
    unit sphere of the dual norm on the cotangent space.}
  \label{fig:normballs}
  \end{figure}

  In \cite{thurston:MSM}, Thurston studies the facets of the unit ball
  in $T_X \T(S)$, showing in particular that they correspond to simple
  curves on the surface.  We will require a slight extension of the
  result about these facets given by Thurston in Theorem 10.1 of that
  paper.  While a corresponding result for any surface is suggested by
  Thurston's work, here we will use an \emph{ad hoc} argument specific
  to the punctured torus case.

  Let $\RL \subset \CL$ be the set of canonical completions of
  maximal chain-recurrent geodesic laminations on \torus.  Thus for any
  simple curve $\alpha$ on $\torus$ we have $\alpha^+, \alpha^- \in \RL$,
  and any $\lambda \in \RL$ is either of this form or is a completion of a
  measured lamination without closed leaves.
  \begin{theorem}
    \label{thm:max-flat-segments}
    Let $L$ be a support line of the unit ball of $\tnorm{\param}$.  Then either:
    \begin{rmenumerate}
    \item $L \cap T_X^1\T(\torus)$ is a line segment with endpoints
      $v_X(\alpha^+)$ and $v_X(\alpha^-)$ for a simple curve $\alpha$, in
        which case $L = \{ v \st (d_X \log \ell_\alpha)(v) = 1\}$, or 
    \item $L \cap T_X^1\T(\torus)$ is a point, and is equal to
      $\{v_X(\Tilde{\lambda})\}$ for $\Tilde{\lambda}$ the canonical
        completion of a measured lamination $\lambda$ with no closed leaves.
    \end{rmenumerate}
  \end{theorem}

  \begin{proof}
    First, note that Theorem 5.1 implies that
    $v_X(\alpha^+) \neq v_X(\alpha^-)$, so case (i) always yields a
    (nondegenerate) line segment.

    By the duality between the embedding of $\PML(S)$ in
    $T_X^* \T(\torus)$ and the norm ball $T_X^{\leq 1} \T(\torus)$,
    the support lines of the latter are exactly the sets
    $$L_\mu = \{ \mu \st (d_X \log \ell_\mu)(v) = 1 \}$$
    for nonzero $\mu \in \ML(\torus)$.  Thus it suffices to
    characterize the set
    $$L_\mu' := L_\mu \cap T_X^1\T(\torus)$$
    for such $\mu$.  Since $L_\mu$ is a support line of
    $T_X^{\leq 1} \T(\torus)$, we have that $L_\mu'$ is a compact
    convex subset of a line, i.e.~either a point or a segment.  If
    $L_\mu'$ is a segment, then at any interior point $p$ of this
    segment the line $L_\mu$ is the \emph{unique} support line of
    $T_X^{\leq 1}\T(\torus)$ through $p$.

    Suppose $\alpha$ is a simple curve.  Then
    $v_X(\alpha^\pm) \in L_\alpha'$ since
    $\alpha \subset \alpha^\pm$.  By convexity of $L_\alpha'$, the
    closed segment with endpoints $v_X(\alpha^\pm)$ is also a subset
    of $L'$.

    If $L_\alpha'$ properly contained this segment, then at least one
    of $v_X(\alpha^+)$ or $v_X(\alpha^-)$ would be an interior point
    of $L'$, and hence there would be a neighborhood of that point in
    $T_X^1\T(\torus)$ in which $L_\alpha$ is the unique support line.

    To see that this is not the case, choose $\lambda \in \RL$ that
    does not contain $\alpha$ (such as $\lambda = \beta^+$ for $\beta$
    a simple curve that intersects $\alpha$).  Then the sequence of
    Dehn twists $\lambda_n = D_\alpha^n(\lambda)$ converges to
    $\alpha^\pm$ in the Hausdorff topology as $n \to \pm \infty$, and
    the stump $\mu_n$ of $\lambda_n$ has
    $[\mu_n] \neq [\alpha] \in \PML(\torus)$ for all $n$. By Lemma
    \ref{lem:vx-continuity}, the sequence $v_X(\lambda_n)$ converges (again as $n \to \pm \infty$) to
    $v_X(\alpha^\pm)$.  Also, $v_X(\lambda_n)$ lies on the support line
    $L_{\mu_n}$.  Since $\PML(\torus)$ embeds in $T_X^* \T(\torus)$
    (Theorem \ref{thm:pml-embedding}), the lines $L_{\mu_n}$ are all
    distinct from $L_\alpha$.  This shows that $L_\alpha$ is not the
    unique support line in any neighborhood of $v_X(\alpha^\pm)$, and
    that (i) holds in this case.

    Now consider $L_\mu'$ for $\mu$ a measured lamination with no
    closed leaves.  Let $\tilde{\mu} \in \RL$ be the canonical
    completion of $\mu$.  Then $v_X(\tilde{\mu}) \in L_\mu'$.  To
    complete the proof we show $L_\mu' = \{v_X(\tilde{\mu})\}$, so
    that these support lines give case (ii).

    Suppose for contradiction that $L_\mu'$ contains a nontrivial
    segment.  Then $L_\mu$ is the unique support line of
    $T_X^1\T(\torus)$ in the interior of that segment, which has
    $v_X(\tilde{\mu})$ in its closure.

    We can approximate $\tilde{\mu}$ in the Hausdorff topology by
    completions $\alpha_n^+$ of simple curves $\alpha_n$, and can
    furthermore do so with $[\alpha_n] \in \PML(S)$ converging to
    $[\mu] \in \PML(S)$ from either side (recalling that
    $\PML(S) \simeq S^1$, so that removing $[\mu]$ separates a
    small neighborhood in $\PML(S)$ into two sides).  Thus the directions of
    the support lines $L_{\alpha_n^+}$ can be taken to converge to
    that of $L_{\mu}$ from a given side.  As in the previous case,
    Lemma \ref{lem:vx-continuity} shows that $v_X(\alpha_n^+)$
    converges to $v_X(\tilde{\mu})$, and since $v_X(\alpha_n^+)$ lies
    on $L_{\alpha_n}$, this convergence can be taken to be from either
    side of $v_X(\tilde{\mu})$.  Since the support lines
    $L_{\alpha_n}$ are distinct from $L_\mu$, this shows that the
    support line is not unique in any interval whose closure contains
    $v_X(\tilde{\lambda})$, which is the desired contradiction.
  \end{proof}

  Having established that the maximal line segments in $T_X^1\T(\torus)$ are
  exactly those with endpoints $v_X(\alpha^\pm)$ for $\alpha$ a simple
  curve, we now study the geometry of these segments.  From now on we
  refer to these simply as \emph{facets}.  Let $\abs{F(X,\alpha)}$
  denote the length of the facet corresponding to a curve
  $\alpha$ with respect to \tnorm{\param}, i.e.
  \[
  \abs{F(X, \alpha)} = \tnorm{ v_X(\alpha^+) - v_X(\alpha^-)}
  \]

  To estimate this length, we first need the following lemma.

  \begin{lemma}
  \label{Lem:EQ}
  Let $\EQ_\alpha(X,t)$ be the earthquake path along $\alpha$ with
  $\EQ_\alpha(X,0) = X$.  Let $\left . \dot{\EQ_\alpha} = \frac{d\:}{dt}
  \EQ_\alpha(t) \right |_{t=0}$.  Then we have
  \[
  \left \| \dot{\EQ_\alpha}  \right \|_{\mathrm{Th}} \emul_X \ell_\alpha(X).
  \] 
  \end{lemma}

  \begin{proof}
    In fact, this is true for arbitrary measured laminations, for the
    Teichm\"uller space of any surface $S$, and for any norm on $T_X
    \T(S)$.  It is essentially just a rephrasing of \cite[Theorem
    5.2]{thurston:MSM} and the subsequent discussion.
  
    The map $\lambda \mapsto \dot{\EQ_\lambda}$ is a homeomorphism $\ML(S)
    \to T_X \T(S)$ (compare \cite[Theorem 5.1]{gardiner:one-dimensional}),
    and in particular the tangent vector to the earthquake path of a
    nonzero lamination is always nonzero.  The function $\lambda \mapsto
    \dot{\EQ_\lambda} / \ell_\lambda(X)$ is invariant under scaling of
    $\lambda$ and hence gives a well-defined continuous map $\PML(S) \to
    T_X \T \setminus \{0\}$.  By compactness of $\PML(S)$ the function $\|
    \dot{\EQ_\lambda} / \ell_\lambda(X) \|$ is bounded above and below by
    positive constants, which is equivalent to the claim of the Lemma.
  \end{proof}

  \begin{proposition} \label{Prop:Segment}
  For every curve $\alpha$, we have
  \[
  \abs{F(X, \alpha)} \emul_X \ell_\alpha(X)^2 \, e^{-\ell_\alpha(X)}.
  \]
  \end{proposition}

  \begin{proof}
    Let $X_t^+$ and $X_t^-$ be as in \subsecref{in-and-out}. These are
    paths with $X_0^+ = X_0^- = X$ and with tangent vectors
    $v_X\big(\alpha^+\big)$ and $v_X \big(\alpha^- \big)$,
    respectively, at $t=0$.  Note that the length of $\alpha$ is the
    same in $X_t^+$ and $X_t^-$, hence
  \begin{equation}
  \label{eqn:xtplus-eq}
  X_t^+ = \EQ_\alpha( X_t^-, \Delta(t) ),
  \end{equation}
  where as before $\EQ_\alpha$ is the earthquake map along $\alpha$
  and $\Delta$ is the function
  \[
  \Delta(t) = \tau_\alpha(X_t^+) - \tau_\alpha(X_t^-) = 4 e^t \, \log
  \coth \frac{\ell_\alpha(X) }{2} - 4\log \coth \frac{e^t \,
    \ell_\alpha(X)}{2}.
  \]
  Note that $\Delta(0)=0$, and define $\dot{\Delta} =
  \ddtzero{\Delta(t)}$. Differentiating \eqref{eqn:xtplus-eq} at $t=0$
  we find
  \begin{equation}
  \label{eqn:xtplusdot-eq}
  \ddtzero{X_t^+} = (D_1
  \EQ_\alpha)_{(X,0)}( \ddtzero{X_t^-}) + (D_2\EQ_\alpha)_{(X,0)}(\dot{\Delta})
  \end{equation}
  
  Here, $D_1$ and $D_2$ denote the derivatives of $\EQ_\alpha$ with
  respect to its first and second arguments, respectively. Now, as
  observed above, the left hand side of \eqref{eqn:xtplusdot-eq} is
  $v_X\big(\alpha^+\big)$.  Also, since $\EQ_\alpha(Y,0) = Y$ for all
  $Y$, we have that $(D_1 \EQ_\alpha)_{(X,0)}$ is the identity map,
  and the first term on the right hand side of
  \eqref{eqn:xtplusdot-eq} becomes $v_X\big(\alpha^-\big)$.  Recalling
  that $\dot{\EQ}_\alpha = \ddtzero{\EQ_\alpha(X,t)}$, the second term
  on the right hand side of \eqref{eqn:xtplusdot-eq} is equal to
    $\dot{\Delta} \, \dot{\EQ}_\alpha$.

  Thus we have
  \[
  v_X\big(\alpha^+\big) = v_X\big(\alpha^-\big) + \dot{\Delta} \, \dot{\EQ}_\alpha
  \]
  and hence
  \begin{equation}
  \label{eqn:fxalpha-formula}
  |F(X,\alpha)| = \tnorm{ v_X(\alpha^+) - v_X(\alpha^-)} =
  |\dot{\Delta}| \, \tnorm{\dot{\EQ}_\alpha}.
  \end{equation}

  Using the formula for $\Delta(t)$ given above we compute
  \[
  \dot{\Delta} =  4 \log \coth \frac{\ell_\alpha(X) }{2}
  + 4 \ell_\alpha(X) \, \text{csch} \big( \ell_\alpha(X) \big) > 0.
  \]
  For large values of $x$ we have
  \[
  \log \coth (x) \sim 2e^{-2x}
  \qquad\text{and}\qquad
  \text{csch} (x) \sim 2 e^{-x}
  \]
  Hence for large values of $\ell_\alpha(X)$, we have
  \[
  |\dot{\Delta}| = \dot{\Delta}
  \sim
  8 e^{-\ell_\alpha(X)} + 4 \ell_\alpha(X) e^{-\ell_\alpha(X)}
  \emul \ell_\alpha(X) e^{-\ell_\alpha(X)},
  \]
  and by \lemref{EQ},
  \[
  \tnorm{\dot{\EQ_\alpha}} \emul_X \ell_\alpha(X).
  \]
  Substituting these estimates for $\dot{\Delta}$ and
  $\tnorm{\dot{\EQ_\alpha}}$ into \eqref{eqn:fxalpha-formula} gives
  the proposition.
  \end{proof}

  \begin{theorem}
  \label{Thm:Length}
  Let $\alpha$ and $\beta$ be curves with $\I(\alpha, \beta)=1$.  Let
  $\beta_n = D_\alpha^n(\beta)$. Then
  \[
  \lim_{n \to \infty} \frac{\abs{\log \abs{F(X, \beta_n)}}}{n}
  = \ell_\alpha(X).
  \]
  \end{theorem}

  \begin{proof}
  For large values of $n$,
  \[
  \ell_{\beta_n}(X) \eadd n \ell_\alpha(X).
  \]
  The theorem now follows from \propref{Segment}.
  \end{proof}

  Using the results above we can now show that the shape of the unit
  sphere in $T^1_X\T(\torus)$ determines $X$ up to the action of the
  mapping class group.

  \begin{proof} [Proof of Theorem \ref{thm:infinitesimal-rigidity}]
    Within the convex curve $T^1_X \T(\torus)$ let $U_0$ denote an
    open arc disjoint from $F(X,\alpha)$ which has $v_X(\alpha^+)$ as
    one endpoint.  We use ``interval notation'' to refer to open arcs
    within $U_0$, where $(x,y)$ refers to the open arc in $U_0$ with
    endpoints $x,y$.  Thus for example $U_0$ itself is
    $( v_X(\alpha^+), y)$ for some $y$.

    Let $\curves(U_0)$ denote the set of curves $\gamma$ such that
    $F(X,\gamma) \subset U_0$.  Thus $\curves(U_0)$ corresponds to the
    rational points of an interval in $\PML(\torus)$ with $\alpha$ as
    one of its endpoints.  Any sequence of simple closed geodesics in
    this interval converging to $\alpha$ in $\PML(\torus)$ also
    converges in the Hausdorff topology, to $\alpha^+_0$.
    
    Thus for any $\ep > 0$, by choosing $U_0$ small enough we can
    assume that all of the curves $\gamma \in \curves(U_0)$ have
    geodesic representative in $X$ that is contained in an
    $\ep$-neighborhood of the geodesic lamination $\alpha^+_0$.  This
    neighborhood has the structure of a thickened train track $\tau$
    with three branches (as shown in Figure \ref{fig:traintrack}):
    Along $\alpha$ there is a ``thick'' branch and a ``thin'' branch,
    and there is a third branch $\kappa$ which connects one side of
    $\alpha$ to the other.  Such a curve $\gamma$ is therefore
    determined by a pair of coprime nonnegative integers $(p,q)$,
    where $p$ is the weight of the thin branch along $\alpha$ and $q$
    is the weight of $\kappa$.  (By the switch relations, these two
    weights determine the weight of the thick branch to be $p+q$).  We
    call $(p,q)$ the \emph{coordinates} of $\gamma$.  In terms of
    these coordinates, $q$ is the geometric intersection number of the
    curve with $\alpha$.
    
   \begin{figure}
    \begin{center}
      \includegraphics{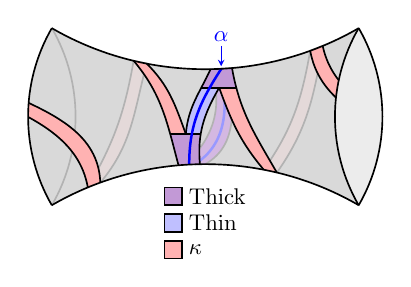}
    \end{center}
    \caption{The train track structure of a neighborhood of $\alpha^+_0$ (shown in a collar about $\alpha$).}
    \label{fig:traintrack}
    \end{figure}   

    Let $\ell_1$ denote the length of the branch $\kappa$, which we
    define to be the minimum length of a path in the rectangle joining
    its short sides, and let $\ell_{\thick}, \ell_{\thin}$ denote the
    lengths of the branches along $\alpha$.  Note that $\ell_1$
    increases without bound as $\ep \to 0$, so we assume from now
    on that $\ell_1$ is much larger than $\ell_\alpha(X)$.  Also,
    since the union of the thick and thin branches gives a small
    neighborhood of $\alpha$ we have
    $\ell_{\thick} + \ell_{\thin} = \ell_\alpha(X) + O(\ep)$.

    The point of this train track representation for curve in
    $\curves(U_0)$ is that it gives a simple estimate for hyperbolic
    length.  Specifically, the simple curve carried by $\tau$ with
    coordinates $(p,q)$ breaks into $p$ arcs in $\kappa$, $q$ arcs in
    thin branch, and $p+q$ arcs in the thick branch.  Each arc has
    length equal to that of its branch, up to an error proportional to
    $\ep$.  Thus the overall length is
    \begin{equation*}
      \begin{split}
        \ell(p,q) &= q \, \ell_{1} + p \, \ell_{\thin} + (p+q) \, \ell_{\thick} + O((p+q) \ep)\\
        &= q \, (\ell_1 + \ell_{\thick}) + p \, (\ell_{\thin} + \ell_{\thick}) + O((p+q) \ep)\\
        &= q \,(\ell_1 + \ell_{\thick}) + p \,\ell_\alpha(X) + O((p+q)\ep),
      \end{split}
    \end{equation*}
    where the factor of $\ep$ accounts for the difference between
    the length of the branch and the length of a segment of $\gamma$
    contained in the branch.
    
    The quantity $p/q$ (the slope of the curve, for a suitable
    homology basis) is an affine coordinate for a neighborhood of
    $\alpha$ in $\PML(\torus) \simeq \RP^1$ in which $\alpha$
    corresponds to $1/0$, thus $\curves(U_0)$ corresponds to curves
    with coordinates satisfying $p/q > m$ for some constant
    $m \in \RR$.  By the length estimate above, after shrinking $U_0$
    so that $\ep$ is much smaller than $\ell_\alpha(X)$ and $\ell_1$,
    we find the minimum length of a curve in $\curves(U_0)$ is
    attained for $q=1$ and the smallest integer $p$ with that $p > m$.
    Denote these length minimizing coordinates by $(p_0,1)$, and the
    corresponding curve by $\gamma_0$.  Note that any other curve in
    $\curves(U_0)$ has length exceeding this minimum by at least a
    fixed positive multiple of $\ell_\alpha(X)$.

    By Proposition~\ref{Prop:Segment}, the length of a facet
    corresponding to a curve whose hyperbolic length is bounded below
    is exponentially decreasing in length of the curve, up to a fixed
    multiplicative error.  (Here, assuming a lower bound on the length
    allows us to ignore the factor $\ell_\gamma(X)^2$ in that
    Proposition, as it is overwhelmed by the exponential decay.)
    Therefore, long curves with a sufficiently large difference in
    hyperbolic length have associated facets whose lengths compare in
    the opposite way.  By taking $\ell_{\alpha}(X)$ to be large enough
    and $U_0$ small enough so that all curves in $\curves(U_0)$ are
    long, the hyperbolic length gap between the minimizer $\gamma_0$
    and any other curve in $\curves(U_0)$ noted above implies that
    $F(X,\gamma_0)$ is the longest facet in $U_0$.

    Now, we can shrink $U_0$ to exclude $F(X,\gamma_0)$, find the new
    longest facet, and iterate this construction.  That is, we apply
    the argument above to the arc
    $(v_X(\alpha^+), v_X(\gamma_0^-))$ and find hyperbolic length
    minimizer and facet length maximizer $\gamma_1$.  Taking
    $v_X(\gamma_0^-)$ as the endpoint means that the coordinates $(p,q)$ of
    curves whose facets lie in this arc now satisfy $p/q > p_0$,
    so arguing exactly as above we find that the coordinates of
    $\gamma_1$ are $(p_0+1,1)$.  Continuing inductively we obtain a
    sequence $\gamma_i$ of curves, each corresponding to the longest
    facet in $(v_X(\alpha^+), v_X(\gamma_{i-1}^-))$ and having
    coordinates $(p_0+i,1)$.  We call this the \emph{sequence of
      longest facets}.

    Recall that the Dehn twist about $\alpha$ acts in these
    coordinates by adding $1$ to the slope of the curve.  Thus, in more
    invariant terms we have shown that the sequence of longest facets
    to one side of $\alpha^+$ corresponds to the sequence of all
    sufficiently large positive powers of the Dehn twist about
    $\alpha$ applied to a curve intersecting $\alpha$ once.  This is
    the sort of collection considered in Theorem \ref{Thm:Length},
    which shows that the hyperbolic length of $\alpha$ is determined
    by the asymptotic behavior of these facet lengths.
    
    An argument very similar to the one above shows that the sequence
    of longest facets in a small neighborhood of $v_X(\alpha^-)$
    corresponds to large \emph{negative} powers of the Dehn twist
    about $\alpha$ applied to a curve intersecting $\alpha$ once, and
    that through the asymptotics of their lengths, the geometry of the
    norm sphere near $v_X(\alpha^-)$ also determines the length of
    $\alpha$.  As before this applies to any simple curve $\alpha$
    that is sufficiently long on $X$.  Collectively, we refer to the
    arguments above as the \emph{longest facet construction}.

  Now for $X,Y \in \T(\torus)$, assume that there is a norm preserving linear
  map
  \[
  L \from T_X \T(\torus) \to T_Y \T(\torus).
  \]
  Since $L$ is linear, it maps the facets in $T_X^1 \T(\torus)$
  bijectively to those in $T_Y^1 \T(\torus)$.  By
  Theorem~\ref{thm:max-flat-segments}, this induces some permutation
  on the simple curves that label the facets: For a simple curve
  $\gamma$ we denote by $\gamma^*$ the simple curve such that
  $L(F(X,\gamma)) = F(Y,\gamma^*)$.

  Choose a simple curve $\alpha$ so that $\ell_X(\alpha)$ and
  $\ell_Y(\alpha^*)$ are large enough so that the longest facet
  construction applies to both of them.
  Then we obtain a sequence of curves
  $\gamma_i = D_\alpha^i \beta$ which satisfy
  $\I(\gamma_i, \alpha) = 1$, and whose facets $F(X,\gamma_i)$
  approach one endpoint of $F(X,\alpha)$ with each being longest in some
  neighborhood of that endpoint.  As $L$ is an isometry, the image
  facets $F(X,\gamma_i^*)$ approach some endpoint of $F(X,\alpha^*)$
  and are locally longest in the same sense.  Thus the curves
  $\gamma_i^*$ are also obtained by applying powers (positive or
  negative) of a Dehn twist about $\alpha^*$ to a fixed curve and they
  satisfy $\I(\gamma_i^*,\alpha^*)=1$.  Since
  $|F(X,\gamma_i)| = |F(Y,\gamma_i^*)|$ we conclude
  $\ell_X(\alpha) = \ell_X(\alpha^*)$.

  Now choose an integer $N$ so that $\ell_X(\gamma_N)$ and
  $\ell_Y(\gamma_N^*)$ are large enough to apply the longest
  facet construction (to $\gamma_N$ and $\gamma_N^*$, respectively).
  Proceeding as in the previous paragraph, we find $\ell_X(\gamma_N) =
  \ell_Y(\gamma_N^*)$.

  At this point we have two pairs of simple curves intersecting once,
  $(\alpha, \gamma_N)$ and $(\alpha^*, \gamma_N^*)$, and the lengths of
  the first pair on $X$ are equal to those of the second pair on $Y$.
  This implies that $X$ and $Y$ are in the same orbit of the extended
  mapping class group: Take a mapping class $\phi$ with
  $\phi(\alpha) = \alpha^*$ and $\phi(\gamma_N) = \gamma_N^*$.  Then
  $\alpha$ has the same length on $X$ and $\phi^{-1}(Y)$, so these
  points differ only in the Fenchel-Nielsen twist parameter (relative to
  pants decomposition $\alpha$).  Since the length of $\gamma_N$ is the
  same as well, either the twist parameters are equal and
  $X = \phi^{-1}(Y)$ or the twist parameters differ by a sign and
  $X = r(\phi^{-1}(Y))$ where $r$ is the orientation-reversing mapping
  class which preserves both $\alpha$ and $\gamma_N$ while reversing
  orientation of $\gamma_N$.
  \end{proof}

  \subsection{Local and global isometries}

  Before proceeding with the proof of Theorem \ref{thm:isometries-local}
  we recall some standard properties of the extended mapping class group
  action on $\T(\torus)$.  (For further discussion, see for example
  \cite[Section~2]{keen:fundamental} \cite[Section~2.2.4]{farb-margalit:primer}.)

  The mapping class group
  $\Mod(\torus) = \mathrm{Homeo}^+(\torus) / \mathrm{Homeo}_0(\torus)$
  of the punctured torus is isomorphic to $\SL(2,\ZZ)$, and identifying
  $\T(\torus)$ with the upper half-plane $\HH$ in the usual way, the
  action of $\Mod(\torus)$ becomes the action of $\SL(2,\ZZ)$ by linear
  fractional transformations.  Similarly, the extended mapping class
  group
  $\Mod^\pm(\torus) = \mathrm{Homeo}(\torus) / \mathrm{Homeo}_0(\torus)$
  can be identified with $\GL(2,\ZZ)$, where an element
  $\left ( \begin{smallmatrix}a&b\\c&d\end{smallmatrix}\right )$ of
  determinant $-1$ acts on $\HH$ by the conjugate-linear map
  $z \mapsto \frac{a \bar{z} + b}{c \bar{z} + d}$.  Neither of these
  groups acts effectively on $\HH$, since in each case the elements
  $\pm I$ act trivially; thus when considering the action on
  $\T(\torus)$ it is convenient to work with the quotients $\PSL(2,\ZZ)$
  and $\PGL(2,\ZZ)$ which act effectively.

  \begin{figure}
  \begin{center}
  \includegraphics[width=0.75\textwidth]{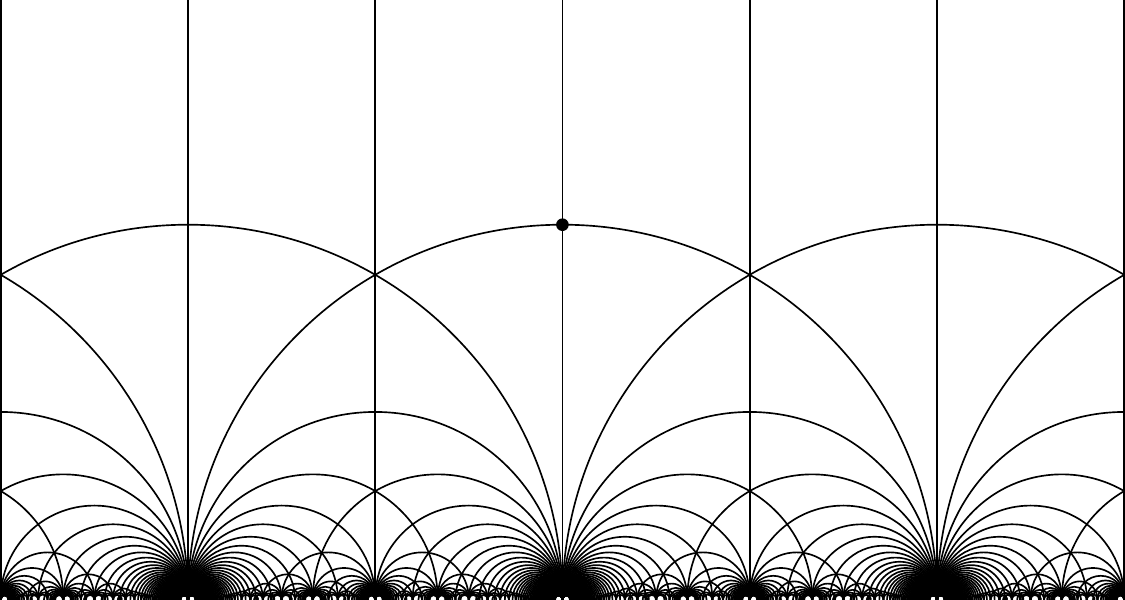}
  \end{center}
  \caption{The standard $(2,3,\infty)$ triangle tiling of the upper half
    plane.  The marked point is the imaginary unit $i$.}
     \label{Fig:modular}
  \end{figure}

  The properly discontinuous action of $\PGL(2,\ZZ)$ on $\HH$ preserves
  the standard $(2,3,\infty)$ triangle tiling of $\HH$ (see
  \figref{modular}), with the cells of each dimension in this tiling
  corresponding to different types of isotropy; specifically, we have:
  \begin{itemize}
  \item A point in the
  interior of a triangle has trivial stabilizer in $\PGL(2,\ZZ)$.
  \item A point in the interior of an edge has stabilizer in
  $\PGL(2,\ZZ)$ isomorphic to $\ZZ/2$ and generated by a \emph{reflection},
  i.e.~an element conjugate to a either $z \mapsto -\bar{z}$ or
  $z \mapsto -\bar{z} + 1$.
  \end{itemize}
  While vertices of the tiling have larger stabilizers, the only
  property of such points we will use is that they form a discrete set.

  \begin{proof}[Proof of Theorem \ref{thm:isometries-local}.]
  Let $U$ be an open connected set in $\T(\torus)$ and let
  $f : (U,\dth) \to (\T(\torus),\dth)$ be an isometric embedding.

  By Theorem \ref{thm:thurston-norm-lipschitz}, the Thurston norm is
  locally Lipschitz (locally $C^{0,1}_{loc}$).  By
  \cite[Theorem~A]{matveev-troyanov} an isometry of such Finsler spaces is
  $C^{1,1}_{loc}$ and its differential is norm-preserving.  Therefore, 
  for each $X \in U$ the differential
  \[
  d_Xf : T_X \T(\torus) \to T_{f(X)} \T(\torus)
  \]
  is an isometry for the Thurston norm and by Theorem
  \ref{thm:infinitesimal-rigidity} there exists $\Phi(X) \in \PGL(2,\ZZ)$
  such that
  \begin{equation}
  \label{eqn:phix}
  f(X) = \Phi(X) \cdot X
  \end{equation}
  This property may not determine $\Phi(X) \in \PGL(2,\ZZ)$
  uniquely, however, choosing one such element for each point of $U$ we
  obtain a map $\Phi : U \to \PGL(2,\ZZ)$. 

  Let $X_0 \in U$ be a point with trivial stabilizer in $\PGL(2,\ZZ)$.
  Using proper discontinuity of the $\PGL(2,\ZZ)$ action, we can select
  neighborhoods $V$ of $X_0$ and $W$ of $f(X_0)$ so that
  \[\{ \phi \in \PGL(2,\ZZ) \st \phi \cdot V \cap W \neq \emptyset \} =
  \{ \Phi(X_0) \} \]
  However, by continuity of $f$ and \eqref{eqn:phix} we find that the
  $\Phi(X)$ is an element of this set for all $X$ near $X_0$.  That is,
  the map $\Phi$ is locally constant at $X_0$.  More generally, this
  shows $\Phi$ is constant on any connected set consisting of points
  with trivial stabilizer.

  Now we consider the behavior of $\Phi$ and $f$ in a small neighborhood
  $V$ of a point $X_1$ with $\ZZ/2$ stabilizer---that is, a point in the
  interior of an edge $e$ of the $(2,3,\infty)$ triangle tiling.  Taking
  $V$ to be a sufficiently small disk, we can assume $V \setminus e$ has
  two components, which we label by $V_\pm$, and that each component
  consists of points with trivial stabilizer (equivalently, $V$ does not
  contain any vertices of the tiling).  By the discussion above $\Phi$
  is constant on $V_+$ and on $V_-$, and we denote the respective values
  by $\phi_+$ and $\phi_-$.  By continuity of $f$, the element
  $\phi_+^{-1} \phi_- \in \PGL(2,\ZZ)$ fixes $e \cap V$ pointwise and is
  therefore either the identity or a reflection.  In the latter case $f$
  would map both sides of $e$ (locally, near $X_1$) to the same side of
  the edge $f(e)$, and hence it would not be an immersion at $X_1$.
  This is a contradiction, for we have seen that the differential of $f$
  is an isomorphism at each point.  We conclude $\phi_+ = \phi_-$, and
  $f$ agrees with this extended mapping class on $V \setminus e$.  By
  continuity of $f$ the same equality extends over the edge $e$.

  Let $U' \subset U$ denote the subset of points with trivial or $\ZZ/2$
  stabilizer.  We have now shown that for each $X \in U'$ there exists a
  neighborhood of $X$ on which $f$ is equal to an element of
  $\PGL(2,\ZZ)$.  An element of $\PGL(2,\ZZ)$ is uniquely determined by
  its action on any open set, so this local representation of $f$ by a
  mapping class is uniquely determined and locally constant.  Thus on
  any connected component of $U'$ we have that $f$ is equal to a mapping
  class.  However $U'$ is connected, since $U$ is connected and open and
  the set of points in $\T(\torus)$ with larger stabilizer (i.e.~the
  vertex set of the tiling) is discrete.

  We have therefore shown $f = \phi$ on $U'$, for some $\phi \in
  \PGL(2,\ZZ)$.  Finally, both $f$ and $\phi$ are continuous, and $U'$ is
  dense in $U$, equality extends to $U$, as required.
  \end{proof}

  \vspace{1.1em}

  \noindent Department of Mathematics, Statistics, and Computer Science\\
  University of Illinois at Chicago\\
  Chicago, IL, USA\\
  \texttt{david@dumas.io}

  \medskip

  \noindent Department of Mathematics\\
  University of Rennes\\
  Rennes, FR\\
  \texttt{anna.lenzhen@univ-rennes1.fr}

  \medskip

  \noindent Department of Mathematics\\
  University of Toronto\\
  Toronto, CA\\
  \texttt{rafi@math.toronto.edu}

  \medskip

  \noindent Department of Mathematics\\
  University of Oklahoma\\
  Norman, OK, USA\\
  \texttt{jing@ou.edu}

\end{document}